\documentclass{daj}

\dajAUTHORdetails{%
  title = {Notes on nilspaces -- algebraic aspects}, 
  author = {Pablo Candela},
  plaintextauthor = {Pablo Candela},
  runningtitle = {Notes on nilspaces: algebraic aspects}, 
  runningauthor = {P. Candela},
}   

\dajEDITORdetails{%
   year={2017},
   number={15},
   received={31 May 2016},   
   revised={25 May 2017},    
   published={7 September 2017},  
   doi={10.19086/da.2105},       
}   

\usepackage{amsmath, amsfonts, amssymb, amsthm, mathtools}
\usepackage{enumerate}
\usepackage{enumitem}
\usepackage{fancyhdr}
\usepackage[utf8]{inputenc}
\usepackage{tikz,tikz-cd}
\usepackage[hang,flushmargin]{footmisc}
\usepackage[compact]{titlesec}
\usepackage{graphicx}
\usepackage{xr}

\titleformat{\chapter}[display]
{\normalfont\huge\bfseries}{\chaptertitlename\ \thechapter}{20pt}{\Huge}

\newtheorem{theorem}{Theorem}[section]
\newtheorem{lemma}[theorem]{Lemma}

\newtheorem{corollary}[theorem]{Corollary}
\newtheorem{proposition}[theorem]{Proposition}

\theoremstyle{definition}
\newtheorem{defn}[theorem]{Definition}
\newtheorem{remark}[theorem]{Remark}
\newtheorem{example}[theorem]{Example}

\newcommand{\id}{\mathrm{id}}
\newcommand{\Zmod}[1]{\Z/#1\Z}
\newcommand*{\sbr}[1]{\scalebox{0.8}{$(#1)$}}

\providecommand{\arr}[1]{\langle#1\rangle}

\def\Z{\mathbb{Z}}
\def\R{\mathbb{R}}
\def\T{\mathbb{T}}

\def\N{\mathbb{N}}

\DeclareMathOperator{\aut}{Aut}
\DeclareMathOperator{\ab}{Z}
\DeclareMathOperator{\bnd}{B}
\DeclareMathOperator{\tran}{\Theta}

\DeclareMathOperator{\ia}{\wp}
\DeclareMathOperator{\poly}{poly}

\DeclareMathOperator{\supp}{Supp}

\DeclareMathOperator{\codim}{codim}
\DeclareMathOperator{\q}{c}
\DeclareMathOperator{\ns}{X}
\DeclareMathOperator{\nss}{Y}
\DeclareMathOperator{\co}{\circ\hspace{-0.02 cm}}
\DeclareMathOperator{\cu}{C}
\DeclareMathOperator{\cs}{s}
\DeclareMathOperator{\trem}{\Psi}

\DeclareMathOperator{\cD}{\mathcal{D}}

\DeclareMathOperator{\cF}{\mathcal{F}}

\DeclareMathOperator{\cT}{\mathcal{T}}

\DeclareMathOperator{\tIm}{Im}
\begin{document}

\begin{frontmatter}[classification=]

\title{\Huge{Notes on nilspaces: algebraic aspects}
} 
\author[pc]{Pablo Candela}
\end{frontmatter}


\tableofcontents

\titlespacing*{\chapter}{0pt}{1.5cm}{40pt}

\chapter{Introduction}\label{chap:intro}\vspace{-.5cm}
These notes constitute the first part of a detailed exposition of the theory of nilspaces developed by Camarena and Szegedy in the paper \cite{CamSzeg}. We treat what can be called the algebraic part of the theory, in which nilspaces are studied without any topological assumption. We elaborate on some of the central concepts and arguments from the first two chapters of \cite{CamSzeg}, present more detailed proofs, and illustrate some of the key notions with concrete examples.

The main objects studied in \cite{CamSzeg} are \emph{nilspaces}. These objects are defined in terms of some axioms, which rely on the basic notions of  discrete cubes and cube morphisms. We shall begin in the next section by recalling these notions and their elementary properties. \\
\indent As stated in \cite{CamSzeg}, important examples of nilspaces are provided by filtered groups and filtered nilmanifolds. In Chapter \ref{chap:Centrexs}, we elaborate on this by treating the theory of cubes on filtered groups, which was pioneered by Host and Kra \cite{HK,HKparas} and developed in several other works including \cite{GTOrb,GTZ,HM}. We treat this theory from the viewpoint of general nilspaces. This provides natural motivation for various important objects and results related to this theory, notably polynomial maps and Leibman's theorem, and it also provides illustrations of several tools and ideas used by Camarena and Szegedy.\\
\indent One of the central results in the algebraic part of \cite{CamSzeg} is the description of a general nilspace as an iterated \emph{abelian bundle}.  In Chapter \ref{chap:chargenils}, after having gathered some basic tools in Section \ref{sec:BasNot}, we present a detailed proof of this result in Section \ref{sec:bundledecomp}, following the arguments of Camarena and Szegedy. The final section of the chapter concerns additional algebraic tools introduced in \cite{CamSzeg}. These tools are of inherent interest but are also important for the topological part of the theory of nilspaces.\\
\indent We treat the topological part of the theory separately, in \cite{Cand:Notes2}. That part studies nilspaces equipped with a compact topology  that is compatible with the cube structure, called \emph{compact nilspaces}. This study is motivated principally by the analysis of \emph{uniformity norms} in arithmetic combinatorics, and of the analogous \emph{uniformity seminorms} in ergodic theory, leading in particular to the theory of \emph{higher order Fourier analysis}; see \cite{GHFA} for an introduction to this theory from the viewpoint of arithmetic combinatorics. An alternative treatment of compact nilspaces is given by Gutman, Manners, and Varj\'u in \cite{GMV1,GMV2,GMV3}.

\vspace{0.6cm}

\noindent \textbf{Acknowledgements.} I am very grateful to Bal\'azs Szegedy for discussions that were crucial for my understanding of \cite{CamSzeg}, and to  anonymous referees for advice that helped to improve this paper. I also thank Yonatan Gutman, Frederick Manners and P\'eter Varj\'u for informing me of their work prior to publication, and for useful comments. The present work was supported by the ERC Consolidator Grant No. 617747.

\section{Discrete cubes}\label{sec:discubes}
Recall that an \emph{affine homomorphism} from an abelian group $\ab_1$ to an abelian group $\ab_2$ is a map of the form $x \mapsto \phi(x)+t$, where $\phi:\ab_1\to\ab_2$ is a homomorphism and $t$ is some fixed element of $ \ab_2$.
\begin{defn}\label{def:D-cubes}
Let $n$ be a non-negative integer. The \emph{discrete $n$-cube} is the set $\{0,1\}^n$. We denote the elements $(0,\ldots,0), (1,\ldots,1)$ of  $\{0,1\}^n$ by $0^n,1^n$ respectively. A  \emph{discrete-cube morphism} is a map $\phi:\{0,1\}^m\to \{0,1\}^n$ that is the restriction of an affine homomorphism $\Z^m\to \Z^n$. An \emph{automorphism} of the discrete $n$-cube is a bijective morphism $\{0,1\}^n\to \{0,1\}^n$. These automorphisms form a group, denoted by $\aut(\{0,1\}^n)$.
\end{defn}
\noindent We recall that $\aut(\{0,1\}^n)$ is generated by the group $S_n$ of permutations of $[n]=\{1,2,\ldots,n\}$ together with the reflections $v\sbr{i}\mapsto 1-v\sbr{i}$, $i\in [n]$, more precisely we have $\aut(\{0,1\}^n)\cong S_n \ltimes (\Zmod{2})^n$.\\
\indent For $v\in\{0,1\}^n$, we shall write $\supp(v)$ for the \emph{support} of $v$, i.e.
\[
\supp(v)=\{i\in [n]: v\sbr{i}=1\},
\]
and we shall denote by $|v|$ the cardinality of $\supp(v)$.\\
\indent
Morphisms of discrete cubes can be described more explicitly in several ways. Below we give a description that will be quite useful later on. Let us first denote the four maps $\{0,1\}\to \{0,1\}$ as follows: we denote by $\mathbf{0}$ and $\mathbf{1}$ the maps with constant value 0 and 1 respectively, we denote the identity map by $\tau_0$, and we denote by $\tau_1$ the reflection map (the bijection sending 0 to 1).
\begin{lemma}\label{lem:cube-morph-char}
A map $\phi:\{0,1\}^m\to \{0,1\}^n$ is a discrete-cube morphism if and only if it is of the form $\phi\big(v\sbr{1},\ldots, v\sbr{m}\big) = \big(w\sbr{1},\ldots , w\sbr{n}\big)$ where, for each $j\in [n]$, either $w\sbr{j}$ is a constant function of $v$ or there is a unique $i= i(j)\in [m]$ and $k=k(j)\in \{0,1\}$ such that $w\sbr{j} = \tau_k (v\sbr{i})$.
\end{lemma}
\begin{example}\label{ex:morphs}
To illustrate the lemma, we can consider a morphism $\phi:\{0,1\}^2\to \{0,1\}^3$. We could have for instance $\phi(v)=(v\sbr{1}, 1-v\sbr{1}, v\sbr{2}\big)$, where no coordinate of $\phi(v)$ is a constant function of $v$. In this case, embedding $\{0,1\}^3$ in $\R^3$ as the set of vertices of the unit cube, the image of $\phi$  consists of the vertices of a square going diagonally through the cube, namely $\{(0,1,0),(1,0,0),(0,1,1),(1,0,1)\}$. We could also have that some coordinate of $\phi(v)$ is constant, for instance $\phi(v)=(v\sbr{1},0, 1-v\sbr{2}\big)$. In this case, in the embedding just mentioned, the image of $\phi$ consists of the vertices of a face of the unit cube, namely $\{(0,0,1),(1,0,1),(0,0,0),(1,0,0)\}$.
\end{example}
\begin{proof}[Proof of Lemma \ref{lem:cube-morph-char}]
Note that, firstly, from the definition of a morphism  we can deduce that $\phi(v)=Mv+\phi(0^m)$, for some matrix $M\in \Z^{n\times m}$. Then, since $\phi$ takes values in $\{0,1\}^n$, for each $j\in [n]$ the $j$-th row of $M$ has at most one non-zero entry, with index $i=i(j)$. This entry must then equal 1 or $-1$, which then determines the $i$-th entry of $\phi(0^m)$ to be 0 or 1 respectively, and the result follows.
\end{proof}
\noindent We shall now define an important family of subsets of a discrete cube, namely its \emph{faces}. For this we use the following notation. Given a morphism $\phi:\{0,1\}^m\to \{0,1\}^n$ and $i\in [m]$, we denote by $J_\phi(i)$ the set of indices of the coordinates of $\phi(v)$ that are not constant functions of $v\sbr{i}$ (in other words, the coordinates of $\phi(v)$ which vary as $v\sbr{i}$ varies), that is
\begin{equation}\label{eq:Ji-set}
J_\phi(i):=\{ j\in [n]: \phi(v)\sbr{j}= \tau_0(v\sbr{i})\textrm{ or }\tau_1(v\sbr{i})\}.
\end{equation}
\noindent In the example $\phi:\{0,1\}^2\to \{0,1\}^3$, $v\mapsto (v\sbr{1}, 1-v\sbr{1}, v\sbr{2}\big)$ seen above, we would have $J_\phi(1)=\{1,2\}$ and $J_\phi(2)=\{3\}$.

Note that a morphism $\phi:\{0,1\}^m\to \{0,1\}^n$ is injective if and only if $J_\phi(i)$ is non-empty for every $i\in [m]$. We define
\begin{equation}\label{eq:J-set}
J_\phi:=\bigsqcup_{i\in [m]} J_\phi(i).
\end{equation}
This is the set of coordinates of $\phi(v)$ that are not constant functions of $v$. We denote the set of subsets of size $k$ of a set $S$ by $\binom{S}{k}$.
\begin{defn}[Faces and face maps]\label{def:faces}
Let $m\leq n$ be non-negative integers. An $m$\emph{-face} of $\{0,1\}^n$ (short for $m$-dimensional face) is a set $F\subset \{0,1\}^n$ specified by fixing $n-m$ coordinates of $v\in \{0,1\}^n$, that is 
$F=\{v\in \{0,1\}^n: v\sbr{i}=x\sbr{i}\;\;\forall\, i\in I\}$, for some fixed set $I\in \binom{[n]}{n-m}$ and $x\in \{0,1\}^I$. We call a morphism $\phi:\{0,1\}^m\to \{0,1\}^n$ an $m$-\emph{face map} if $\phi$ is injective and the image $\tIm(\phi)$ is an $m$-face. Equivalently, $\phi$ is an $m$-face map if $|J_\phi(i)|= 1$ for every $i\in [m]$.
\end{defn}
\noindent For instance, the map $\phi:\{0,1\}^2\to \{0,1\}^3$, $v\mapsto (v\sbr{1}, 0, 1-v\sbr{2}\big)$ from Example \ref{ex:morphs} is a 2-face map and $\tIm(\phi)$ is the 2-face $\{w\in \{0,1\}^3: w\sbr{2}=0\}$. The map $\phi:\{0,1\}^2\to \{0,1\}^3$, $v\mapsto (v\sbr{1}, 1-v\sbr{1}, v\sbr{2}\big)$ is an injective morphism, but is not a face map.

We shall often specify an $m$-face $F\subset \{0,1\}^n$ by the data $(I,x)$. We shall then refer to the cardinality of $I$ as the \emph{codimension} of $F$, denoted by $\codim(F)$. For later use, let us also  reserve the special notation $\phi_F$ for the $m$-face map that has image equal to $F$ in the ``simplest way". More precisely, $\phi_F$ will always denote the $m$-face map such that for all $j < j'$ in $J_{\phi_F}$ we have $i(j) < i(j')$ in $[m]$, and $\phi_F(v)\sbr{j}= v\sbr{i(j)}$ for all $j\in J_{\phi_F}$. For instance if $F=\{w\in \{0,1\}^3: w\sbr{2}=0\}$ then $\phi_F:\{0,1\}^2\to \{0,1\}^3$, $v\mapsto (v\sbr{1}, 0, v\sbr{2}\big)$. Note that if $m=n$ then $\phi_F$ is just the identity in $\aut(\{0,1\}^n)$.
\begin{defn}[Face restriction]\label{def:face-res}
Let $F\subset \{0,1\}^n$ be an $m$-face and let $g$ be a map defined on $\{0,1\}^n$. The \emph{face restriction} of $g$ to $F$ is the map $g \co \phi_F$ on $\{0,1\}^m$.
\end{defn}

\section{Definition of a nilspace}

\begin{defn}\label{def:ns}
A \emph{nilspace} is a set $\ns$ together with a collection of sets $\cu^n(\ns)\subset \ns^{\{0,1\}^n}$, for each non-negative integer $n$, satisfying the following axioms:
\begin{enumerate}[leftmargin=0.8cm]
\item (Composition)\quad For every morphism $\phi:\{0,1\}^m\to \{0,1\}^n$ and every $\q \in \cu^n(\ns)$, we have\\ $\q \co \phi\in \cu^m(\ns)$.\\\vspace{-0.7cm}

\item (Ergodicity)\quad $\cu^1(\ns)=\ns^{\{0,1\}}$.\\
\vspace{-0.7cm}

\item (Corner completion)\quad Let $\q' :\{0,1\}^n\setminus \{1^n\}\to \ns$ be such that every restriction of $\q'$ to an $(n-1)$-face containing $0^n$ is in $\cu^{n-1}(\ns)$. Then there exists $ \q \in \cu^n(\ns)$ such that $ \q (v)= \q'(v)$ for all $v\neq 1^n$.
\end{enumerate}
\end{defn}
\noindent The elements of $\cu^n(\ns)$ are referred to as the $n$\emph{-cubes} (or $n$-dimensional cubes) on $\ns$. A map $\q'$ satisfying the premise of the third axiom is called an $n$\emph{-corner}. We shall call an $n$-cube $\q$ satisfying the third axiom for $\q'$ a \emph{completion} of $\q'$.

\begin{defn}
A $k$\emph{-step nilspace} (or $k$-\emph{nilspace}) is a nilspace such that every $(k+1)$-corner has a unique completion.
\end{defn}
\noindent As a first manipulation of the axioms, one may check that the set underlying a 0-step nilspace, if non-empty, must be a singleton.

Occasionally we will have to consider spaces $\ns$ for which one of the last two axioms above may fail. We shall have to deal with such spaces explicitly only from Chapter \ref{chap:chargenils} onwards, so we postpone their discussion until then, recording only the following definition for now.
\begin{defn}\label{def:cubespace}
A \emph{cubespace} is a set $\ns$ together with a collection of sets $\cu^n(\ns)\subset \ns^{\{0,1\}^n}$, for each $n\geq 0$,  satisfying the composition axiom from Definition \ref{def:ns} and such that $\cu^0(\ns)=\ns$. A cubespace $\ns$ is said to be \emph{$k$-fold ergodic} if $\cu^k(\ns)=\ns^{\{0,1\}^k}$.
\end{defn}
\noindent The nilspace axioms can be seen as variants of the axioms defining the notion of \emph{parallelepiped structures}, notion introduced by Host and Kra in \cite{HKparas}. For instance, the corner completion axiom appears in an equivalent form in \cite[Definition 4]{HKparas}. There is in fact a sense in which the concept of parallelepiped structures is equivalent to that of a nilspace. We shall detail this in Subsection \ref{subsec:altnsdef}.\\
\indent The term `nilspace' may seem deceptive at first, in that it is not  clear from the axioms that a nilspace should be related to anything involving nilpotency. It is a  nontrivial fact that there is indeed such a relation. A first explicit description of this phenomenon appeared in the setting of ergodic theory, in the work of Host and Kra \cite{HK}. The notion of cubes on a filtered group, treated in the next chapter, originates in that work. In the later paper \cite{HKparas}, Host and Kra initiated the program of understanding the above-mentioned relation in a more conceptual and general way, starting from the axioms that define parallelepiped structures. This greater generality consisted especially in that there was no measure-preserving system underlying  the definition of these structures, unlike in \cite{HK}. The results included the determination of a nilpotent group naturally associated with such a structure (see \cite[Section 3.8]{HKparas}). The work of Camarena and Szegedy in \cite{CamSzeg} can be seen as a continuation of this program (see the end of Subsection \ref{subsec:altnsdef}).

In the next chapter we begin to study some natural examples of nilspaces, starting with the cube structures on filtered groups introduced by Host and Kra. In fact, the results in later chapters will indicate that these filtered groups, together with their quotients (including nilmanifolds), constitute the main examples of nilspaces.

\chapter{Main examples of nilspaces}\label{chap:Centrexs}

In this chapter we begin our study of nilspaces by examining some natural examples.\\
\indent We start in Section \ref{sec:abcubes} with a basic case, namely that of cubes on abelian groups. We shall provide several characterizations of these cubes and show that an abelian group together with these cubes satisfies the nilspace axioms.\\
\indent In Section \ref{sec:filgps} we shall then generalize this cube structure to any filtered group (not necessarily commutative) and show that this also yields a nilspace structure on such groups.\\
\indent The nilspaces covered in Section \ref{sec:filgps} have the particular feature that the cubes of a fixed dimension form a group under pointwise multiplication. Taking quotients of these nilspaces in certain ways produces new nilspaces that do not have this group property; this is illustrated in Section \ref{sec:filnilm}.\\ \vspace{0.2cm}

\section{Standard cubes on abelian groups}\label{sec:abcubes}

Let $(G,+)$ be an abelian group. An $n$-cube on $G$ is a map $\q:\{0,1\}^n\to G$ of the form\footnote{Strictly speaking, the term `$n$-dimensional parallelepiped' may be more accurate here, but we shall use the shorter `$n$-cube' for convenience.}
\[
\q(v)=x+v\sbr{1}\,h_1+\cdots+v\sbr{n}\,h_n,
\]
for some fixed elements $x,h_1,\ldots,h_n\in G$. This definition of cubes on $G$ is quite familiar in arithmetic combinatorics (it is involved in the definition of the Gowers uniformity norms, for instance). In Proposition \ref{prop:1-degree-cubes} below we shall record other equivalent ways to view these cubes. To this end we use the following function.

\begin{defn}\label{def:abelian-sign}
Let $G$ be an abelian group. We define the function $\sigma_2:G^{\{0,1\}^2}\to G$ as follows: for $g:\{0,1\}^2 \to G$, \; $\sigma_2(g):= g\sbr{00}-g\sbr{10}+g\sbr{11}-g\sbr{01}$.
\end{defn}
\noindent Note that $\q$ is a $2$-cube on $G$ if and only if $\sigma_2 (\q)=0$. Note also that if $\q$ is a $2$-cube on $G$ then from the expression $\q(v)=x+v\sbr{1}h_1+v\sbr{2}h_2$ it clearly follows that $\q$ can be extended to an affine map $\Z^2\to G$, namely the map $(n_1,n_2)\mapsto x+n_1h_1+n_2h_2$. The converse holds also. 

These remarks generalize in a straightforward way to $n$-cubes for $n>2$, as follows.\\

\begin{proposition}\label{prop:1-degree-cubes}
Let $\q:\{0,1\}^n\to G$ be a map. The following properties are equivalent.
\begin{enumerate}
\item There exist $x,h_1,h_2,\ldots,h_n\in G$ such that for every $v\in\{0,1\}^n$ we have
\begin{equation}\label{eq:standard-cube}
\q(v)=x+v\cdot h:=x+v\sbr{1}h_1+\cdots +v\sbr{n}h_n.
\end{equation}
\item The map $\q$ extends to an affine homomorphism $\Z^n\to G$.
\item Every $2$-face map $\phi:\{0,1\}^2\to \{0,1\}^n$  satisfies
\begin{equation}\label{eq:st-cube-2-face}
\sigma_2(\q\co\phi)=0.
\end{equation}
\end{enumerate}
We call the map $\q$ a \textup{(}standard\textup{)} $n$\emph{-cube} on $G$ if it satisfies any of these properties.
\end{proposition} 
\begin{proof}
The equivalence of ($i$) and ($ii$) is clear. It is also clear that ($ii$) implies ($iii$), since every 2-face map $\phi$ is by definition the restriction to $\{0,1\}^2$ of an affine homomorphism $\Z^2\to \Z^n$, and so $\q\co\phi$ is the restriction of an affine homomorphism $\Z^2\to G$, i.e. a 2-cube. 
If $\q$ satisfies ($iii$), then let $x:= \q( 0^n)$ and for each $i\in [n]$ let $h_i:= \q(e_i)-x$, where $e_i$ denotes the $i$-th vector in the standard basis of $\R^n$. We claim that $\q$ satisfies \eqref{eq:standard-cube} with this choice of elements $x, h_1,\ldots, h_n$. This can be shown by induction on $|v|$. For $|v|=0,1$ the property holds by definition of $x$ and $h_i$, so let $|v|>1$ and suppose that $\q$ satisfies ($ii$) for all $v'$ with $|v'| < |v|$. Now  there clearly is a 2-face-map $\phi:\{0,1\}^2\to\{0,1\}^n$ such that $v=\phi(1^2)$  and $|\phi(w)|< |v| $ for all $w\neq 1^2$, so we have by induction $\q(\phi (w))= x+\phi (w) \cdot h$ for each $w\in \{0,1\}^2\setminus \{1^2\}$. By \eqref{eq:st-cube-2-face}, we then have $\q(\phi (11))= \q(\phi (10))+\q( \phi (01))-\q( \phi (00)) = x + (\phi (10)+\phi (01)-\phi (00) )\cdot  h = x+  v\cdot h$ as required.
\end{proof}
\noindent The cubes characterized by this proposition yield a basic example of a nilspace.
\begin{proposition}\label{prop:1-step-ns}
An abelian group together with the collection of all standard $n$-cubes on $G$ \textup{(}for each non-negative integer $n$\textup{)} is a 1-step nilspace.
\end{proposition}
\begin{proof}
The ergodicity axiom holds (from \eqref{eq:standard-cube} say).  Composition is also clear (by property $(ii)$ from Proposition \ref{prop:1-degree-cubes}, for instance). The unique corner completion for $n\geq 2$ follows from the fact that, by \eqref{eq:standard-cube}, a standard $n$-cube is determined by its values at elements $v\in \{0,1\}^n$ with $|v|\leq 1$. Thus, given  an $n$-corner $\q'$ on $G$, letting $x=\q'(0^n)$ and $h_i=\q'(e_i)-x$ for each $i\in [n]$, the unique $n$-cube $\q$ completing $\q'$ is given by the formula $\q(v)=x+v\sbr{1}h_1+\cdots+v\sbr{n}h_n$. 
\end{proof}

\begin{remark}\label{rem:abfacegp}
For an element $g\in G$ and a set $F\subset \{0,1\}^n$, let $g^F$ denote the element of $G^{\{0,1\}^n}$ defined by $g^F(v)=g$ if $v\in F$ and $0_G$ otherwise. Then it follows from \eqref{eq:standard-cube} that the standard $n$-cubes form the subgroup of $G^{\{0,1\}^n}$ generated by the `constant elements' $x^{\{0,1\}^n}$ and the elements $h_i^{F_i}$, $i\in [n]$, where $F_i$ is the face $\{v\in \{0,1\}^n: v\sbr{i}=1\}$ of codimension 1. (In fact each $n$-cube is generated like this in a unique way.) In the next subsection we shall generalize the standard cubes using this group-theoretic viewpoint.
\end{remark}

\begin{remark}\label{rem:alternprop3}
In Proposition \ref{prop:1-degree-cubes} one can replace `Every 2-face map' with `Every morphism' without affecting the equivalence. We omit the details here, as a more general remark of this kind will be made later (see Remark \ref{rem:alternprop3gen}).
\end{remark}

\section{Filtered groups of degree $k$ as $k$-step nilspaces}\label{sec:filgps}

\noindent In this section the main objectives are the following. First we shall extend the notion of standard cubes to any group and show that these cubes form a nilspace structure; see Proposition \ref{prop:FilGp=nilspace}. We shall then seek alternative characterizations of these cubes to obtain a generalization of Proposition \ref{prop:1-degree-cubes}. This will be obtained eventually as Proposition \ref{prop:G-cubes}, and the task will involve in particular a treatment of polynomial maps between filtered groups, in Subsection \ref{subsec:polys}. 

\subsection{Cubes on a general filtered group}

Let us begin by recalling the definition of a filtered group. Given a group $G$, the commutator $[g,h]$ of two elements $g,h\in G$ is defined by $[g,h] = g^{-1}h^{-1} g h$. Thus $gh=hg [g,h]$. Given two subgroups $H_1,H_2$ of $G$, we denote by $[H_1,H_2]$ the subgroup of $G$ generated by all the commutators $[h_1,h_2]$, $h_i \in H_i$.
\begin{defn}
A \emph{filtration} on a group $G$ is a sequence $G_\bullet= (G_i)_{i=0}^\infty$ of subgroups of $G$ satisfying $G \geq G_0 \geq G_1 \geq \cdots$ and such that $[G_i,G_j] \subset G_{i+j}$ for all $i, j \geq 0$.  We refer to $(G,G_\bullet)$ as a \emph{filtered group}. If some term in the sequence $G_\bullet$ is the trivial subgroup $\{\id_G\}$, then the \emph{degree} of the filtration, denoted by $\deg(G_\bullet)$, is the smallest integer $k$ such that $G_{k+1}=\{\id_G\}$. The filtered group $(G,G_\bullet)$ is then said to be \emph{of degree} $k$.
\end{defn}
\noindent We shall usually assume that $G_0=G_1=G$, except in some clearly indicated places. Note that $G_i$ is a normal subgroup of $G$ for each $i\geq 0$. We shall denote by $\pi_i$ the quotient homomorphism $G\to G/G_i$.\\
\indent One can always take as a filtration the lower central series on $G$, and for many purposes this filtration is the natural one to use. However, other filtrations do arise naturally (see for instance Subsection \ref{sect:deg-k-str}), and in some settings, typically quantitative ones, it is convenient to work with a general filtration \cite{GTOrb}. We shall do so here, thus for the remainder of this section we suppose that an arbitrary filtration $G_\bullet$ has been fixed for the given group $G$.\\

\noindent The general cube structure on a filtered group that we shall study  originated in the work of Host and Kra \cite{HK}. The definition relies on the following notion, which builds up on Remark \ref{rem:abfacegp}.

\begin{defn}[Face groups]\label{def:face-gps}
Given a face $F$ in $\{0,1\}^n$ and an element $g \in G$, we write $g^F$ for the element of $G^{\{0,1\}^n}$ defined by $g^F(v)=g$ if $v\in F$ and $\id_G$ otherwise. We define
\[
G_{(F)}=\{g^F: g\in G_{\codim(F)}\} \leq G^{\{0,1\}^n}.
\]
We refer to these subgroups $G_{(F)}$ as the \emph{face groups} in $G^{\{0,1\}^n}$.
\end{defn}

\begin{defn}[Cubes on a filtered group]\label{def:G-cubes}
Let $(G,G_\bullet)$ be a filtered group. The group of $n$-cubes on $(G,G_\bullet)$, denoted by $\cu^n(G_\bullet)$, is the subgroup of $G^{\{0,1\}^n}$ generated by the face groups.
\end{defn}

\noindent Let us make a remark on the notation $\cu^n(G_\bullet)$. On one hand this involves the same `$\cu^n$' as in Definition \ref{def:ns}; this will be  justified by the main result of this subsection, which shows that $G$ together with $(\cu^n(G_\bullet))_{n\geq 0}$ is a $k$-step nilspace with $k=\deg(G_\bullet)$. On the other hand, as we shall see, these cubes do depend significantly on the filtration, which is why the notation specifies the filtration rather than just the group. For instance, if $G$ is abelian and $G_\bullet$ is the lower central series $G_0=G_1=G$, $G_2=\{\id_G\}$, then $\cu^n(G_\bullet)$ consists of the standard $n$-cubes seen in the previous section, while if $G_\bullet$ is the degree-$k$ filtration $G_0=G_1=\cdots=G_k$, $G_{k+1}=\{\id_G\}$ then we obtain a rather different but still very natural nilspace structure (see Subsection \ref{sect:deg-k-str}).

Note that it is immediate from the definition that $\cu^n(G_\bullet)$ is globally invariant under composition with an automorphism of $\{0,1\}^n$.

We now give a more explicit description of these cubes, in terms of certain `coefficients', that generalizes the expression $\q(v)=x+v\sbr{1}h_1+\cdots +v\sbr{n} h_n$ for standard abelian cubes. To that end we use the following notion.

\begin{defn}[Upper faces]\label{def:up-faces}
An \emph{upper face} of $\{0,1\}^n$ is a set of the form
\[
F(v)=\{u\in \{0,1\}^n: \supp(u)\supseteq \supp(v)\}
\]
for some $v\in \{0,1\}^n$. Note that $\codim(F)=|\supp(v)|$.
We order the upper faces of $\{0,1\}^n$ by writing $F(u)<F(v)$ if $u<v$ in the colex order, thus
\[
F_0=\{0,1\}^n=F(0^n)<F((1,0,\ldots,0))<\cdots<F(1^n)=\{1^n\}=F_{2^n-1}.
\]
\end{defn}
\noindent Note that given two upper faces $F_i<F_j$, their intersection is an upper face $F_k\geq F_j$, and $\codim(F_k)\leq \codim(F_i)+\codim(F_j)$.

As observed in Remark \ref{rem:abfacegp}, any standard abelian cube is a unique sum of upper-face group elements for faces of codimension at most 1.  This can be generalized as follows.
\begin{lemma}[Unique factorization of cubes]\label{lem:cubefactn}
An element $\q\in G^{\{0,1\}^n}$ lies in $\cu^n(G_\bullet)$ if and only if it has a unique factorization of the form
\begin{equation}\label{eq:cubefactn}
\q= g_0^{F_0}g_1^{F_1}\cdots g_{2^n-1}^{F_{2^n-1}},
\end{equation}
with $g_i\in G_{\codim(F_i)}$ $\forall\, i$.
\end{lemma}
\noindent Recall that for each face $F\subset \{0,1\}^n$ there is unique data $I\subset [n],\, x\in \{0,1\}^I$ such that 
\[
F=\{ v\in \{0,1\}^n : v\sbr{i}=x\sbr{i}\;\;\forall\,i\in I\}.
\]
Upper faces are precisely those such that the corresponding element $x$ has all entries equal to 1.
\begin{proof}
First we claim that every element $\q\in \cu^n(G_\bullet)$ can be written as a product of face-group elements $g^F$ where $F$ is an \emph{upper} face, in other words that $\cu^n(G_\bullet)$ is generated just by the upper-face groups $G_{(F_i)}$. This follows from the fact that every face-group element $g^F$ is a product of upper-face group elements. Indeed, the indicator function $1_F: \{0,1\}^n\to\{0,1\}$ can always be written as an integer combination of indicator functions of upper faces of codimension at most $\codim(F)$ (this can be proved by induction on $\codim(F)$).

Our task is thus reduced to showing that any product of finitely many upper-face group elements has the claimed unique factorization. This has been done in several places, for instance in \cite[Appendix E]{GTlin}, but we recall the argument here for completeness.

Observe that if $F<F'$ are two upper faces, then $G_{(F')}\cdot G_{(F)} \subset G_{(F)}\cdot G_{(F')} \cdot G_{(F\cap F')}$. Indeed, for every $g\in G_{\codim(F)}, g'\in G_{\codim(F')}$, we have
${g'}^{F'}\; g^F = g^F \; {g'}^{F'} \; [{g'}^{F'},g^F]$, where $[{g'}^{F'},g^F]= [g',g]^{F' \cap F}\in G_{(F'\cap F)}$. Using this fact we can obtain the desired rearrangement by first moving all elements of $G_{(F_0)}$ to the left of the product, then using the closure of this group to combine them as a single element $g_0^{F_0}$; next we move to the left all $G_{(F_1)}$ elements in the remaining product, to put them next to $g_0^{F_0}$, and combine these to obtain $g_1^{F_1}$. Continuing this way yields the claimed factorization.

Uniqueness follows from \eqref{eq:cubefactn}. Indeed, note first that $\supp(v_i)\supset\supp(v)$ implies $v_i\geq v$. This implies that for any upper face $F(v_i)$, the only other upper faces $F(v)$ containing $v_i$ are those such that $\supp(v)\subset\supp(v_i)$ and $v< v_i$. From  \eqref{eq:cubefactn} we can therefore determine the coefficients $g_i$ by induction as follows:
\begin{equation}\label{eq:cubecoeffs}
 g_0=\q(v_0),\text{ and for each }i>0,\; g_i=  \Big( \prod_{\substack{ j < i : \;\\ \supp(v_j)\subset\supp(v_i)}} g_j^{F_j}\Big)^{-1}\cdot \q \,(v_i),
\end{equation}
where the order from left to right in the product is increasing in the colex order of $v_j$.
\end{proof}

For $j\in [n]$, let $\{0,1\}^n_{\leq\,j}:=\{v\in \{0,1\}^n: |v|\leq j\}$.

\begin{corollary}\label{cor:cube-det}
If $G_\bullet$ has degree $d$, then for every $\q\in \cu^n(G_\bullet)$, every value $\q(v)$ with $|v|>d$ is a word in the values $\q(v)$ with $|v|\leq d$. In particular, the map $\q$ is entirely determined by its restriction to $\{0,1\}^n_{\leq\, d}$.
\end{corollary}
\begin{proof}
Let $F_i$ be an upper face with $\codim(F_i)=|v_i|\leq d$. By \eqref{eq:cubecoeffs}, $g_i$ is a word in the values $\q(v_j)$ with $v_j\leq v_i$ with $\supp(v_j)\subset\supp(v_i)$. On the other hand any $g_i$ with $\codim(F_i)>d$ is trivial by assumption. Hence we are done, by uniqueness of \eqref{eq:cubefactn}.
\end{proof}

\noindent This corollary will be used to show that $\cu^n(G_\bullet)$ satisfies the corner-completion axiom. The following result shows that these cubes also satisfy the composition axiom.

\begin{lemma}\label{lem:G-facemap-inv}
Let $\q\in \cu^n(G_\bullet)$ and let $\phi:\{0,1\}^m \to \{0,1\}^n$ be a morphism. Then $\q\co \phi\in \cu^m(G_\bullet)$.
\end{lemma}
\noindent Given a map $g$ defined on $\{0,1\}^n$, for each $i\in\{0,1\}$ we define $g(\cdot,i)$ on $\{0,1\}^{n-1}$ by
$g(v,i)= g( (v,i) )$, where $(v,i):=\big(v\sbr{1},\ldots, v\sbr{n-1}, i\big)\in \{0,1\}^n$.

\begin{proof}
Let us first assume that $\phi$ is injective, thus $\phi$ is a bijection $\{0,1\}^m\to \tIm(\phi)$. For every $v\in \{0,1\}^m$, we have
\[
\q(\phi(v))= \prod_{i=0}^{2^n-1} g_i^{F_i}(\phi(v)) = \prod_{i:F_i\cap \tIm(\phi) \neq\emptyset} g_i^{F_i}(\phi(v)).
\]
For each face $F_i$ such that $ F_i \cap \tIm(\phi) \neq\emptyset$, with data $(I_i,x_i)$, fix some $y_i=\phi(v_i)\in \tIm(\phi)$ such that ${y_i}|_{I_i}=x_i$. Let $\beta$ denote the function $J_\phi\to [m]$, $j\mapsto i$ for all $j\in J_\phi(i)$ (recall \eqref{eq:Ji-set}). We can then see that $F'_i:=\phi^{-1} (F_i\cap \tIm(\phi))$ is a face of $\{0,1\}^m$ of codimension at most $\codim(F_i)$. Indeed, $F_i'$ is the face with data $I'=\beta(J_\phi\cap I_i)$ and $x'=\phi^{-1}(y_i)|_{ I'}={v_i}|_{ I'}$. (To see this, note that $v\in F_i'$ if and only if $\phi(v)\in \tIm(\phi) \cap F_i$, if and only if $\phi(v)\sbr{j}=x_i\sbr{j}$ for each $j\in I_i$, and this in turn holds if and only if $\phi(v)\sbr{j}
=y_i\sbr{j}$ for each $j\in J_\phi \cap I_i$. Thus we have $v\in F_i'$ if and only if $\phi(v)|_{J_\phi \cap I_i}=  \phi(v_i)|_{J_\phi \cap I_i}$, and this holds if and only if $v|_{I'}=v_i|_{I'}$.)\\
\indent Therefore we have
\[
\q(\phi(v))= \prod_{i: F_i\cap \tIm(\phi) \neq\emptyset} g_i^{F_i\cap \tIm(\phi)}(\phi(v)) =\prod_{i: F_i\cap \tIm(\phi) \neq\emptyset} g_i^{\phi^{-1}(F_i\cap \tIm(\phi))}(v) =\prod_{i: F_i\cap \tIm(\phi) \neq\emptyset} g_i^{ F_i'}(v),
\]
hence $\q\co \phi$ is a product of face-group elements and is therefore in $\cu^m(G_\bullet)$. 

Now suppose that there is a set $S\subset [m]$ of indices of coordinates of $v$ on which $\phi(v)$ does not depend. We show by induction on $|S|$ that $\q\co \phi\in \cu^m(G_\bullet)$. The case $|S|=0$ corresponds to $\phi$ being injective. If $|S|>0$, fix any $s\in S$, and note that it suffices to show that $\q\co\phi\co \theta \in \cu^m(G_\bullet)$ for the automorphism $\theta$ permuting coordinate-indices $s$ and $m$, so we may assume that $s=m$. Thus $g':=\q\co\phi$ is a map $\{0,1\}^m\to G$ such that $g'(\cdot,0)=g'(\cdot,1)$ is an element of $\cu^{m-1}(G_\bullet)$ (by induction), with coefficients $g_i$. We claim that $g' \in \cu^m(G_\bullet)$. In fact, $g'$ can be checked to be the $m$-cube with the following coefficients: if $v_i'\in \{0,1\}^m$ has $v_i'\sbr{m}=0$, then $g_i'$ equals the coefficient $g_j$ corresponding to  $v_j=(v_i'\sbr{1},v_i'\sbr{2},\ldots,v_i'\sbr{m-1})$, and otherwise $g_i'=\id_G$.
\end{proof}
We can now give the main result of this subsection.
\begin{proposition}\label{prop:FilGp=nilspace}
Let $(G,G_\bullet)$ be a filtered group. Then $G$ together with $(\cu^n(G_\bullet))_{n\geq 0}$ is a nilspace. It is a $k$-step nilspace if and only if $\deg(G_\bullet)\leq k$.
\end{proposition}
\begin{proof}
Ergodicity holds if and only if the first two terms of $G_\bullet$ are $G_0=G_1=G$. The composition axiom follows from Lemma \ref{lem:G-facemap-inv}. To prove the completion axiom, we argue by induction on $d=\deg(G_\bullet)$. For $d=1$ the group $G$ is abelian and completion was proved for Proposition \ref{prop:1-step-ns}. Let $d>1$ and let $G' = G/G_d$, with filtration $G'_\bullet=(G_i/G_d)$ of degree at most $d-1$. By induction $\cu^n(G'_\bullet)$ satisfies the completion axiom. Suppose that $\q'$ is an $n$-corner on $G$. Then $\pi_d\co \q'$ is clearly an $n$-corner on $G/G_d$, so it has a completion, which by surjectivity of $\q\mapsto \pi_d\co \q$ equals $\pi_d\co \tilde \q$ for some $\tilde \q\in \cu^n(G_\bullet)$. We shall now produce a sequence $\q_0,\q_1,\ldots, \q_d$ of elements of $\cu^n(G_\bullet)$ such that $\q_j$ agrees with $\q'$ on all of $\{0,1\}^n_{\leq\, j}$.

First, we obtain $\q_0$ such that $\q_0(0^n) = \q'(0^n)$ and $\pi_d\co \q_0 = \pi_d \co \q'$ (just left-multiplying every entry of $\tilde \q$ by $\q(0^n)\, \tilde \q(0^n)^{-1}\in G_d$). Now for $1\leq j\leq d$ suppose that a cube $\q_{j-1}$ has been constructed such that $\pi_d\co \q_{j-1} = \pi_d \co \q'$ and $\q_{j-1}(v)=\q'(v)$ for all $v\in \{0,1\}^n_{<\, j}$. We shall obtain $\q_j$ by correcting the differences between $\q_{j-1}$ and $\q'$ at elements $v$ with $|v| = j$. Note that $\pi_d\co \q_{j-1} =\pi_d\co \q'$ implies that these discrepancies involve only elements of $G_d$, that is for any such $v$ there is $g_v \in G_d\subset G_j$ such that $\q_{j-1}(v) g_v = \q(v)$. Letting $F(v)$ be the corresponding upper face (of codimension $j$), we therefore have $g_v^{F(v)}\in G_{(F(v))}$. We define $\q_j= \q_{j-1} \cdot \prod_{v:|v|=j} g_v^{F(v)}$. This lies in $\cu^n(G_\bullet)$, and agrees with $\q$ on  $\{0,1\}^n_{\leq\, j}$ as required (note that each $F(v)$ intersects $\{0,1\}^n_{\leq\, j}$ only at $v$).

We have thus obtained a cube $\q_d \in \cu^n(G_\bullet)$ that agrees with $\q'$ on all of $\{0,1\}^n_{\leq\, d}$. This agreement is now easily extended to every $(n-1)$-face containing $0^n$ (and thus to all of $\{0,1\}^n\setminus\{1^n\}$) applying composition and Corollary \ref{cor:cube-det}. 
\end{proof}
\noindent The general nilspace structure described in Proposition \ref{prop:FilGp=nilspace} is the one that we shall consider by default on a filtered group. Thus, from now on, given such a group $(G,G_\bullet)$, by ``a cube on $G$" we shall always mean an element of $\cu^n(G_\bullet)$ for some $n\geq 0$. There is, however, a slightly more general version of this cube structure, that can also be useful.
\begin{defn}\label{def:genindcube}
Let $(G,G_\bullet)$ be a filtered group, let $i_1,\dots,i_n$ be non-negative integers, and let $F=F(v)$ be an upper face in $\{0,1\}^n$.  We denote by $G_{(F)}^{(i_1,\ldots,i_n)}$ the subgroup of $G^{\{0,1\}^n}$ consisting of elements $g^F$ where $g\in G_{\sum_{j\in \supp v} i_j}$. We denote by $\cu^n_{(i_1,\ldots,i_n)}(G_\bullet)$ the subgroup of $G^{\{0,1\}^n}$ generated by the groups $G_{(F)}^{(i_1,\ldots,i_n)}$.
\end{defn}
\noindent This notion appeared already in \cite[Definition B.2]{GTZ}. (Note that $\cu^n(G_\bullet)=\cu^n_{(1,\dots,1)}(G_\bullet)$.) Arguing as in the proof of Lemma \ref{lem:cubefactn}, we obtain the following result.
\begin{lemma}
An element $\q\in G^{\{0,1\}^n}$ lies in $\cu^n_{(i_1,\ldots,i_n)}(G_\bullet)$ if and only if it has a unique factorization of the form
\begin{equation}\label{eq:cubefactn2}
\q= g_0^{F_0}g_1^{F_1}\cdots g_{2^n-1}^{F_{2^n-1}},
\end{equation}\
with $g_i\in  G_{\sum_{j\in \supp v_i} i_j}$ $\forall\, i$.
\end{lemma}
\noindent We now wish to complete the picture concerning the cube structures $\cu^n(G_\bullet)$ by showing that they satisfy an analogue of Proposition \ref{prop:1-degree-cubes}.  We have already done this for part $(i)$ of that proposition (in Lemma \ref{lem:cubefactn}). For part $(ii)$, we need an appropriate analogue of an affine homomorphism $\Z^n\to G$. A suitable notion turns out to be that of a \emph{polynomial map} from $\Z^n$ to $G$. This motivates the discussion of general polynomial maps between filtered groups, given in the next subsection. We shall see in particular that these maps are precisely the morphisms that preserve the nilspace structures formed by these cubes.

\medskip
\subsection{Nilspace morphisms between filtered groups: polynomial maps}\label{subsec:polys}
\medskip

\begin{defn}[Nilspace morphism]
Let $\ns, \nss$ be nilspaces. A map $g: \ns\to \nss$ is a \emph{nilspace morphism} if for every $n\geq 0$, for every cube $\q\in \cu^n(\ns)$ we have $g\co\q\in \cu^n(\nss)$. We denote the set of these morphisms by $\hom(\ns,\nss)$.
\end{defn}
\noindent In the case of two filtered groups $(G,G_\bullet)$, $(H,H_\bullet)$ we shall denote the set of morphisms between the corresponding nilspaces by $\hom( H_\bullet, G_\bullet)$.
\begin{defn}[Polynomial maps]
Let $(G,G_\bullet)$ and $(H,H_\bullet)$ be filtered groups. Given a map $g:H\to G$, and $h\in H$, we define the map $\partial_h g:H\to G$ by $\partial_h g(x) = g(x)^{-1} g(xh)$. 
We say that $g$ is a \emph{polynomial map} (adapted to $H_\bullet,G_\bullet$) if for every sequence of non-negative integers $i_1,i_2,\ldots,i_n$ and elements $h_j\in H_{i_j}$, $j\in[n]$, we have $\partial_{h_1} \partial_{h_2}\cdots \partial_{h_n} g (x) \in G_{i_1+\cdots + i_n}$ for all $x\in H$. The set of such maps is denoted by $\poly(H_\bullet,G_\bullet)$.
\end{defn}

\begin{example}\label{ex:polymaps}
For $H=\Z$ with the lower central series, the corresponding set of polynomial maps is denoted by $\poly(\Z,G_\bullet)$ and its elements are called \emph{polynomial sequences} on $G$ (one can add `adapted to $G_\bullet$' when the filtration needs to be specified) \cite{GTOrb,Leib}. If $H$ is an  abelian group with lower central series $H_\bullet$, and $G$ is abelian with $G_\bullet$ being the degree $d$ filtration $G_0=\cdots=G_d=G$, $G_{d+1}=\{\id_G\}$, then  $\poly(H_\bullet,G_\bullet)$ is the set of maps $g$ such that for every $x\in H, h=(h_1,\ldots,h_{d+1})\in H^{d+1}$ we have $\sum_{v\in \{0,1\}^{d+1}} (-1)^{|v|} g(x+v\cdot h)=0$. For $H=\R/\Z$ these are the `globally polynomial phase functions', familiar in arithmetic combinatorics; see for instance \cite[\S 3]{GTinv}.
\end{example}

The main result of this section is the following important fact.

\begin{theorem}\label{thm:homs=polys}
Let $(G,G_\bullet)$, $(H,H_\bullet)$ be filtered groups. Then
\begin{equation}\label{eq:homs=polys}
\hom(H_\bullet, G_\bullet)=\poly(H_\bullet,G_\bullet).
\end{equation}
\end{theorem}
\noindent This result yields a functor from the category of filtered groups with polynomial maps to the category of nilspaces with nilspace morphisms, namely the functor that sends a filtered group $(G,G_\bullet)$ to the nilspace $\big(G, (\cu^n(G_\bullet))_{n\geq 0}\big)$ and sends a map in $\poly(H_\bullet,G_\bullet)$ to the same map viewed as a morphism in $\hom(H_\bullet, G_\bullet)$. (This functor is forgetful; see Remark \ref{rem:forget}.)

Theorem \ref{thm:homs=polys} has the following remarkable consequence.
\begin{corollary}\label{cor:LazLeib}
For any filtered groups $(G,G_\bullet)$, $(H,H_\bullet)$, the set $\poly(H_\bullet,G_\bullet)$ with pointwise multiplication is a group.
\end{corollary}
\noindent This corollary follows from \eqref{eq:homs=polys} together with the fact that $\hom(H_\bullet, G_\bullet)$ is a group under pointwise multiplication, which is a special case of the following  lemma.
\begin{lemma}
Let $\ns$ be a nilspace and let $(G,G_\bullet)$ be a filtered group. Then $\hom(\ns,G_\bullet)$ is a group under pointwise multiplication.
\end{lemma}
\noindent This lemma follows clearly from the fact that the cubes $\cu^n(G_\bullet)$ form a group.\\ 
\indent Corollary \ref{cor:LazLeib} was first established for polynomial sequences. This special case is known as the Lazard-Leibman theorem, it was proved by Leibman in \cite{Leib}, and earlier, for Lie groups, by Lazard \cite{Laz}. The more  general case of polynomial maps between groups was then proved by Leibman in \cite{Leib2}. These original proofs did not go via the charaterization of polynomial maps as nilspace morphisms, as we do here. A special case of this characterization appeared in \cite[Proposition 6.5]{GTOrb}, and the general version appeared (with slightly different language) in \cite[Theorem B.3]{GTZ}.\\
\indent Another nice consequence of Theorem \ref{thm:homs=polys} is the following result, which gives the desired generalization of part (ii) of Proposition \ref{prop:1-degree-cubes}.
\begin{corollary}\label{cor:cubes=respoly}
A map $\q:\{0,1\}^n\to G$ is in $\cu^n(G_\bullet)$ if and only if it extends to a polynomial map $g\in \poly(\Z^n,G_\bullet)$.
\end{corollary}
\noindent Note that throughout this section the implicit filtration on $\Z^n$ is just the lower central series $G_0=G_1=\Z^n$, $G_2=\{0^n\}$. Let us establish this corollary before turning to the proof of the theorem.
\begin{proof}[Proof of Corollary \ref{cor:cubes=respoly}]
The `if' direction is the easiest. Note first that the map $\theta$ that embeds $\{0,1\}^n$ in $\Z^n$ in the natural way is trivially in $\cu^n(\Z^n)$. 
Hence if $\q=g\co\theta$ for $g\in \poly(\Z^n,G_\bullet)$, then, since by the theorem $g$ is a morphism of cubes, we have $\q \in \cu^n(G_\bullet)$.

For the `only if' direction, given an $n$-cube $\q$, we write the unique factorization
\[
\q(v)=g_0^{F_0}(v)\; g_1^{F_1}(v)\; \cdots \;g_{2^n-1}^{F_{2^n-1}}(v),
\]
and we then define the map $g:\Z^n\to G$ by
\[
g(\mathbf{t})=g(t_1,\ldots,t_n) = \prod_{j=0}^{2^n-1} g_j^{\binom{\mathbf{t}}{v_j}},
\]
where $\binom{\mathbf{t}}{v_j}:=\binom{t_1}{v_j\sbr{1}}\binom{t_2}{v_j\sbr{2}}\cdots \binom{t_n}{v_j\sbr{n}}$, and where $v_j$ is the element determining the upper face $F_j=F(v_j)$ (recall Definition \ref{def:up-faces}). Note that the restriction of $g$ to $\{0,1\}^n$ equals $\q$. Observe also that each map $\mathbf{t}\mapsto g_j^{\binom{\mathbf{t}}{v_j}}$ is in $\poly(\Z^n,G_\bullet)$ (this is easily seen from the definition) and so, by the group property for $\poly(\Z^n,G_\bullet)$ (which follows from Theorem \ref{thm:homs=polys}) we have $g\in \poly(\Z^n,G_\bullet)$ as well.
\end{proof}
\noindent To prove Theorem \ref{thm:homs=polys}, we shall use the following construction, which will also play a crucial role in the next chapter.

\begin{defn}[Arrows]\label{def:arrow}
Let $\ns$ be a set, let $n,k\in \N$, and let us denote elements of $\{0,1\}^{n+k}$ by pairs $(v,w)$, $v\in \{0,1\}^n, w\in \{0,1\}^k$. Given two maps $\q_0,\q_1:\{0,1\}^n\to \ns$, we define the \emph{$n$-dimensional $k$-arrow}, denoted by $\arr{\q_0,\q_1}_k$, to be the following map:
\begin{eqnarray*}
\arr{\q_0,\q_1}_k\quad: \quad \{0,1\}^{n+k} & \to & \quad \ns\\
(v,w) \hspace{0.5cm}& \mapsto & \left\{\begin{array}{lr} \q_0(v),& w\neq 1^k \\
  \q_1(v), & w=1^k\end{array}\right..
\end{eqnarray*}
\end{defn}
\noindent Usually, the dimension of the cubes $\q_0,\q_1$ (and hence of the arrow $\arr{\q_0,\q_1}_k$) will be clear from the context, and we shall then call $\arr{\q_0,\q_1}_k$ just `the $k$-arrow of $\q_0,\q_1$'.

We shall see in Chapter \ref{chap:chargenils} that arrows yield a useful  general construction of new nilspaces from old ones. For now we shall use arrows just for cubes on filtered groups.\\
\indent The following result provides a convenient way to check whether an arrow on a filtered group is itself a cube. Given a filtration $G_\bullet$ and $\ell\in \N$, we denote by $G_\bullet^{+\ell}$ the shifted filtration, with $i$-th term $G_{i+\ell}$.
\begin{lemma}\label{lem:filtarrow}
Let $(G,G_\bullet)$ be a filtered group, let $\q_0,\q_1:\{0,1\}^n\to G$, and let $k\in \N$. Then $\arr{\q_0,\q_1}_k$ lies in $\cu^{n+k}_{(i_1,\ldots,i_{n+k})}(G_\bullet)$ if and only if $\q_0\in \cu^n_{(i_1,\ldots,i_n)}(G_\bullet)$ and $\q_0^{-1} \q_1\in\cu^n_{(i_1,\ldots,i_n)}(G_\bullet^{+\ell})$, where $\ell=\sum_{j=n+1}^{n+k} i_j$.
\end{lemma}
\begin{proof}
If $f=\arr{\q_0,\q_1}_k$ is in $\cu^{n+k}_{(i_1,\ldots,i_{n+k})}(G_\bullet)$ then by \eqref{eq:cubefactn2} we have that $\q_0\in \cu^n_{(i_1,\ldots,i_n)}(G_\bullet)$. Moreover, from \eqref{eq:cubefactn2} and the definition of $f$ we see that every coefficient $g_{(v,w)}$ of $f$ corresponding to an upper face $F((v,w))$ of $\{0,1\}^{n+k}$ with $w\neq 0^k,1^k$ is trivial. It follows that for $v\in \{0,1\}^n$ we have
\[
\q_0^{-1}(v)\q_1(v)= f(v,0^k)^{-1}\, f(v,1^k)= \prod_{u\in \{0,1\}^n} g_{(u,1^k)}^{F((u,1^k))}(v,1^k),
\]
where $g_{(u,1^k)}\in G_{\sum_{j\in \supp u} i_j+\sum_{j=n+1}^{n+k} i_j}$, so $\q_0^{-1}\,\q_1$ is indeed in $\cu^n_{(i_1,\ldots,i_n)}(G_\bullet^{+\ell})$ as claimed.

For the converse, note that the factorizations for $\q_0$ and $\q_0^{-1}\,\q_1$ given by \eqref{eq:cubefactn2} combine to give a factorization of the form \eqref{eq:cubefactn2} for $\arr{\q_0,\q_1}_k$.
\end{proof}
\noindent We now prove Theorem \ref{thm:homs=polys} in two steps. 
\begin{lemma}\label{lem:polys-are-homs}
We have $\hom(H_\bullet, G_\bullet) \subset \poly(H_\bullet,G_\bullet)$.
\end{lemma}
\begin{proof}
We have to show that for every $g\in \hom(H_\bullet, G_\bullet)$, for every integer $n\geq 0$, integers $i_1,\dots,i_n\geq 0$ and elements $h_j\in H_{i_j}$, $j\in [n]$, we have $\partial_{h_1}\cdots\partial_{h_n}g(x)\in G_{i_1+\cdots+i_n}$ for all $x\in H$. This is clear for $n=0$. For $n>0$, by induction it suffices to show that for every $i\geq 0$ and $h\in H_i$ we have $\partial_h g\in \hom(H_\bullet, G_\bullet^{+i})$.  Given any $\q\in \cu^m(H_\bullet)$, we have by Lemma \ref{lem:filtarrow} that $\arr{\q,\q\cdot h}_i\in \cu^{m+i}(H_\bullet)$. We then have $g\co \arr{\q,\q\cdot h}_i\in \cu^{m+i}(G_\bullet)$. But $g\co \arr{\q,\q\cdot h}_i=\arr{g\co \q,g\co (\q\cdot h)}_i$, and then this being in $\cu^{m+i}(G_\bullet)$ and $g\co \q$ being in $\cu^m(G_\bullet)$ implies by Lemma \ref{lem:filtarrow} that $(g\co \q)^{-1}\,\cdot\,g\co (\q\cdot h) \in \cu^m(G_\bullet^{+i})$. Since $(g\co \q)^{-1}\,\cdot\,g\co (\q\cdot h)=(\partial_h g)\co \q$, the proof is complete.
\end{proof}
\noindent We now prove a result which implies the opposite inclusion $\hom(H_\bullet,G_\bullet)\supset \poly(H_\bullet,G_\bullet)$. 
\begin{lemma}\label{lem:homs-are-polys}
For every integer $n\geq 0$ the following holds. Given any filtered groups $(G,G_\bullet)$, $(H,H_\bullet)$ and any map $g\in \poly(H_\bullet,G_\bullet)$, for any $i_1,\ldots,i_n\geq 0$ and any $\q\in \cu^n_{(i_1,\ldots,i_n)}(H_\bullet)$ we have $g\co\q \in \cu^n_{(i_1,\ldots,i_n)}(G_\bullet)$.
\end{lemma}
\noindent Thus polynomial maps conserve in fact the more general cube structures from Definition \ref{def:genindcube}. Note that the last two  lemmas combined tell us that nilspace morphisms also conserve these structures. 
\begin{proof}
(We follow an argument from \cite[Appendix B]{GTZ}.) We argue again by induction on $n$. Given $n>0$ let $\q\in \cu^n_{(i_1,\ldots,i_n)}(H_\bullet)$ and note that, letting $\q_0:=\q(\cdot,0), \q_1:=\q(\cdot,1)$, by Lemma \ref{lem:filtarrow} we have $\q_0\in \cu^{n-1}_{(i_1,\ldots,i_{n-1})}(H_\bullet)$ and $\q_0^{-1}\q_1\in \cu^{n-1}_{(i_1,\ldots,i_{n-1})}(H_\bullet^{+i_n})$. We want to show that $g\co \q$ lies in $\cu^n_{(i_1,\ldots,i_n)}(G_\bullet)$. Since $(g\co \q) (\cdot,t)= g\co \q_t$ for $t=0,1$, and by induction $g\co\q_0\in \cu^{n-1}_{(i_1,\ldots,i_{n-1})}(G_\bullet)$, by Lemma \ref{lem:filtarrow} it suffices to show that $(g\co \q_0)^{-1} g\co \q_1 \in \cu^{n-1}_{(i_1,\ldots,i_{n-1})}(G_\bullet^{+i_n})$.\\
\indent Now $\q_0^{-1}\q_1=h_1^{F_1}h_2^{F_2}\cdots h_t^{F_t}$ for upper-face group elements $h_j^{F_j}\in \cu^{n-1}_{(i_1,\ldots,i_{n-1})}(H_\bullet^{+i_n})$. Moreover, note that $(g\co \q_0)^{-1} g\co \q_1 $ has the telescoping expansion
\begin{eqnarray*}
&& (g\co \q_0)^{-1}\;\; g\co \big(\q_0 h_1^{F_1}\big)\; g\co \big(\q_0 h_1^{F_1}\big)^{-1}\;\; g\co \Big(\big(\q_0 h_1^{F_1}\big) h_2^{F_1}\Big)\; g\co \Big(\big(\q_0 h_1^{F_1}\big) h_2^{F_1}\Big)^{-1} \cdots \\
&&\hspace{4cm}\cdots g\co \Big(\q_0 h_1^{F_1}h_2^{F_2}\cdots h_{t-1}^{F_{t-1}}\Big)\;g\co \Big(\q_0 h_1^{F_1}h_2^{F_2}\cdots h_{t-1}^{F_{t-1}}\Big)^{-1} g\co \q_1.
\end{eqnarray*}
Therefore, it suffices to show that for every cube $\q \in  \cu^{n-1}_{(i_1,\ldots,i_{n-1})}(H_\bullet)$ and every upper-face-group element $h^F\in  \cu^{n-1}_{(i_1,\ldots,i_{n-1})}(H_\bullet^{+i_n})$, we have $(g\co \q)^{-1} g\co (\q h^F)\in \cu^{n-1}_{(i_1,\ldots,i_{n-1})}(G_\bullet^{+i_n})$ .\\
\indent Let $f(v):=(g\co \q)^{-1} g\co (\q h^F)(v)$, and note that this equals $\partial_h g(\q(v))$ for $v\in F$ and $\id_G$ otherwise. Let $w\in \{0,1\}^{n-1}$ be the element with support $I$ of size $k$ such that $F=F(w)$. By composing with an automorphism (permuting the indices $i_j$ accordingly) we can assume that $I=[n-k,n-1]$ and then $f$ is the $(n-1)$-dimensional arrow $\arr{\id_G,(\partial_h g) \co \q}_k$.
By Lemma \ref{lem:filtarrow}, $f$ is therefore in $\cu^{n-1}_{(i_1,\ldots,i_{n-1})}(G_\bullet^{+i_n})$ if and only if $(\partial_h g) \co \q$ restricted to $\{0,1\}^{n-1-k}$ lies in $\cu^{n-1-k}_{(i_1,\ldots,i_{n-1-k})}(G_\bullet^{+i_n+\sum_{j\in I}i_j})$. Since $h\in G_{\sum_{j\in I}i_j}$, we have that $\partial_h g$ is a $G_{\sum_{j\in I}i_j}$-valued polynomial map on $H$, so by the induction hypothesis $\partial_h g$ maps $\cu^{n-1}_{(i_1,\ldots,i_{n-1})}(H_\bullet^{+i_n})$ into $\cu^{n-1}_{(i_1,\ldots,i_{n-1})}(G_\bullet^{+i_n+\sum_{j\in I}i_j})$ under composition. In particular, since $\q$ restricted to $\{0,1\}^{n-1-k}$ lies in $\cu^{n-k-1}_{(i_1,\ldots,i_{n-1})|_{[n-1]\setminus I}}(H_\bullet^{+i_n})$, we have $(\partial_h g) \co \q|_{\{0,1\}^{n-1-k}}\in\cu^{n-1-k}_{(i_1,\ldots,i_{n-1-k})}(G_\bullet^{+i_n+\sum_{j\in I}i_j})$ as required.
\end{proof}
\noindent In the next subsection we complete the generalization of Proposition \ref{prop:1-degree-cubes} by giving the analogue of part $(iii)$. This analogue will consist in a characterization of cubes on filtered groups in terms of certain equations that generalize \eqref{eq:st-cube-2-face}. Recall that Corollary \ref{cor:cube-det} tells us that if $G$ has degree $d$ then the value of every $(d+1)$-cube at every point $v\in \{0,1\}^{d+1}$ is a word in the values at the other points. The equations just mentioned will tell us explicitly what this word is. This third characterization of cubes on filtered groups will be very convenient to describe an important type of nilspace structure on an abelian group, namely the so-called \emph{degree-$k$ structure} (see Subsection \ref{sect:deg-k-str}).

\medskip
\subsection{Defining cubes on a filtered group in terms of equations}\label{subsec:cubeqs}
\medskip

We shall use the following generalization of the function $\sigma_2$ from Definition \ref{def:abelian-sign}.
\begin{defn}\label{defn:gen-sign}
Let $G$ be a group. We define the function $\sigma_n: G^{\{0,1\}^n}\to G$ recursively, for each integer $n\geq 0$, as follows. Let $\sigma_0(g):=g$ for any $g\in G=G^{\{0,1\}^0}$, and for $n>0$, given $g:\{0,1\}^n\to G$, let
\begin{equation}\label{eq:recurGray}
\sigma_n(g):= \sigma_{n-1}\big(g(\cdot ,1)\big)^{-1}\; \sigma_{n-1}\big(g(\cdot ,0)\big).
\end{equation}
\end{defn}
\noindent We have $\sigma_1(g)= g_1^{-1}\,g_0$, $\sigma_2(g) = g_{01}^{-1}\; g_{11} \;g_{10}^{-1}\; g_{00}$, and so on. Expanding the product $\sigma_n(g)$ and reading right to left (say), we see that it is an alternating product taken along a Gray code (or reflective binary code) on the discrete $n$-cube.
\begin{defn}
The \emph{Gray order} $\gamma_n:\{0,1\}^n\to [0,2^n-1]$ is defined by induction on $n$ as follows. For $n=1$ we set $\gamma_1(0)=0, \gamma_1(1)=1$. For $n>1$ we set
$\gamma_n(v\sbr{1},\ldots,v\sbr{n}) = \gamma_{n-1}(v\sbr{1},\ldots,v\sbr{n-1})$ if $v\sbr{n}=0$, and if $v\sbr{n}=1$ we set $\gamma_n(v\sbr{1},\ldots,v\sbr{n})=2^n-\gamma_{n-1}(v\sbr{1},\ldots,v\sbr{n-1})$ .
\end{defn}
\noindent We note the expression $\sigma_n(g) = \prod_{j=0}^{2^n-1} g(\gamma_n^{-1}(j) )^{(-1)^j}$. We will have more use below for expression \eqref{eq:recurGray}, which is helpful in some inductive arguments. In fact, that expression leads to  a third viewpoint on $\sigma_n$ described in the following lemma, which will be crucial below.
\begin{lemma}
Let $G$ be a group and let $g:\Z^n\to G$. Then
\begin{equation}\label{eq:sigma=deriv}
\sigma_n(g)=\partial_{e_n}\partial_{e_{n-1}}\cdots \partial_{e_1} g(0^n).
\end{equation}
\end{lemma}
\noindent Here $\{e_i: i\in [n]\}$ is the standard basis of $\R^n$. 
\begin{proof}
For $n=1$ we clearly have $\sigma_1(g) = \partial_{e_1} g(0)$. Suppose that $n>1$, suppose the claim holds for $m<n$ and let $g:\Z^n\to G$. Then by induction $g(\cdot ,0)$ and $g(\cdot ,1)$ satisfy
\[
\sigma_{n-1}(g(\cdot,i)) = \partial_{e_{n-1}} \cdots \partial_{e_1} g (0^{n-1},i),\quad i=0,1.
\]
By \eqref{eq:recurGray} we then have
\begin{eqnarray*}
\sigma_n(g) & = &  \sigma_{n-1}(g(\cdot ,1))^{-1} \;\sigma_{n-1}(g(\cdot ,0))\\
& = & \big(\partial_{e_{n-1}} \cdots \partial_{e_1} g (0^{n-1},1)\;\big)^{-1} \;\partial_{e_{n-1}} \cdots \partial_{e_1} g (0^{n-1},0)\\
& = & \partial_{e_n}\partial_{e_{n-1}}\cdots \partial_{e_1} g(0^n).
\end{eqnarray*}
\end{proof}
\noindent Note that the increasing sequence of elements of $\{0,1\}^n$ taken in the Gray order gives a Hamiltonian path in the graph of 1-faces on this cube, starting from $0^n$.

The main result of this section is the following.
 
\begin{proposition}\label{prop:G-cubes=Gray-cubes}
Let $(G,G_\bullet)$ be a filtered group. We have $\q\in \cu^n(G_\bullet)$ if and only if for every $m$-face map $\phi$ into $\{0,1\}^n$ we have $\sigma_m(\q\co\phi)\in G_m$.
\end{proposition}
\noindent Thus, the $n$-cubes on $(G,G_\bullet)$ are the maps $\q:\{0,1\}^n\to G$ that satisfy the equation
\[
\pi_m\co \sigma_m(\q\co\phi)=\id
\]
for every face map $\phi:\{0,1\}^m\to \{0,1\}^n$.

We split the proof of Proposition \ref{prop:G-cubes=Gray-cubes} into two parts. Firstly, we establish the following stronger statement in one direction.

\begin{lemma}\label{lem:G-cubes=>Gray-cubes}
If  $\q\in \cu^n(G_\bullet)$ then for every morphism $\phi: \{0,1\}^m\to \{0,1\}^n$ we have $\sigma_m(\q\co\phi)\in G_m$.
\end{lemma}

\begin{proof}
Lemma \ref{lem:G-facemap-inv} implies $\q \co \phi\in\cu^m(G_\bullet)$. Corollary \ref{cor:cubes=respoly} gives $g\in \poly(\Z^m,G)$ such that $\q\co\phi (v)=g(v)$ for all $v\in \{0,1\}^m$. The result then follows from \eqref{eq:sigma=deriv} and the definition of $\poly(\Z^m,G_\bullet)$.
\end{proof}
\noindent Secondly, we have the following result in the converse direction.
\begin{lemma}\label{lem:Gray-cubes=>G-cubes}
Suppose that $\q:\{0,1\}^n\to G$ satisfies $\sigma_m(\q\co\phi)\in G_m$ for every $m$-face map $\phi$ into $\{0,1\}^n$.  Then $\q\in \cu^n(G_\bullet)$.
\end{lemma}

\begin{proof}
We argue by induction on the degree $d$ of $G_\bullet$. For $d=1$, $G$ is abelian and the claim follows from (the proof of) Proposition \ref{prop:1-degree-cubes}, so let $d>1$ and suppose that the lemma holds for every filtered group of degree less than $d$. Let $G' = G/G_d$, and let $G'_\bullet=(G_i/G_d)_{i\geq 0}$, a filtration of degree at most $d-1$. Suppose that $\q: \{0,1\}^n\to G$ satisfies the premise of the lemma. Then $\pi_d\co \q$ satisfies the premise in $(G',G'_\bullet)$, so $\pi_d \co \q \in \cu^n(G'_\bullet)$.

Note that $\q\mapsto \pi_d\co \q$ is a surjection $\cu^n(G_\bullet)\to \cu^n(G'_\bullet)$ (this follows from the surjectivity of $\pi_d:G\to G'$ combined with \eqref{eq:cubefactn}, or just with Definition \ref{def:G-cubes}). Hence there exists $\q'\in \cu^n(G_\bullet)$ such that $\pi_d\co \q=\pi_d\co \q'$.  We shall now produce a sequence $\q_0,\q_1,\ldots, \q_d$ of elements of $\cu^n(G_\bullet)$ such that $\q_j$ agrees with $\q$ on all of $\{0,1\}^n_{\leq\, j}$.

First, we obtain $\q_0$ such that $\q_0(0^n) = \q(0^n)$ and $\pi_d\co \q_0 = \pi_d \co \q$, by left-multiplying every entry of $\q'$ by $\q(0^n)\, \q'(0^n)^{-1}\in G_d$. Now for $1\leq j\leq d$ suppose that $\q_{j-1}$ has been constructed such that $\pi_d\co \q_{j-1} = \pi_d \co \q$ and $\q_{j-1}(v)=\q(v)$ for all $v\in \{0,1\}^n_{\leq j-1}$. We shall obtain $\q_j$ by correcting the discrepancies between $\q_{j-1}$ and $\q$ at elements $v$ with $|v| = j$. Note that $\pi_d\co \q_{j-1} =\pi_d\co \q$ implies that these discrepancies involve only elements of $G_d$, that is for any such $v$ there is $g_v \in G_d\subset G_j$ such that $\q_{j-1}(v) g_v = \q(v)$. Letting $F(v)$ be the corresponding upper face (of codimension $j$), we therefore have $g_v^{F(v)}\in G_{(F(v))}$. We define $\q_j= \q_{j-1} \cdot \prod_{v:|v|=j} g_v^{F(v)}$. This lies in $\cu^n(G_\bullet)$, and agrees with $\q$ on  $\{0,1\}^n_{\leq\, j}$ as required (note that each $F(v)$ intersects $\{0,1\}^n_{\leq\, j}$ only at $v$).

We have thus obtained $\q_d$. We now claim that the agreement between $\q_d$ and $\q$ extends to all of $\{0,1\}^n$, and so $\q=\q_d\in \cu^n(G_\bullet)$. Note that the above process cannot be continued in order to prove this claim, because now the new elements $g_v^{F(v)}$ that we would multiply by are not guaranteed to lie in $\cu^n(G_\bullet)$ anymore, since now $\codim(F(v))>d$. However, we have that $\q$ and $\q_d$ both satisfy  the premise of the lemma (for $\q_d$ this follows from Lemma \ref{lem:G-cubes=>Gray-cubes}) and agree on $\{0,1\}^n_{\leq d}$. This can be used to complete the proof, as follows. First let us show that for any $v$ with $|v|=d+1$ we have $\q(v)=\q_d(v)$. Since $|v|=d+1$, we can find a $(d+1)$-face map $\phi$ into $\{0,1\}^n$ with $\phi(0^{d+1})=v$ and $|\phi(w)|\leq d$ otherwise. Using the premise we then have $\q(v) = \q\co \phi(0^{d+1})  =  \q\co \phi(0^{d+1})\;\sigma_{d+1}(\q\co \phi)^{-1} =  \q_d \co \phi(0^{d+1})\;\sigma_{d+1}(\q_d \co \phi)^{-1} = \q_d(v)$. We can now use a similar argument to see that  $\q(v)=\q_d(v)$ for $|v|=d+2$, then for $|v|=d+3$, and so on, concluding finally that $\q=\q_d\in \cu^n(G_\bullet)$.
\end{proof}

Let us record the generalization of Proposition \ref{prop:1-degree-cubes} that we have finally obtained.

\begin{proposition}\label{prop:G-cubes}
Let $(G,G_\bullet)$ be a filtered group, and let $\q:\{0,1\}^n\to G$. We have $\q\in \cu^n(G_\bullet)$ if and only if one of the following equivalent statements holds:\vspace{-0.2cm}
\begin{enumerate}
\item There is a unique expression $\q= g_0^{F_0}g_1^{F_1}\cdots g_{2^n-1}^{F_{2^n-1}}$ with $g_i\in G_{\codim(F_i)}$ for each $i$.\vspace{-0.2cm}

\item The map $\q$ extends to a polynomial map in $\poly(\Z^n,G_\bullet)$.\vspace{-0.2cm}

\item For every $m$-face map $\phi:\{0,1\}^m \to \{0,1\}^n$ we have $\sigma_m(\q\co\phi)\in G_m$.
\end{enumerate}
\end{proposition}

\begin{remark}\label{rem:alternprop3gen}
One may replace ``every $m$-face map" in property $(iii)$ above by ``every morphism". Indeed, if the property holds with every $m$-face map, then the proposition tells us that $\q$ is a cube. But then, given any morphism $\phi: \{0,1\}^m\to \{0,1\}^n$, by composition $\q\co \phi$ is also a cube, so applying property $(iii)$ using as an $m$-face map the identity map on $\{0,1\}^m$, we deduce that $\sigma_m(\q\co\phi)\in G_m$ as required.
\end{remark}

\subsection{Cubes of degree $k$ on abelian groups}\label{sect:deg-k-str}

In this section we record a special case of the previous nilspace structures that will play an important role in the sequel. This is a $k$-step nilspace structure defined on an abelian group.

\begin{defn}\label{def:k-deg-ab-cubes}
Let $\ab$ be an abelian group. The \emph{degree-$k$ structure} on $\ab$, denoted by $\cD_k(\ab)$, is the nilspace formed by $\ab$ together with the cubes $\cu^n(\ab_{\bullet (k)})$, $n\in \N$, where $\ab_{\bullet(k)}$ denotes the maximal filtration of degree $k$ on $\ab$, namely the filtration with $\ab_0=\cdots=\ab_k=\ab$,\; $\ab_{k+1}=\{0_{\ab}\}$.
\end{defn}
\noindent Note that, by Proposition \ref{prop:G-cubes}, we have
\begin{equation}\label{eq:k-ab-cubes}
\cu^n(\cD_k(\ab))= \{\q:\{0,1\}^n\to \ab ~|~ \textrm{for every face map }\phi:\{0,1\}^{k+1}\to \{0,1\}^n,\; \sigma_{k+1}(\q\co\phi)=0\},
\end{equation}
where in this abelian setting we have $\sigma_{k+1}(\q\co\phi)=\sum_{v\in\{0,1\}^{k+1}}(-1)^{|v|} \q\co\phi(v)$. Note also that the nilspace $\cD_k(\ab)$ is $k$-fold ergodic.

\section{Quotients of filtered groups of degree $k$ as $k$-step nilspaces}\label{sec:filnilm}
The purpose of this section is to enlarge further our family of basic examples, to include some natural nilspaces that do not have a group structure. Given a filtered group $(G,G_\bullet)$ and a subgroup $\Gamma$ of $G$, the set of cosets $G/\Gamma$ may not have a group operation ($\Gamma$ may not be a normal subgroup), but one can still define a cube structure on it by projecting the cubes in $\cu^n(G_\bullet)$ down to $G/\Gamma$.

\begin{proposition}\label{prop:quotex}
Let $(G,G_\bullet)$ be a filtered group of degree $k$, let $\Gamma$ be a subgroup of $G$, and let $\pi_\Gamma:G\to G/\Gamma$ be the canonical projection. Then $\ns=G/\Gamma$ together with the sets
\begin{equation}\label{eq:quotex}
\cu^n(\ns)=\{\pi_\Gamma\co \q: \q\in \cu^n(G_\bullet)\} = (\cu^n(G_\bullet)\cdot \Gamma^{\{0,1\}^n})\,/\,\Gamma^{\{0,1\}^n} 
\end{equation}
is a $k$-step nilspace. We shall refer to such nilspaces as \emph{coset nilspaces}.
\end{proposition}

\begin{proof}
The composition axiom follows clearly from that for $(G,G_\bullet)$, and ergodicity follows from the transitivity of the action of $G$ on $G/\Gamma$. To check the completion axiom, one can argue again by induction on the degree $d$. For $d=1$, $G/\Gamma$ is just an abelian group. Let $d>1$, let $\q'$ be an $n$-corner on $\ns=G/\Gamma$ and let $\overline{\pi_d}$ be the projection $\ns\to \ns':= G'/\Gamma'$, where $\pi_d$ denotes the quotient homomorphism $G\to G'=G/G_d$, and $\Gamma'=(\Gamma\cdot G_d)/ G_d$. We thus have the following commutative diagram.
\[
\begin{tikzcd}
G \arrow{r}{\pi_d} \arrow[swap]{d}{\pi_\Gamma} & G' \arrow{d}{\pi_{\Gamma'}} \\
\ns=G/\Gamma \;\;\arrow{r}{\overline{\pi_d}} &\;\; \ns'=G'/\Gamma'
\end{tikzcd}
\]
The surjectivity of the map $\cu^n(G_\bullet)\to \cu^n(G_\bullet')$, $\q\mapsto \pi_d\co \q$  implies the surjectivity of $\cu^n(\ns)\to \cu^n(\ns')$, $\q\mapsto \overline{\pi_d}\co \q$. Note that $G'$ with $G'_\bullet=(G_i/G_d)_i$ is a filtered group of degree $d-1$. By induction there exists a cube in $\cu^n(\ns')$ completing $\overline{\pi_d}\co \q'$, and by the surjectivity just mentioned there exists $\q_0\in \cu^n(\ns)$ such that this completion is $\overline{\pi_d}\co \q_0$, that is we have $\overline{\pi_d}\co \q_0(v)=\overline{\pi_d}\co \q'(v)$ for all $v\neq 1^n$. By definition of $\cu^n(\ns)$ there exists $\tilde \q_0\in \cu^n(G_\bullet)$ such that $\q_0 = \pi_\Gamma\co \tilde \q_0$. Choose also any map $\tilde \q'\in G^{\{0,1\}^n\setminus\{1^n\}}$ such that $\pi_\Gamma\co \tilde \q'=\q'$.
Then $\overline{\pi_d}\co \q_0(v)=\overline{\pi_d}\co \q'(v)$ means that $\tilde \q_0 (v) G_d \Gamma = \tilde \q'(v)G_d \Gamma$. This implies that for each $v\neq 1^n$ there exists $g_v\in G_d$ such that $\pi_\Gamma(\tilde\q_0(v)g_v)=\q'(v)$. We can now carry out the same process as in the proof of Proposition \ref{prop:FilGp=nilspace} to correct $\tilde \q_0$ to a new cube $\tilde \q\in \cu^n(G_\bullet)$ satisfying $\pi_\Gamma(\tilde\q(v))=\q'(v)$ for all $v\neq 1^n$. Then $\q=\pi_\Gamma\co \tilde \q$ is a completion of $\q'$.
\end{proof}
\noindent The coset nilspace construction is important in that it captures the algebraic structure of a large class of nilspaces, namely \emph{filtered nilmanifolds}. From a purely algebraic point of view, filtered nilmanifolds are indeed examples of coset nilspaces, however they also have crucial topological properties afforded by strong topological assumptions on $G$ and $\Gamma$, namely that $G$ is a connected nilpotent Lie group and $\Gamma$ is a discrete, cocompact subgroup of $G$. Nilmanifolds play a crucial role in the topological part of the theory of nilspaces, which studies the \emph{compact nilspaces} mentioned in the introduction (see Chapter 3 of \cite{CamSzeg}; see also \cite{GTlin,GTOrb,GTZ} for more information on nilmanifolds). In these notes we shall only look at the algebraic properties of such spaces. It turns out that many of these general properties are often fully illustrated already by 2-step nilspaces. Among these, the following well-known example can be used as a source of intuition for several concepts and tools treated in Chapter \ref{chap:chargenils}.

\begin{example}[The Heisenberg nilmanifold]\label{ex:Heisen}
Consider the Heisenberg group
\[
H=\begin{psmallmatrix} 1 & \R & \R\\[0.1em]  & 1 & \R \\[0.1em]  &  & 1 \end{psmallmatrix}:=\Big\{ \begin{psmallmatrix} 1 & x_1 & x_3\\[0.1em]  & 1 & x_2 \\[0.1em]  &  & 1 \end{psmallmatrix}: x_i \in \R\Big\},
\]
let $\Gamma=\begin{psmallmatrix} 1 & \Z & \Z\\[0.1em]  & 1 & \Z \\[0.1em]  &  & 1 \end{psmallmatrix}$, and let $H_\bullet$ denote the lower central series on $H$, that is the filtration with $H_{(2)}=\begin{psmallmatrix} 1 & 0 & \R\\[0.1em]  & 1 & 0 \\[0.1em]  &  & 1 \end{psmallmatrix}$ and $H_{(3)}=\{\id_H\}$. Let $\ns$ be the nilspace consisting of the set $H/\Gamma$ together with the cube sets $\cu^n(H_\bullet)$ projected on $H/\Gamma$. By Proposition \ref{prop:quotex} we have that $\ns$ is a 2-step nilspace. We can identify the set $H/\Gamma$ with the fundamental domain $[0,1)^3$ for the right action of $\Gamma$ on $H$.
\end{example}

\section{Motivation for a general algebraic characterization}\label{sec:Motiv}

Looking back at the examples of nilspaces that we have treated so far, we see that while not all of them have a group structure, on all of them there is an \emph{action} by some filtered group $G$, and the cubes on the nilspace are the cubes in $\cu^n(G_\bullet)$ composed with this action. For instance, in the case of a quotient $G/\Gamma$ as in the last section, the action is defined by $(g,x\Gamma)\mapsto (gx)\Gamma$, and a map $\{0,1\}^n\to G/\Gamma$ is a cube if and only if it is of the form $v\mapsto \q(v) x\Gamma$ for some cube $\q\in \cu^n(G_\bullet)$ and some coset $x\Gamma$. 
We can also observe that, whether $\ns$ consists of a group or a coset space, we can always view $\ns$ algebraically as a type of principal bundle with fibre isomorphic to an abelian group. For instance, if $\ns$ consists of $G/\Gamma$ for a filtered group $(G,G_\bullet)$ of degree $d$, then the abelian group $\ab=G_d/(G_d\cap \Gamma)$ has a well-defined free action on $G/\Gamma$. Then, identifying points of $G/\Gamma$ lying in the same orbit of $\ab$ gives a projection (or bundle map) from $G/\Gamma$ to $(G/G_d)/((\Gamma\cdot G_d)/ G_d)$. This observation can be iterated, expressing $G/G_d$ in turn as the same type of bundle using the abelian group  $G_{d-1}/G_d$, and so on. \\
\indent In the next chapter we present a description of the above kind for a general nilspace. More precisely, on one hand we formalize the observation above using the notion of an \emph{abelian bundle} (Definition \ref{def:AbelianBundle}) and we present a central theorem from \cite{CamSzeg}  that describes a general nilspace $\ns$ as an iterated abelian bundle (Subsection \ref{subsec:bundecomp}). On the other hand, we shall define a certain non-trivial group $G$ acting on $\ns$, called the \emph{translation group}, which comes naturally equipped with a filtration $G_\bullet$, such that for each $n$ every cube in $\cu^n(G_\bullet)$ composed with this action is a cube in $\cu^n(\ns)$ (Subsection \ref{subsec:Facegp}). (Note however that, unlike for the nilspaces in this chapter, in general not every cube on $\ns$ will be of this form; see Remark \ref{rem:transcubes}.) To motivate this general description further, we close this chapter by establishing the simplest case of the description, namely, the following characterization of a 1-step nilspace.

\begin{proposition}\label{prop:abel}
Let $\ns$ be a $1$-step nilspace. Then there is an abelian group $\ab$ such that $\ns$ is a principal homogeneous space of $\ab$ and the cubes on $\ns$ are the projections of the degree-1 cubes on $\ab$.
\end{proposition}

\begin{proof}
Fix a point $e\in \ns$. The abelian group $\ab$ is the set $\ns$ together with the binary operation $+$ defined as follows: given $x,y\in \ns$, we let $x+y$ be the unique element $z\in \ns$ such that the map $\q:\{0,1\}^2\to X$ with values $\q(00)=e,\q(10)=x,\q(01)=y,\q(11)=z$ is in $\cu^2(\ns)$. We claim that this is a commutative group operation. To check commutativity, let $\theta \in \aut(\{0,1\}^2)$ be the symmetry permuting $01$ and $10$. Then by composition we have $\q\co \theta \in \cu^2(\ns)$, so by uniqueness of completion we must have $x+y=y+x$. For associativity, suppose that $x,y,z\in \ns$ and consider the following $3$-corner $\q':\{0,1\}^3\setminus \{1^3\}\to \ns$.
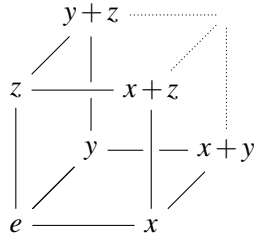
\[
\begin{tikzpicture}[%
  back line/.style={densely dotted},
  cross line/.style={preaction={draw=white, -,line width=6pt}}]
  \node (A) {$z$};
  \node [right of=A, node distance=1.8cm] (B) {$x+z$};
  \node [below of=A, node distance=1.8cm] (C) {$e$};
  \node [right of=C, node distance=1.8cm] (D) {$x$};
 
  \node (A1) [right of=A, above of=A, node distance=1cm] {$y+z$};
  \node [right of=A1, node distance=1.8cm] (B1) {};
  \node [below of=A1, node distance=1.8cm] (C1) {$y$};
  \node [right of=C1, node distance=1.8cm] (D1) {$x+y$};
  
  \draw (A) -- (A1) -- (C1) -- (C) -- (A);
  \draw (D1) -- (C1) -- (C) -- (D) -- (D1);
  \draw[back line] (D1) -- (B1) -- (A1);
  \draw[back line] (B) -- (B1);
  \draw[cross line] (A)  -- (B) -- (D);
\end{tikzpicture}
\]
This has a unique completion $\q$. Considering the morphisms $\phi_1:\{0,1\}^2\to \{0,1\}^3$, $v\mapsto (v\sbr{1},v\sbr{1},v\sbr{2})$ and $\phi_2:\{0,1\}^2\to \{0,1\}^3$, $v\mapsto (v\sbr{1},v\sbr{2},v\sbr{2})$, we then have
\[
(x+y)+z=\q\co \phi_1(11)=\q(111)=\q\co \phi_2(11)=x+(y+z).
\]
Finally, given $x$, the corner $\q'(00)=x$, $\q'(01)=\q'(10)=e$ has a completion $y$, which then satisfies $x+y=e$. This abelian group $\ab$, obtained by fixing a point $e$ in $\ns$, clearly has a free and transitive action on $\ns$, and the cubes on $\ns$ are precisely the compositions of degree-1 cubes on $\ab$ with this action.
\end{proof}
\noindent Henceforth, the expression `affine abelian group with the degree-1 structure' (or `degree-1 abelian torsor') will refer to the structure in the above result, that is, a principal homogeneous space of some abelian group $\ab$, with a nilspace structure given by the standard cubes on $\ab$ composed with the $\ab$-action on $e$.
\begin{remark}\label{rem:forget}
Note that if we start with an abelian group $\ab$ and then view it as a 1-step nilspace, then in doing so we forget which element is the  identity. In other words, to think of $\ab$ as a 1-step nilspace amounts to viewing it as an \emph{affine} abelian group. Fixing any element $e$ of the set underlying $\ab$, and applying the argument in the above proof to the 1-step nilspace structure, we recover the original group structure on $\ab$, but only up to a shift. This is an example of the forgetfulness of the functor from the category of filtered groups with polynomial maps to the category of nilspaces with morphisms, mentioned after Theorem \ref{thm:homs=polys}. Alternatively,  one could work in the category of \emph{rooted} nilspaces (i.e. nilspaces with a distinguished point). Such nilspaces are considered for instance in \cite{Szegedy:Regul,Szegedy:HFA}.
\end{remark}

\chapter[Algebraic characterization of nilspaces]{Algebraic characterization of nilspaces}\label{chap:chargenils}

\section{Some basic notions}\label{sec:BasNot}

In this section we collect several tools providing different ways to construct new cubes or new nilspaces out of old ones.

\subsection{Cubespaces, simplicial completion, and concatenation}
Let us recall the definition of a cubespace from Chapter \ref{chap:intro}.
\begin{defn}
A \emph{cubespace} is a set $\ns$ together with a collection of sets $\cu^n(\ns)\subset \ns^{\{0,1\}^n}$, $n\geq 0$,  satisfying the composition axiom from Definition \ref{def:ns} and such that $\cu^0(\ns)=\ns$.
\end{defn}

\begin{defn}\label{def:prodsubspace}
Given cubespaces $\ns,\nss$, their \emph{product} is the cubespace consisting of the cartesian product $\ns\times \nss$ with the cube sets $\cu^n(\ns\times \nss):=\cu^n(\ns)\times \cu^n(\nss)$ for each $n\geq 0$. We say that $\nss$ is a \emph{subcubespace} of $\ns$ if $\cu^n(\nss)\subset \cu^n(\ns)$ for every $n\geq 0$.
\end{defn}
Clearly, the product of two $k$-nilspaces is a $k$-nilspace.
\begin{defn}\label{def:ext-property} Let $P$ be a cubespace and $Q$ be a subcubespace of $P$. We say that $Q$ has the \emph{extension property} in $P$ if for every non-empty nilspace $\ns$ and every morphism $g': Q \to \ns$ there is a morphism $g: P\to\ns$ with $g|_Q=g'$. 
\end{defn}
\noindent If $S$ is a finite set and $h$ is a subset of $S$ then we denote by $\{0,1\}^S_h$ the set of elements of $\{0,1\}^S$ that are supported on $h$. Note that $\{0,1\}^S_h$ can be viewed as the discrete cube of dimension $| h|$.

\begin{defn} Let $S$ be a finite set and let $H$ be a set system in $S$ (a  collection of subsets of $S$). The collection of all cube morphisms $g:\{0,1\}^n\to \{0,1\}^S_h$, $n\geq 0,\, h\in H$,  defines a cubespace structure on $P=\bigcup_{h\in H}\{0,1\}^S_h$.  Cubespaces arising this way will be called \emph{simplicial cubespaces}.
\end{defn}
\noindent In this definition $P$ is a subcubespace of $\{0,1\}^S$.
Note that by composing such morphisms $g$ with discrete-cube morphisms, and invoking the composition axiom, we can deduce that for every $h$ in $H$, for every $h'\subset h$, every $n$-cube morphism $\phi: \{0,1\}^n \to \{0,1\}^S_{h'}$ is a cube in the above cubespace. It follows that we can assume without loss of generality that $H$ is a downset, or simplicial complex,  i.e. that it satisfies the following property: if $h\in H$ and $h'\subset h$ then $h'\in H$.

\begin{lemma}[Simplicial completion]\label{lem:simpcomp} Let $S$ be a finite set and $P$ be a simplicial cubespace corresponding to a set system $H$ in $S$. Then $P$ has the extension property in $\{0,1\}^S$.
\end{lemma}

\begin{proof} As noted above, we may assume that $H$ is a simplicial complex, so in particular the empty set is in $H$. If $H$ is the whole power set $2^S$ then there is nothing to prove. Otherwise, starting from some set $S\not\in H$ we can go downward in the lattice $2^S$ (passing to subsets) until we find a set $h'$ that is not in $H$ and such that every proper subset of $h'$ is contained in $H$. The new system $H'=H\cup\{h'\}$ is also a simplicial complex. Let $g:\bigcup_{h\in H}\{0,1\}_h^S\to \ns$ be a morphism into some nilspace $\ns$. The restriction of $g$ to $\{0,1\}^{h'}\setminus \{1^{h'}\}$ is an $|h'|$-corner. By the corner-completion axiom, we can therefore extend $g$ to a morphism $\bigcup_ {h\in H'}\{0,1\}^S_h\to \ns$. By repeating this procedure we can extend $g$ to all of $\{0,1\}^S$. 
\end{proof}
\noindent We shall now describe a basic way to form a new cube from two given cubes on a nilspace.

\begin{defn}[Adjacent maps and concatenations]\label{def:adj&concat}
Given a map $f$ defined on $\{0,1\}^n$, recall that for each $i\in\{0,1\}$ we define $f(\cdot,i)$ on $\{0,1\}^{n-1}$ by
$f(v,i)= f( (v,i) )$, where $(v,i):=\big(v\sbr{1},\ldots, v\sbr{n-1}, i\big)\in \{0,1\}^n$. A map $f_1$ on $\{0,1\}^n$ is said to be \emph{adjacent} to another such map $f_2$ if $f_1(\cdot, 1) = f_2(\cdot, 0)$, and we then write $f_1\prec f_2$. We say that $f$ is a \emph{concatenation} of $f_1,f_2,\ldots, f_\ell$ if $f(\cdot,0) = f_1(\cdot, 0)$, $f(\cdot, 1) = f_\ell(\cdot, 1)$, and $f_i\prec f_{i+1}$ for $i\in [\ell-1]$.
\end{defn}

\begin{lemma}\label{lem:concat}
Let $\ns$ be a nilspace and suppose that $\q_1,\q_2$ are cubes in $\cu^n(\ns)$ with $\q_1\prec \q_2$. Then the concatenation of $\q_1,\q_2$ is also in $\cu^n(\ns)$.
\end{lemma}

\begin{proof} Let $S=[n+1]$, $h_1=[n]$, $h_2=[n+1]\setminus\{n\}$ and $H=\{h_1,h_2\}$. Let $P$ be the simplicial cubespace corresponding to $H$. For $i=1,2$, let $\phi_i$ be an invertible morphism $\{0,1\}^{n+1}_{h_i}\to \{0,1\}^n$ such that the maps $\q_i\co \phi_i$ agree on $\{0,1\}^{n+1}_{h_1}\cap \{0,1\}^{n+1}_{h_2}=\{0,1\}^{n+1}_{[n-1]}$. (We can take $\phi_1:(v\sbr{1},\dots,v\sbr{n},0)\mapsto (v\sbr{1},\dots,v\sbr{n-1},1-v\sbr{n})$ and $\phi_2:(v\sbr{1},\dots,v\sbr{n-1},0,v\sbr{n+1})\mapsto (v\sbr{1},\dots,v\sbr{n-1},v\sbr{n+1})$.) We can then define $f:P\to \ns$ by $f|_{\{0,1\}^{n+1}_{h_i}}=\q_i\co \phi_i$. Then $f$ is a morphism, so by Lemma \ref{lem:simpcomp} it extends to a morphism $f':\{0,1\}^{n+1}\to \ns$. Consider the morphism $\phi:\{0,1\}^n\to \{0,1\}^{n+1},\,v\mapsto \big(v\sbr{1},\ldots,v\sbr{n-1},1-v\sbr{n},v\sbr{n}\big)$, and note that the concatenation of $\q_1,\q_2$ is $f'\co\phi$. By the composition axiom, the result follows.
\end{proof}
\noindent Sometimes we shall have to consider spaces $\ns$ for which the ergodicity axiom is not satisfied. As a first application of concatenations, let us show that such a space can always be partitioned into components each of which satisfies the ergodicity axiom.
\begin{lemma}
Suppose that $\ns$ satisfies the axioms in Definition \ref{def:ns} with ergodicity replaced by the weaker condition $\cu^0(\ns)=\ns$. Then $\ns$ can be decomposed into a disjoint union of nilspaces.
\end{lemma}
\begin{proof}
We define a relation on $\ns$ by writing $x\sim y$ if the map $\q :\{0,1\}\to \ns$ with $\q(0)=x$ and $\q(1)=y$ is in $\cu^1(\ns)$. This is an equivalence relation, indeed reflexivity follows from $\cu^0(\ns)=\ns$ and composition with $\mathbf{0}$, symmetry follows also from composition, and transitivity follows from Lemma \ref{lem:concat}. On any equivalence class $\ns'$ the cubes in $\bigcup_n \cu^n(\ns)$ with values in $\ns'$ form an ergodic nilspace. (Note that composition implies that the completion of an $\ns'$-valued corner is also $\ns'$-valued.)
\end{proof}
\noindent Note that if $G$ is a group and a map $g:\{0,1\}^n\to G$ is the concatenation of two other such maps $g_1\prec g_2$, then by \eqref{eq:recurGray} we have the following equality (useful in particular when $G$ is abelian):
\begin{equation}\label{eq:concat-sgn}
\sigma_n(g) = \sigma_n (g_2) \; \sigma_n(g_1).
\end{equation}
We end the subsection with a basic lemma that can be helpful to reduce questions about injective morphisms to ones about face maps. For a morphism $\phi:\{0,1\}^m\to \{0,1\}^n$, recall from \eqref{eq:J-set} that $J=J_\phi\subset [n]$ is the set of indices of coordinates of $\phi(v)$ that genuinely depend on one of the coordinates of $v\in \{0,1\}^m$. We know that $\phi$ is injective if and only if $|J|\geq m$. 
\begin{lemma}\label{lem:inj-morph-decomp}
Let $\phi:\{0,1\}^m\to\{0,1\}^n$ be an injective morphism with $|J_\phi|>m$. Then there is an automorphism $\theta$ of $\{0,1\}^m$ such that $\phi\co\theta$ is the concatenation of at most four injective morphisms $\phi_i$ with $|J_{\phi_i}| < |J_\phi|$ for all $i$.
\end{lemma}
\begin{proof}
We have $|J_\phi(i)|\geq 2$ for some $i$. Let $\theta$ be the automorphism that transposes coordinates $i$ and $m$. Relabelling $\phi\co\theta$ as $\phi$ we can now assume that $i=m$. We distinguish two cases. In the first case, both maps $\tau_0,\tau_1$ (recall the paragraph before Lemma \ref{lem:cube-morph-char}) occur among the varying coordinates $\phi(v)\sbr{j}$, $j\in J_\phi(m)$. In the second case only one such map occurs.

For the first case, define $\phi_1$ by changing in $\phi(v)|_{J_\phi(m)}$ all maps $\tau_0$ to $\mathbf{0}$, and define $\phi_2$ by changing in $\phi(v)|_{J_\phi(m)}$ all maps $\tau_1$ to $\mathbf{0}$. Then we have $|J_{\phi_1}|, |J_{\phi_2}| < |J_\phi|$, and $\phi$ is the concatenation of $\phi_1\prec \phi_2$.

For the second case, suppose that only $\tau_0$ occurs. Fix any $j\in J_\phi(m)$, then define $\phi_1$ from $\phi$ by switching $\tau_0$ to $\mathbf{0}$ at index $j$, define $\phi_2$ from $\phi$ by switching $\tau_0$ to $\tau_1$ at every index in $J_\phi(m)\setminus \{j\}$, and define $\phi_3$ from $\phi$ by switching $\tau_0$ to $\mathbf{1}$ at $j$. Observe that $\phi$ is the concatenation of $\phi_1\prec \phi_2 \prec \phi_3$. We also have $|J_{\phi_1}|, |J_{\phi_3}|$ both less than $|J_\phi|$ and, although $|J_{\phi_2}|=|J_\phi |$, the first case above applies to $\phi_2$, so we are done. (If only $\tau_1$ occurs, a similar argument works.)
\end{proof}

\subsection{Nilspaces and parallelepiped structures}\label{subsec:altnsdef}

Before we continue the development of the theory, in this subsection we pause to illustrate the use of some of the tools from the previous subsection. We shall  apply these tools to  treat a natural question concerning nilspaces, namely the question of the relation between the nilspace axioms from \cite{CamSzeg} and the axioms from \cite{HKparas} defining parallelepiped structures. In the introduction, we mentioned that the former axioms are a variant of the latter, and here we shall detail this remark.\medskip

\noindent In \cite{HKparas} the notion of an abstract parallelepiped structure was introduced. It was defined explicitly only for parallelepipeds of dimension 2 and 3 but, as noted at the end of \cite{HKparas}, the generalization to higher  dimensions is clear. We formulate this generalization below using notation from previous sections. Recall in particular the notation $\arr{\q_0,\q_1}_k$ for $k$-arrows from Definition \ref{def:arrow}, the notation $\phi_F$ for the face-restriction maps from Definition \ref{def:face-res}, and the notation $\aut(\{0,1\}^n)$ for the group of automorphisms of the discrete cube,   from Definition \ref{def:D-cubes}.

\begin{defn}[Parallelepiped structures]\label{def:paras}
Let $\ns$ be a set and let $n$ be a positive integer. We say that a finite sequence $(P_m)_{m=1}^n$ of sets $P_m\subset \ns^{\{0,1\}^m}$ is an \emph{$n$-parallelepiped structure} on $\ns$ if $P_1=\ns^{\{0,1\}}$ and  the following axioms are satisfied for every $m\in [2,n]$:
\begin{enumerate}[leftmargin=0.8cm]
\item (Face restrictions)\quad  For every $p\in P_m$ and every $(m-1)$-face $F\subset \{0,1\}^m$, we have $p\co \phi_F\in P_{m-1}$.\\\vspace{-0.7cm}

\item (Symmetries)\quad For every $p\in P_m$ and every $\theta\in \aut(\{0,1\}^m)$, we have $p\co \theta\in P_m$.\\
\vspace{-0.7cm}

\item (Equivalence relation)\quad The relation $\approx_{P_m}$ on $P_{m-1}$,  defined by $p\approx_{P_m} p'$ if and only if $\langle p,p'\rangle_1\in P_m$, is an equivalence relation.\\
\vspace{-0.7cm}

\item (Closing property)\quad Let $p' :\{0,1\}^m\setminus \{1^m\}\to \ns$ be such that for every $(m-1)$-face $F$ containing $0^m$ we have $p'\co\phi_F \in P_{m-1}(\ns)$. Then there exists $p\in P_m(\ns)$ such that $p(v)=p'(v)$ for all $v\neq 1^m$.
\end{enumerate}
\end{defn}
\noindent These axioms are the natural generalizations to dimension $n$  of the four axioms in \cite[Definition 4]{HKparas}. Note that the fourth axiom above is just the corner-completion axiom from Definition \ref{def:ns} stated for the elements of $P_m$.

To relate parallelepiped structures to nilspaces, we want to show that the sets $P_m$ above can be viewed as cube sets forming a nilspace structure on $\ns$. However, note that on a nilspace $\ns$ there is a cube set $\cu^n(\ns)$ defined for \emph{every} $n\in \N$, whereas an $n$-parallelepiped structure defines parallelepipeds only of dimension up to $n$. Nevertheless, we can consider what we shall call a \emph{sequence of parallelepiped structures} on a set $\ns$,  namely an infinite sequence $(P_m)_{m\in \N}$ such that for every $n\in\N$ the  initial segment $(P_m)_{m\in [n]}$ is an $n$-parallelepiped structure on $\ns$. We then have the following result.

\begin{proposition}\label{prop:nsparaequiv}
Let $\ns$ be a set and for each $n\in \N$ let $P_n\subset \ns^{\{0,1\}^n}$. Then $\big(\ns, (P_n)_{n\in \N}\big)$ is a nilspace if and only if $(P_n)_{n\in \N}$ is a sequence of parallelepiped structures on $\ns$. 
\end{proposition}

\begin{proof}
To see the forward implication, note first that the composition axiom for the nilspace immediately implies axioms $(i)$ and $(ii)$ above for each $m$, and trivially implies also axiom $(iv)$ since we are assuming the completion axiom. Moreover, axiom $(iii)$ above is also satisfied, indeed the reflexivity and symmetry of the relation $\approx_{P_m}$ are again direct consequences of the composition axiom, and the transitivity is a consequence of the concatenation property (Lemma \ref{lem:concat}).

To see the backward implication, the only difficulty is in checking that the composition axiom holds, that is, if $p\in P_n$ and $\phi:\{0,1\}^m\to\{0,1\}^n$ is a morphism, then $p\co \phi\in P_m$. We check this in several steps as follows. 

First let us suppose that $\phi$ is an $m$-face map, with image an $m$-face $F\subset \{0,1\}^n$. Note that we may then assume that $\phi$ is just the restriction map $\phi_F$, by axiom $(ii)$ above. We now prove that $p\co\phi_F\in P_m$ by induction on $n-m$, as follows. If $n=m+1$ then we have $p\co\phi_F\in P_m$ immediately by axiom $(i)$. If $n>m+1$ then let $F'$ be an $(n-1)$-face in $\{0,1\}^n$ that contains $F$, and note that by axiom $(i)$ we have $p\co \phi_{F'}\in P_{n-1}$. Since $\phi_{F'}$ is injective we can take its inverse map $\phi_{F'}^{-1}:F'\to \{0,1\}^{n-1}$. Note that $\phi_{F'}^{-1}\co \phi_F$ is a face map $\{0,1\}^m\to\{0,1\}^{n-1}$. Hence, by induction, we have that $P_m\ni (p\co \phi_{F'})\co(\phi_{F'}^{-1}\co \phi_F) = p\co \phi_F$, as required.

Our next step is to check that $p\co \phi\in P_m$ if $\phi$ is an injective morphism. Thus we still have $m\leq n$, we assume that $|J_\phi(i)|\geq 1$ for every $i\in [m]$, and we recall that we have equality here for every $i$ if and only if $\phi$ is an $m$-face map. We shall prove that $p\co\phi\in P_m$ by induction on $|J_\phi|=\sum_{i\in [m]} |J_\phi(i)|$. The basic case $|J_\phi|=m$ is clear by the previous paragraph. Suppose then that $|J_\phi|> m$. By Lemma \ref{lem:inj-morph-decomp} there is $\theta\in \aut(\{0,1\}^m)$ and morphisms $\phi_1,\ldots,\phi_k:\{0,1\}^m\to\{0,1\}^n$, $k\leq 4$, such that $\phi\co\theta$ is the concatenation of $\phi_1\prec\cdots\prec\phi_k$ and $|J_{\phi_i}|<|J_\phi|$ for each $i$. We have by axiom $(ii)$ above that $p\co\phi\in P_m$ if and only if $p\co\phi\co\theta\in P_m$, so for our purposes we may assume that $\theta$ is the identity. Thus $\phi$ is the concatenation of $\phi_1\prec\cdots\prec\phi_k$, and therefore $p\co\phi$ is the concatenation of $(p\co\phi_1)\prec\cdots\prec(p\co\phi_k)$. By induction, each map $p\co\phi_i$ is in $P_m$. On the other hand, if two maps in $P_m$ are adjacent, then by the transitivity of the relation $\approx_{P_m}$ their concatenation is also in $P_m$. Applying this to the maps $p\co\phi_i$ we deduce that $p\co\phi\in P_m$.
 
Finally, we check that $p\co\phi\in P_m$ for \emph{every} morphism $\phi:\{0,1\}^m\to \{0,1\}^n$. Let $\ell=\ell(\phi)$ denote the number of indices $i\in [m]$ such that $J_\phi(i)=\emptyset$. We shall prove that $p\co\phi\in P_m$ by induction  on $\ell$. The basic case $\ell=0$ is clear by the previous paragraph, since $\phi$ is then injective. If $\ell>0$, then $J_\phi(i)=\emptyset$ for some $i$, and note that we may assume that $i=m$ (using an automorphism of $\{0,1\}^m$ if necessary, as in the previous paragraph). But then letting $F$ denote the $(m-1)$-face $\{v_m=0\}\subset \{0,1\}^m$, we have $\ell(\phi\co\phi_F)<\ell(\phi)$, so by induction we have $p\co\phi\co\phi_F\in P_{m-1}$. Since $\phi(v)$ does not depend on $v_m$, we also have that $p\co\phi= \langle p\co\phi\co\phi_F,p\co\phi\co\phi_F\rangle_1$. Hence, by the reflexivity of $\approx_{P_m}$, it follows that $p\co\phi\in P_m$.
\end{proof}

\begin{remark}
In \cite{HKparas} Host and Kra use the term \emph{strong parallelepiped structure} to refer to an $n$-parallelepiped structure for which we also have that for each $p'$ satisfying the premise of the closing property, there is a \emph{unique} $p$ satisfying the conclusion. Parallelepiped structures without this uniqueness in the closing property are called \emph{weak} parallelepiped structures. In light of Proposition \ref{prop:nsparaequiv}, we have that $\big(\ns, (P_n)_{n\in \N}\big)$ is a $k$-step nilspace if and only if $(P_n)_{n\in \N}$ is a sequence of parallelepiped structures on $\ns$ and $(P_n)_{n=1}^{k+1}$ is strong.
\end{remark}

\noindent As we mentioned in the introduction, in \cite{HKparas} a relation was established between parallelepiped structures and nilpotent groups. Indeed, Host and Kra showed that for every strong parallelepiped structure of dimension 3 on a set $\ns$ there is a group acting on $\ns$ by bijections that preserve the parallelepipeds in a certain natural sense, and this group is 2-step nilpotent (see \cite[Proposition 12]{HKparas}). In light of Proposition \ref{prop:nsparaequiv}, this result of Host and Kra is applicable to any 2-step nilspace. One of the central ways in which the results in \cite{CamSzeg} continue the program started in \cite{HKparas} is by providing a generalization, for higher-step nilspaces, of this relation to nilpotent groups. This generalization is detailed in Subsection \ref{subsec:Facegp} below, where we treat the \emph{translation group} of a nilspace.

\subsection{Tricubes and the tricube composition}\label{subsec:tricube}

Here we shall describe another operation that gives a very useful construction of new cubes on a nilspace. To this end we shall use special examples of  cubespaces, called tricubes. These objects play an important role in the theory.  Their associated operation, called the tricube composition, can be viewed in a certain sense as an analogue, for cubes on nilspaces, of the convolution operation for functions.

\begin{defn}\label{def:tricube}
The \emph{tricube} of dimension $n$ is the set $T_n=\{-1,0,1\}^n$ together with the following cubespace structure. First we define, for each $v\in\{0,1\}^n$, the map
\begin{equation}\label{eq:cubemb}
\trem_v:\{0,1\}^n\to T_n,\quad\trem_v\big(w\sbr{1},w\sbr{2},\dots,w\sbr{n}\big)\sbr{j}=(2v\sbr{j}-1)(1-w\sbr{j}).
\end{equation}
Thus $\trem_v$ is an injective map embedding $\{0,1\}^n$ into $T_n$ as the subcube with base point $\trem_v(0^n)$ equal to the `outer point' $(2v\sbr{1}-1,2v\sbr{2}-1,\ldots,2v\sbr{n}-1)$ of $T_n$, and with $\trem_v(1^n) = 0^n$. A map $\q:\{0,1\}^m \to T_n$ is declared to be a cube if there exists a morphism $\phi: \{0,1\}^m\to \{0,1\}^n$ and some element $v\in \{0,1\}^n$ such that $\q = \trem_v\co \phi$.
\end{defn}
\noindent Thus an $m$-cube on $T_n$ is just a map $\{0,1\}^m \to T_n$ that factors as a cube-morphism through one of the embeddings $\trem_v$. Note that this cubespace structure on $T_n$ is the smallest one such that all the maps $\trem_v$ are cubes. Note also the following basic fact.
\begin{lemma}
For each $n\in \N$, the cubespace $T_n$ is the product cubespace $(T_1)^n$.
\end{lemma}
Given two maps $f_1,f_2: S\to \ns$ we write $f_1\times f_2$ for the map $S\to \ns\times \ns$, $s\mapsto ( f_1(s),f_2(s) )$.

\begin{proof}
Let $k_0,k_1 \in \N$, let $n=k_0+k_1$, and note that for every $v_0\in \{0,1\}^{k_0}, v_1\in\{0,1\}^{k_1}$ we have from \eqref{eq:cubemb} that $\trem_{v_0}\times \trem_{v_1}(w)=\trem_{(v_0,v_1)}(w)$, $w\in \{0,1\}^n$. Every map $\q:\{0,1\}^m\to T_n$ can be written $\q_0\times \q_1$ for $\q_i:\{0,1\}^m\to T_{k_i}$. Then we have $\q\in \cu^m(T_n)$ if and only if for every $v=(v_0,v_1)\in \{0,1\}^n$ there is a morphism $\phi:\{0,1\}^m\to \{0,1\}^n$ such that $\q=\trem_v\co\phi$. But $\phi$ can also be written $\phi_0\times \phi_1$ and is a morphism if and only if each $\phi_i$ is, whence $\q$ is an $m$-cube on $T_n$ if and only if each map $\q_i\in \cu^m(T_{k_i})$.
\end{proof}

\noindent To form a new cube on a nilspace $\ns$ using $T_n$, we shall typically take $2^n$ cubes $\q_v\in \cu^n(\ns)$, one for each $v\in \{0,1\}^n$, and take  the morphism $\phi: T_n \to \ns$ defined as follows: for each $u\in T_n$, we take any embedding $\trem_v$ containing $u$ in its image, and then we let $\phi(u) = \q_v(\trem_v^{-1}(u))$. Note that this is a well-defined map provided that the cubes $\q_v$ can indeed  be ``glued together into $T_n$", that is for every $v,v'$ the maps $\q_v\co \trem_v^{-1}, \q_{v'}\co \trem_{v'}^{-1}$ agree on the intersection of the images of $\trem_v,\trem_{v'}$. 
Note that if $\phi$ is thus well-defined then it is also a cubespace morphism. Given all this, the new cube on $\ns$ is then obtained by taking the values of $\phi$ on the `outer points' of $T_n$, i.e. the elements of $\{-1,1\}^n$, via the following map.

\begin{defn}[Outer-point map]\label{def:ope}
We denote by $\omega_n$ the embedding $\{0,1\}^n\to T_n$ defined by $\omega_n(v)=\trem_v(0^n)=(2v\sbr{1}-1,\ldots,2v\sbr{n}-1)$.
\end{defn}
\noindent The new cube mentioned above is thus $\phi\co \omega_n$. That this map is indeed a cube is a nontrivial fact, which we record as follows. 

\begin{lemma}[Tricube composition]\label{lem:tricube-comp} Let $\ns$ be a nilspace and let $f:T_n\to \ns$ be a morphism. Then the composition $f \co \omega_n$ is in $\cu^n(\ns)$.
\end{lemma}

\noindent The point of this result lies in that $\omega_n$ is not a cube on $T_n$ (if it were a cube then the result would be trivial, since $f$ is a morphism). The idea of the proof is that, while the cube structure on $T_n$ is not rich enough for $\omega_n$ to be in $\cu^n(T_n)$, there is an injective morphism $q:T_n\to \{0,1\}^{2n}$ such that, firstly, $q\co \omega_n$ is a morphism of discrete cubes, and secondly $q(T_n)$ has the extension property, so that $f\co q^{-1}$ extends to a cube in $\cu^{2n}(\ns)$, and therefore $f\co\omega_n = f\co q^{-1}\co q \co \omega_n$ is a cube on $\ns$. Let us detail this further.

First we describe the embedding of $T_n$, which will also be useful later on.

\begin{lemma}\label{lem:tricube-emb}
Let $\lambda:\{-1,0,1\}\to \{0,1\}^2$ be the function $\lambda(1)=(1,0)$, $\lambda(0)=(0,0)$, $\lambda(-1)=(0,1)$. Then $q=\lambda^n:v\mapsto \big(\lambda(v\sbr{1}),\ldots,\lambda(v\sbr{n})\big)$ is an injective morphism  $T_n\to \{0,1\}^{2n}$. Moreover, $q(T_n)$ is simplicial in $\{0,1\}^{2n}$.
\end{lemma}

\begin{proof}
The map $\lambda$ is easily checked to be a morphism, and it is also easily checked that products of morphisms are morphisms relative to the product cubespace structure, whence $q$ is indeed a morphism. To see that $q(T_n)$ is simplicial, suppose that $w=q(v)\in q(T_n)$ and $u\in \{0,1\}^{2n}$ has $\supp u \subset \supp w$. Let $v'$ be obtained from $v$ by switching to 0 every entry $v\sbr{j}$ such that the indices corresponding to $\lambda(v\sbr{j})$ are not in the support of $u$, i.e. $\{2j-1,2j\}\not\subset \supp u$. Then $u=q(v')$.
\end{proof}

\begin{proof}[Proof of Lemma \ref{lem:tricube-comp}]
Since $q(T_n)$ is simplicial in $\{0,1\}^{2n}$ and $f\co q^{-1}:q(T_n)\to \ns$  is a morphism (relative to $q(T_n)$ as a subcubespace of $\{0,1\}^{2n}$), by Lemma \ref{lem:simpcomp} the map $f\co q^{-1}$ extends to a cube $\q:\{0,1\}^{2n}\to \ns$. Since $f\co\omega_n = \q\co (q \co \omega_n)$, it now suffices to show that $q \co \omega_n: \{0,1\}^n \to \{0,1\}^{2n}$ is a morphism. To see this, note that $\lambda$ restricted to $\{-1,1\}$ equals the map $x\mapsto (\frac{1+x}{2},\frac{1-x}{2})$. For each $j\in [n]$, the $j$-th coordinate of $\omega_n(v)$ is $2v\sbr{j}-1$, whence $\lambda (\omega_n(v)\sbr{j})=(v\sbr{j},1-v\sbr{j})=(\tau_0(v\sbr{j}),\tau_1(v\sbr{j}))$, where $\tau_0,\tau_1$ are the identity and reflection introduced before Lemma \ref{lem:cube-morph-char}. It follows that $q\co \omega_n$ is a morphism $\{0,1\}^n\to\{0,1\}^{2n}$, indeed we have
\[
q\co \omega_n (v)= \Big(\, \big(\tau_0(v\sbr{1}),\tau_1(v\sbr{1})\big)\,,\,\big(\tau_0(v\sbr{2}),\tau_1(v\sbr{2})\big)\,,\ldots, \,\big(\tau_0(v\sbr{n}),\tau_1(v\sbr{n})\big)\,\Big). \qedhere
\]
\end{proof}
 
\begin{remark}
Given Lemma \ref{lem:tricube-emb}, we may give an alternative definition of tricubes. More precisely, if we identify $\{0,1\}^{2n}$ with $(\{0,1\}^2)^n$ the natural way, then thanks to the embedding $q$ we may redefine $T_n$ as the subcubespace of $\{0,1\}^{2n}$ of the form $\{(1,0), (0,0), (0,1)\}^n$. This alternative definition would be natural in the sense that the cubespace structure on the tricube is easily described as the restriction of the one on $\{0,1\}^{2n}$. In particular, the proof of Lemma \ref{lem:tricube-comp} itself relies essentially on this identification of $T_n$ with a simplicial subcubespace of $\{0,1\}^{2n}$. However, the cube structure on $T_n$ as described in Definition \ref{def:tricube} is also intuitively clear and, in addition, we shall make much more use in the sequel of the intuition and notation related to viewing $T_n$ as $\{-1,0,1\}^n$. There may be one way in which this view could be misleading, namely, that one could think that morphisms $T_n\to \ns$ extend automatically to morphisms from more general subsets of $\Z^n$ into $\ns$; let us emphasize that this is not the case.
\end{remark} 
 
\subsection{Arrow spaces}

Recall from Definition \ref{def:arrow} that for two maps $f_0,f_1:\{0,1\}^n\to \ns$, the corresponding $k$-arrow is the map
\[
\arr{f_0,f_1}_k:\{0,1\}^{n+k} \to \ns, \;\;(v,w)\;\mapsto \; \left\{\begin{array}{lr} f_0(v),& w\neq 1^k \\
  f_1(v), & w=1^k\end{array}\right..
\]
Note that taking the $1$-arrow is an invertible operation $\ns^{\{0,1\}^n}\times \ns^{\{0,1\}^n}\to \ns^{\{0,1\}^{n+1}}$, $f_0\times f_1\mapsto \arr{f_0,f_1}_1$, with inverse the map $f\mapsto f(\cdot,0)\times f(\cdot,1)$, which we denote by $\ia_n$.
\begin{defn}[Arrow spaces]\label{def:arrowspace}
Let $\ns$ be a cubespace. For each positive integer $k$, the \emph{$k$-th arrow space} over $\ns$, denoted by $\ns\Join_k \ns$, is the cubespace defined on the cartesian product $\ns\times \ns$ by letting $\q \in \cu^n(\ns\Join_k \ns)$ if and only if $\q=\q_0\times \q_1$ with $\arr{\q_0,\q_1}_k\in \cu^{n+k}(\ns)$.
\end{defn}
\begin{example}
Let us illustrate this construction in the case of filtered groups. If $\ns$ consists of a filtered group $(G,G_\bullet)$ with cube sets $\cu^n(G_\bullet)$, $n\geq 0$, then it follows from Lemma \ref{lem:filtarrow} that a product map $\q_0\times \q_1$ is in $\cu^n(\ns\Join_k\ns)$ if and only if $\q_0,\q_1$ are both in $\cu^n(G_\bullet)$ and the map $v\mapsto \q_0(v)^{-1}\q_1(v)$ is in $\cu^n(G_\bullet^{+k})$ (here recall that given a filtration $G_\bullet$ we denote by $G_\bullet^{+k}$ the shifted filtration, which has $i$-th term $G_{i+k}$). In particular, if $(G,G_\bullet)$ is abelian with the lower central series, then $\q_0\times \q_1$ is in $\cu^n(\ns\Join_1\ns)$ if and only if $\q_0$ is a standard $n$-cube on $G$ and $\q_1$ is a translate of $\q_0$, i.e. for some $h\in G$ we have $\q_1(v)=\q_0(v)+h$ for all $v$. Still in this abelian case, for $k\geq 2$ we have $\q_0\times \q_1\in \cu^n(\ns\Join_k\ns)$ if and only if $\q_1=\q_0$ is a standard $n$-cube on $G$.
\end{example}
\noindent Arrow spaces with $k>1$ will be used starting in Subsection \ref{subsec:Facegp},  for now we use mainly the case $k=1$.

\noindent Note that every cube $\q\in \cu^{n+1}(\ns)$ is a $1$-arrow of $n$-cubes, indeed $\q=\arr{\q(\cdot,0),\q(\cdot,1)}_1$, where each map $\q(\cdot,i)$ is in $\cu^n(\ns)$ by the composition axiom. Therefore, we have
\begin{equation}\label{eq:arrowcubechar}
\cu^n(\ns\Join_1 \ns)=\ia_n(\cu^{n+1}(\ns)) \subset \cu^n(\ns)\times\cu^n(\ns).
\end{equation}
For $k>1$, not every cube in $\cu^{n+k}(\ns)$ is a $k$-arrow of $n$-cubes. Denote by $\textrm{Ar}_k^n(\ns)$ the subset of $\ns^{\{0,1\}^{n+k}}$ that is the image of the set of maps $f_0\times f_1$, $f_i\in \ns^{\{0,1\}^n}$ under $f_0\times f_1\mapsto \arr{f_0,f_1}_k$. Then the latter map is clearly invertible on $\textrm{Ar}_k^n(\ns)$, with inverse denoted by $\ia_{n,k}$ (thus $\ia_{n,1}=\ia_n$). Then we have
\begin{equation}\label{eq:itharrowcubechar}
\cu^n(\ns\Join_k \ns)=\ia_{n,k}(\cu^{n+k}(\ns) \cap \textrm{Ar}_k^n(\ns)) \subset \cu^n(\ns)\times\cu^n(\ns).
\end{equation}

\begin{lemma}\label{lem:arrow}
Let $\ns$ be a $k$-step nilspace and let $i\geq 1$. Then $\ns \Join_i \ns$ is a $k$-step nilspace, not necessarily ergodic. If $\ns$ is $\ell$-fold ergodic then $\ns\Join_i \ns$ is $(\ell-i)$-fold ergodic.
\end{lemma}

\begin{proof}
Let $\nss=\ns\Join_i \ns$. We first check the composition axiom. If $\q_0\times \q_1$ is an $n$-cube on $\nss$ and $\phi:\{0,1\}^m\to \{0,1\}^n$ is a morphism, then $(\q_0\times \q_1)\co \phi = (\q_0\co \phi) \times (\q_1\co \phi)$ satisfies $\arr{\q_0\co \phi, \q_1\co \phi }_i= \arr{ \q_0,\q_1}_i\co \phi':\{0,1\}^{m+i}\to \ns$, where $\phi':\{0,1\}^{m+i}\to \{0,1\}^{n+i}$, $(v,w)\mapsto (\phi(v),w)$ is a morphism. Hence composition for $\cu^n(\nss)$ follows from composition for $\cu^{n+i}(\ns)$.

We now check the completion axiom. Let $\q':\{0,1\}^n\setminus\{1^n\}\to \nss$ be an $n$-corner, and let $\q_0',\q_1':\{0,1\}^n\setminus\{1^n\}\to \ns$ be the maps such that $\q'=\q_0'\times\q_1'$. Then by assumption for each $(n-1)$-dimensional lower-face $F\subset \{0,1\}^n$ we have $\q'\co \phi_F =(\q_0'\co\phi_F)\times (\q_1'\co\phi_F)\in \cu^{n-1}(\nss)$, that is $\arr{\q_0'\co\phi_F,\q_1'\co\phi_F}_i\in \cu^{n-1+i}(\ns)$. In particular $\q_0',\q_1'$ are both $(n-1)$-corners on $\ns$. Let $\q_0$ be a completion of $\q_0'$. Let $\q'':\{0,1\}^{n+i}\setminus \{1^{n+i}\}\to \ns$ be the map defined by $\q''(v,w)= \q_0(v)$ if $w\neq 1^i$, and $\q''(v,1^i)= \q_1'(v)$ for $v\neq 1^n$. We claim that $\q''$ is an $(n+i)$-corner on $\ns$. To see  this, fix any $j\in [n+i]$ and consider the $(n+i-1)$-dimensional lower face $F_j'=\{u \in \{0,1\}^{n+i}: u\sbr{j}=0\}$. If $j\in [n]$, then $\q''\co \phi_{F'_j}=\arr{\q_0'\co\phi_F,\q_1'\co\phi_F}_i$, where $F$ is the $(n-1)$-dimensional lower face of $\{0,1\}^n$ obtained by deleting the last $i$ coordinates from elements of $F'_j$. Hence $\q''\co \phi_{F'_j}\in \cu^{n-1+i}(\ns)$ for such values of $j$. If $j\in [n+1,n+i]$, then $\q''\co \phi_{F'_j}(v,w)=\q_0(v)$ for all $w\neq 1^{i-1}$. But this equals the cube $\q_0$ composed with the morphism $(v,w)\mapsto v$, so we also have $\q''\co \phi_{F'_j}\in \cu^{n-1+i}(\ns)$ in this case. Thus $\q''$ is indeed an $(n+i)$-corner on $\ns$ and so it can be completed to a $(n+i)$-cube, which by construction has the form $\arr{\q_0,\q_1}_i$ for some $n$-cube $\q_1$ completing $\q_1'$. Hence $\q_0\times \q_1$ completes $\q'$.\\
\indent Note that if $\ns$ is $k$-step then the above $n$-cubes $\q_i$ are the unique completions of $\q_i'$ for $i=0,1$ respectively, and therefore the completion of $\q'$ is unique, so that $\nss$ is also $k$-step.

The last claim in the lemma is clear, for if every map $\{0,1\}^\ell\to \ns$ is a cube, then in particular $\arr{ f_0,f_1}_i\in \cu^\ell(\ns)$ for every $f_0,f_1:\{0,1\}^{\ell-i}\to \ns$.
\end{proof}
\noindent The last construction that we define in this section is roughly speaking a means to `differentiate' the cube structure on $\ns$ with respect to a fixed point $x\in \ns$, in such a way that from a $k$-step nilspace we obtain a $(k-1)$-step nilspace.

\begin{defn}\label{def:partial}
Let $\ns$ be a nilspace and let $x\in \ns$. We define the nilspace $\partial_x \ns$ to be the set $\ns$ with the cubespace structure inherited by embedding $\ns$ in $\ns \Join_1 \ns$ via the map $y\mapsto (x,y)$. Thus we declare $\q:\{0,1\}^n\to \ns$ to be in  $\cu^n( \partial_x \ns)$ if the map\footnote{Here we denote by $x$ the map on $\{0,1\}^n$ with constant value $x$.} $x \times \q:\{0,1\}^n\to \ns\times \ns$, $v\mapsto (x,\q(v))$ is in $\cu^n(\ns \Join_1 \ns)$, that is if the 1-arrow $\arr{x,\q}_1$ is in $\cu^{n+1}(\ns)$.
\end{defn}
The following lemma relates higher iterations of $\partial_x$ to higher-level arrow spaces.
\begin{lemma}
We have $\q\in \cu^n(\partial_x^i \ns)$ if and only if $x\times \q\in \cu^n(\ns \Join_i \ns)$.
\end{lemma}
\begin{proof}
This follows from the fact that if we iterate $i$ times the operation of taking a 1-arrow $\q\mapsto \arr{x,\q}_1$ then we obtain the $i$-arrow $ \arr{x,\q}_i$.
\end{proof}
\noindent We now show that $\partial_x(\ns)$ is a $(k-1)$-step nilspace, rather than $k$-step like $\ns \Join_1 \ns$. The reason for the decrease in the step is that, while completing a $k$-corner on $\ns \Join_1 \ns$ amounted to completing a $(k+1)$-cube minus a 1-face on $\ns$, completing a $k$-corner on $\partial_x(\ns)$ amounts to completing a full $(k+1)$-corner on $\ns$.

\begin{lemma}\label{lem:nilspace-dif}
Let $\ns$ be a $k$-step nilspace and let $x\in \ns$. Then $\partial_x \ns$ is a $(k-1)$-step nilspace, not necessarily ergodic. If $\ns$ is $k$-fold ergodic, then $\partial_x \ns$ is $(k-1)$-fold ergodic.
\end{lemma}

\begin{proof}
Composition and completion are clearly inherited from $\ns\Join_1 \ns$. Concerning completion, note that here an $n$-corner in $\partial_x \ns$ is to be completed to an $n$-cube $\q$ in $\cu^n(\partial_x \ns)$ rather than just in $\cu^n(\ns\Join_1 \ns)$, so $\q$ must be of the form $x\times \q_1$, and is therefore the completion of the $(n+1)$-corner on $\ns$ that equals the constant function $x$ on the lower-face $v\sbr{n+1}=0$ and equals some $n$-corner on $\{v: v\sbr{n+1}=1,\, v\neq 1^{n+1}\}$. Hence this completion is unique for $n=k$ as claimed.
\end{proof}

\begin{example}
Let $\ns$ be a group $G$ with a filtration $G_\bullet$ and the associated cube structure $(\cu^n(G_\bullet))_{n\geq 0}$. Then for $x=\id_G$, using \eqref{eq:cubefactn} we see that $\partial_x \ns$ equals $G$ with the filtration $G_\bullet^{+1}$ and its associated cubes. Note that if $G_0=G_1=G$ but $G_2$ is a proper subgroup of $G$, then while $\ns$ is ergodic, the nilspace $\partial_x \ns$ is not.
\end{example}

\section{$k$-step nilspaces as abelian bundles of degree $k$ acted upon by  filtered groups}\label{sec:bundledecomp}

In this section we work towards a general decomposition theorem for nilspaces, which will be eventually obtained as Theorem \ref{thm:bundle-decomp}. 

To begin with, we describe another way to get a new cubespace from a given one.

\begin{defn}[Quotient cubespace]
Let $\ns$ be a cubespace and let $\sim$ be an equivalence relation on $\ns$. 
The \emph{quotient cubespace} is the quotient set $\nss:= \ns/\sim$ together with the cubes $\pi\co \q$, for every $\q\in \cu^n(\ns)$, $n\geq 0$, where $\pi: \ns\to \nss$ is the canonical projection.
\end{defn}

\begin{defn}[Factor of a nilspace]\label{def:factor}
Let $\ns$ be a nilspace and let $\sim$ be an equivalence relation on $\ns$. If the quotient cubespace $\ns/\sim$ is itself a nilspace, then it is called a \emph{factor} of $\ns$.
\end{defn}

\noindent In the general decomposition of $\ns$ that we shall prove, the building blocks will be special factors of $\ns$ called \emph{characteristic factors}. These are obtained as quotients of $\ns$ under certain equivalence relations, which are defined in the following subsection.

\subsection{Characteristic factors}
The following relations on a nilspace are crucial for the decomposition theorem. They can be viewed as analogues for nilspaces of the  \emph{regionally proximal relations of order} $k$ for a dynamical system, which were defined in \cite[\S 3.3]{HKM} (see also \cite{HKparas,HM}).

\begin{defn}\label{def:simdef} Let $\ns$ be a nilspace. For each positive integer $k$ we denote by $\sim_k$ the relation on $\ns$ defined as follows:
\[
x\sim_k y\;\;\;\Leftrightarrow\;\;\; \exists\,\q_0,\q_1\in \cu^{k+1}(\ns)\;\textrm{ such that }\;\q_0(0^{k+1}) = x,\;\; \q_1(0^{k+1})=y,\;\;\q_0(v) = \q_1(v)\;\; \forall\, v\neq 0^{k+1}.
\]
\end{defn}
\noindent This relation is reflexive and symmetric. To prove that it is also transitive, we use the following result.

\begin{lemma}\label{lem:sim1} Two elements $x,y\in \ns$ satisfy $x\sim_k y$ if and only if the map $\q:\{0,1\}^{k+1}\to \ns$ sending $0^{k+1}$ to $y$ and all other elements $v$ to $x$ is a cube in $\cu^{k+1}(\ns)$.
\end{lemma}

\begin{proof}
The `if' direction is clear since every constant map $\{0,1\}^{k+1}\to \ns$ is a cube. For the converse, suppose that $\q_0,\q_1\in \cu^{k+1}(\ns)$ satisfy the condition in Definition \ref{def:simdef}, and let $\phi$ be the map from the tricube $T_{k+1}$ to $\{0,1\}^{k+1}$ defined by $\phi = f^{k+1}$, where $f(-1)=f(1)=1, f(0)=0$ (thus $\phi$ `reflects' each of the $2^{k+1}$ subcubes of $T_{k+1}$ onto $\{0,1\}^{k+1}$). Note that the function $\q_0\co\phi$ is a morphism $T_{k+1}\to \ns$, as explained in the previous section (we can represent this map as $T_{k+1}$ with a reflected copy of $\q_0$ in each of the $2^{k+1}$ sub-cubes). Now observe that, by our assumption on $\q_0,\q_1$, the map obtained from $\q_0\co\phi$ by changing the value at $1^{k+1}$ from $x$ to $y$ is still a morphism $T_{k+1}\to \ns$, which we denote by $g$. By Lemma  \ref{lem:tricube-comp}, we then have $g\co\omega_{k+1} \in \cu^{k+1}(\ns)$, and the result follows.
\end{proof}

\begin{corollary}
For every nilspace $\ns$ and every $k\in\N$ the relation $\sim_k$ is an equivalence relation.
\end{corollary}

\begin{proof} Suppose that $x,y,z\in \ns$ satisfy $x\sim_k y$ and $y\sim_k z$. By symmetry and Lemma \ref{lem:sim1}, there exist $\q_0,\q_1\in \cu^{k+1}(\ns)$ such that $\q_0(0^{k+1})=x,\q_1(0^{k+1})=z$ and $\q_0(v) = \q_1(v)=y$ for every $v\neq 0^{k+1}$. Therefore $x\sim_k z$.
\end{proof}
\begin{example}\label{ex:rels}
We pause to illustrate the relations $\sim_k$ for the main examples of nilspaces from Chapter \ref{chap:Centrexs}. Let $\ns$ consist of a filtered group $(G,G_\bullet)$ with cube sets $\cu^n(G_\bullet)$, $n\geq 0$. Using Lemma \ref{lem:sim1} and composition with an automorphism of $\{0,1\}^{k+1}$ sending $0^{k+1}$ to $1^{k+1}$, we have that $x,y\in \ns$ satisfy $x\sim_k y$ if and only if the map $\q: \{0,1\}^{k+1}\to G$ with $\q(v)=x$ for $v\neq 1^{k+1}$ and $\q(1^{k+1})=y$ is a cube in $\cu^{k+1}(G_\bullet)$. Using the factorization of $\q$ given by Lemma \ref{lem:cubefactn}, we find that the coefficients $g_i$ in this factorization must be as follows: $g_0=x$, then $g_j=\id_G$ for $0<j<2^{k+1}-1$, and then $g_{2^{k+1}-1}$ must   satisfy $x\, g_{2^{k+1}-1}=\q(1^{k+1})=y$. It follows that $x\sim_k y$ if and only if there exists $g\in G_{k+1}$ such that $y=g\, x$. Now if $\ns$ is a quotient $G/\Gamma$ with the coset nilspace structure described in Proposition \ref{prop:quotex}, then two cosets $x\Gamma, y\Gamma$ in $\ns$ satisfy $x\Gamma \sim_k y\Gamma$ if and only if there is a map $\gamma_0:\{0,1\}^{k+1}\to \Gamma$ and a cube $\q\in \cu^n(G_\bullet)$ such that $\q(v)=x\gamma_0(v)$ for $v\neq 1^{k+1}$ and $\q(1^{k+1})=y\gamma_0(1^{k+1})$. Since $\q':v\mapsto x^{-1}\q(v)$ is a cube and $\q'(v)=\gamma_0(v)$ for all $v\neq 1^{k+1}$, by the completion axiom applied on $(\Gamma,\Gamma_\bullet)$ we can find a cube $\gamma\in \cu^{k+1}(\Gamma_\bullet)$ such that $\gamma(v)=\gamma_0(v)$ for all $v\neq 1^{k+1}$. This can then be used to show that in $G$ we have $x\sim_k y \gamma_0( 1^{k+1}) \gamma^{-1}( 1^{k+1})$, and so there exists $g\in G_{k+1}$ such that $y^{-1}gx\in \Gamma$. It follows that the equivalence classes of $\sim_k$ are the orbits of the left action of $G_{k+1}$ on $\ns$.
\end{example}

\noindent Lemma \ref{lem:sim1}, combined with cube automorphisms, tells us that $x\sim_k y$ if and only if, given the $(k+1)$-cube $\q_0$ with constant value $x$ and any vertex $v$, we obtain a $(k+1)$-cube by modifying $\q_0(v)$ from $x$ to $y$. The following lemma generalizes this so that $\q_0$ need not be constant.

\begin{lemma}\label{lem:sim2} For $x,y\in \ns$, we have $x\sim_k y$ if and only if for every $u\in \{0,1\}^{k+1}$ and every $\q_0\in \cu^{k+1}(\ns)$ with $\q_0(u)=x$, the map $\q_1:\{0,1\}^{k+1}\to \ns$, $v\mapsto \; \left\{\begin{array}{lr} \q_0(v),& v\neq u \\
  y, & v=u\end{array}\right.$ is in $\cu^{k+1}(\ns)$.
\end{lemma}

\begin{proof} By composition with automorphisms, it suffices to prove the case $u=0^{k+1}$. The `if' direction is immediate from Lemma \ref{lem:sim1}. For the converse, let $\phi=f^{k+1}:T_{k+1}\to\{0,1\}^{k+1}$ where $f(-1)=0,f(0)=0,f(1)=1$. For each subcube of $T_{k+1}$, with corresponding map $\trem_v:\{0,1\}^{k+1}\to T_{k+1}$ as specified in \eqref{eq:cubemb}, there is a unique subset $S$ of $[k+1]$ of indices of coordinates that can be negative for points in that subcube, namely $S=\{i\in [k+1]:v\sbr{i}=0\}$. Note that $\phi$ is thus the morphism that sends such a subcube to the lower face $\{v\in \{0,1\}^{k+1}: v\sbr{j}=0,\,\forall\, j\in S\}$. In particular, the morphism $\q_0\co \phi$ takes the same values on the set of outer points of $T_{k+1}$ as $\q_0$ does on $\{0,1\}^{k+1}$ (identifying these two sets the natural way). Moreover, on the subcube corresponding to $S=[k+1]$, the map $\q_0\co \phi$ is the constant $x$. By Lemma \ref{lem:sim1}, if we change the value of $\q_0\co \phi\big((-1)^{k+1}\big)$ from $x$ to $y$, then we still have a $(k+1)$-cube and so the resulting function $T_{k+1}\to \ns$ is still a morphism. The result now follows by Lemma \ref{lem:tricube-comp}.
\end{proof}
 By iterating this result, we deduce the following.

\begin{corollary}\label{sim2cor}
Let $k\in\mathbb{N}$ and let $\q_0\in \cu^{k+1}(\ns)$. If a function $\q_1:\{0,1\}^{k+1}\to \ns$ satisfies $\q_1(v)\sim_k \q_0(v)$ for every $v\in\{0,1\}^{k+1}$, then $\q_1\in \cu^{k+1}(\ns)$.
\end{corollary}

Another consequence is the following result generalizing Corollary \ref{cor:cube-det}.

\begin{corollary}\label{cor:sim2cor2} Let $\ns$ be a nilspace and let $k\in \N$. A cube $\q\in \cu^n(\ns/\sim_k)$ is uniquely determined by its values on $\{0,1\}^n_{\leq k}$.
\end{corollary}
\begin{proof} We first check the case $n=k+1$. A cube in $\cu^{k+1}(\ns/\sim_k)$ is by definition of the form $\pi_k\co \q$ for some $\q\in\cu^{k+1}(\ns)$. Suppose that $\pi_k\co \q_0$ and $\pi_k\co \q_1$ are two such cubes with the same value at each $v\neq 1^{k+1}$. Then using Lemma \ref{lem:sim2} at each $v\neq 1^{k+1}$, we can modify $\q_1$ to a new cube $\q_2$ having same values as $\q_0$ except at $1^{k+1}$ where it is still $\q_1(1^{k+1})$. By Definition \ref{def:simdef}, we then have $\q_0(1^{k+1})\sim_k \q_1(1^{k+1})$, so $\pi_k\co \q_0=\pi_k\co \q_1$ as required. For $n>k+1$, one can use an argument by induction on $|v|$ similar to the one in the proof of Lemma \ref{lem:Gray-cubes=>G-cubes} (using the case $k+1$ at each step).
\end{proof}

\begin{lemma}\label{lem:F_k-defin}
Let $k\in\N$ and let $\ns$ be a nilspace. Then $\ns/\sim_k$ with the quotient cubespace structure is a $k$-step nilspace. We denote this factor of $\ns$ by $\cF_k(\ns)$, and $\pi_k$ denotes the canonical projection $\ns\to \cF_k(\ns)$.
\end{lemma}
In particular we have that $\ns$ is $k$-step if and only if $\ns= \cF_k(\ns)$.

\begin{proof} Note first that $\cF_k(\ns)$ inherits ergodicity from $\ns$. To see that $\cF_k(\ns)$ satisfies the completion axiom, let $\q':\{0,1\}^n\setminus\{1^n\}\to \cF_k(\ns)$ be an $n$-corner. It follows from Corollary \ref{sim2cor} that there exists a morphism $\bar{f}$ from the simplicial cubespace $\{0,1\}^n_{\leq k+1}$ to $\ns$, such that $\pi_k\co \bar{f}$ equals $\q'$ on $\{0,1\}^n_{\leq k+1}$. By simplicial completion (Lemma \ref{lem:simpcomp}), $\bar{f}$ extends to a cube $\bar{\q}:\{0,1\}^n \to \ns$. The maps $\pi_k\co \bar{\q}$ and $\q'$ are equal on $\{0,1\}^n_{\leq k+1}$. By Corollary \ref{cor:sim2cor2}, this equality extends to all of $\{0,1\}^n\setminus\{1^n\}$. Hence, the cube $\q=\pi_k\co \bar{\q}$ completes $\q'$. Completion of $(k+1)$-corners is unique by Corollary \ref{cor:sim2cor2}, so $\mathcal{F}_k(\ns)$ is a $k$-nilspace.
\end{proof}
The following useful result enables us to lift $\mathcal{F}_k(\ns)$-valued morphisms to $\mathcal{F}_{k+1}(\ns)$-valued ones.

\begin{lemma}[Lifting a morphism]\label{lem:lifting} Let $P$ be a subcubespace of $\{0,1\}^n$ with the extension property, let $k\in \N$ and let $\ns$ be a nilspace. Then every morphism $f:P\to\cF_k(\ns)$ can be lifted to $\cF_{k+1}(\ns)$, that is there exists a morphism $f':P\to\cF_{k+1}(\ns)$ such that $\pi_k\co f'= f$.
\end{lemma}

\begin{proof} Fix any $n\geq 0$. For $P=\{0,1\}^n$ the claim follows  clearly from the definition of $\cu^n(\mathcal{F}_k(\ns))$. For $P$ a proper subset of $\{0,1\}^n$, we can first complete $f$ to a cube $\q\in \cu^n(\mathcal{F}_k(\ns))$, then lift $\q$ to a cube $\bar{\q}\in\cu^n(\mathcal{F}_{k+1}(\ns))$, then set $f'=\bar{\q}|_P$. 
\end{proof}

\begin{remark}\label{rem:cubesuff} Note that by combining Lemma \ref{lem:lifting} with Corollary \ref{sim2cor} we have that if $\q:\{0,1\}^{k+1}\to \cF_k(\ns)$ is a cube, then any lift of $\q$ given by the last lemma is also a cube on $\cF_{k+1}(\ns)$. In particular, for $\q:\{0,1\}^{k+1}\to \cF_{k+1}(\ns)$ to be a cube it suffices to have $\pi_k\co\q$ being a cube on $\cF_k(\ns)$. For instance if $\ns=\cF_2(\ns)$ is a 2-step nilpotent group $G$ (with lower central series), then this is saying that a map $\{0,1\}^2\to G$ is a $2$-cube if and only if its projection to the abelianization $G/[G,G]$ gives an additive quadruple (i.e. a quadruple $(a_{00},a_{10},a_{01},a_{11})$ such that $a_{00}+a_{10}=a_{01}+ a_{11}$).
\end{remark}

\begin{lemma}\label{lem:cubechar} Let $\ns$ be a $k$-step nilspace and let $n\geq k+2$. A map $\q:\{0,1\}^n\to \ns$ is in $\cu^n(\ns)$ if and only if, for every $(k+1)$-dimensional face $F\subset \{0,1\}^n$ containing some point $v$ with $v\sbr{n}=0$, the restriction $\q\co\phi_F$ is in $\cu^{k+1}(\ns)$.
\end{lemma}

\begin{proof} The forward implication is clear, by the composition axiom. For the converse, let $P=\{0,1\}^n_{\leq k}$, and note that this set is the union of the $k$-dimensional lower faces of the $n$-cube. Any such face $F$ can be embedded in some $(k+1)$-dimensional face, and the latter then contains $0^n$, so by our assumption and composition we have that $\q|_F$ is a cube. Hence $\q|_P$ is a morphism and so by Lemma \ref{lem:simpcomp} there exists $\q'$ in $\cu^n(\ns)$ such that $\q'|_P = \q|_P$. We claim that $\q = \q'$. Let $t$ be the maximal integer such that $\q = \q'$ on $\{0,1\}^n_{\leq t}$. Supposing for a contradiction that $t<n$, there is then $w\in\{0,1\}^n_{t+1}$ such that $\q'(w)\neq \q(w)$. Since $t \geq k$, it follows that we can find a $(k+1)$-face $F$ containing $w$, such that every $v\in F\setminus\{w\}$ has $|v| \leq t$, and such that $F$ contains some point $v$ with $v\sbr{n}=0$. (We can take $F$ to be the lower face defined by $v\sbr{i}=w\sbr{i}$ for every $i\in I$, where $I\subset [n]$ is the complement of the set of the last $k+1$ elements from $\supp(w)$.) Now by our initial assumption we have that $\q|_F\in \cu^{k+1}(\ns)$, so by the uniqueness of completion from $F\setminus\{w\}$ to $F$, we must have $\q|_F=\q'|_F$, contradicting the maximality of $t$.
\end{proof}

\subsection{$k$-fold ergodic nilspaces as degree-$k$ abelian torsors}
Recall from Subsection \ref{sect:deg-k-str} that the degree-$k$ structure $\cD_k(\ab)$ on an abelian group $\ab$ is a $k$-fold ergodic nilspace. An important step towards the general decomposition theorem is to show that all $k$-fold ergodic $k$-step nilspaces arise this way. We have already checked the case $k=1$ of this statement, in Proposition \ref{prop:abel}. Here we shall establish the general case.

\begin{proposition}\label{prop:k-erg}
Let $\ns$ be a $k$-step, $k$-fold ergodic nilspace. Then there is an abelian group $\ab$ such that $\ns$ is isomorphic to $\cD_k(\ab)$. 
\end{proposition}
\noindent In other words, we have that $\ns$ is a principal homogeneous space of $\ab$ and the cubes on $\ns$ are the images of the degree-$k$ cubes on $\ab$. 

Examples of $k$-step $k$-fold ergodic nilspaces arise naturally when we consider the equivalence classes of the relation $\sim_{k-1}$ on a $k$-step coset nilspace $(G/\Gamma,G_\bullet)$. Indeed, as was already mentioned in Section \ref{sec:Motiv}, for such a nilspace the group $\ab=G_k/(G_k\cap \Gamma)$ has a free action on $G/\Gamma$, and combining this with the remarks in Example \ref{ex:rels} one can see that each class of $\sim_{k-1}$ is a principal homogeneous space of $\ab$; one can then see that restricting the cube structure to any such class, one obtains a $k$-fold ergodic $k$-step nilspace, which is precisely $\cD_k(\ab)$. The importance of Proposition \ref{prop:k-erg} lies in that it will enable us to recover a similar picture concerning the equivalence classes of $\sim_{k-1}$ for a \emph{general} $k$-step nilspace.

Note that by Lemma \ref{lem:nilspace-dif}, if we fix an element $e\in \ns$ and apply $\partial_e$ to $\ns$ repeatedly $k-1$ times, then the result is a 1-step nilspace, so it is isomorphic (as a nilspace) to an affine abelian group with the degree-$1$ structure. The proof of the proposition will therefore consist in `integrating' back $k-1$ times to obtain the desired conclusion. To this end, we shall argue by induction on $k$, using the following key fact.

\begin{lemma}\label{lem:k-erg-keyfact}
Let $k\geq 2$ and suppose that Proposition \ref{prop:k-erg} holds for $k-1$. Let $\ns$ be a $k$-fold ergodic $k$-nilspace, fix $e\in \ns$ and let $\ab$ be an abelian group such that $\partial_e^{k-1} \ns$ is isomorphic to $\cD_1(\ab)$. Then two elements $x=(x_0,x_1)$, $y=(y_0,y_1)$ of $\ns\Join_1 \ns$ satisfy $x \sim_{k-1} y$ if and only if $x_0-x_1=y_0-y_1$ in $\ab$.  
\end{lemma}

\begin{proof}
Let $\nss=\ns\Join_1 \ns$. Let $\q_0,\q_1 :\{0,1\}\to \ns$ be the maps defined by $\q_0(i)=x_i$, $\q_1(i)=y_i$ for $i=0,1$. By definition of $\cD_1(\ab)$, we have $x_0-x_1=y_0-y_1$ in $\ab$ if and only if the 1-arrow $f_0=\arr{\q_0,\q_1}_1$ is in $\cu^2(\cD_1(\ab))$. We know that $\cD_1(\ab)\cong \partial_e^{k-1} \ns$, and so $f_0 \in \cu^2(\cD_1(\ab))$ in turn means that the 2-arrow $\arr{e, f_0}_1$ is in $\cu^3(\partial^{k-2}_e \ns)$. This in turn means that the 3-arrow $\arr{e,\arr{e, f_0}_1}_1\in \cu^4(\partial^{k-3}_e \ns)$, and so on. Continuing like this, we conclude that $x_0-x_1=y_0-y_1$ if and only if the $k$-arrow $f_{k-1}:= \arr{e,\arr{e,\ldots,\arr{e, f_0}_1}_1\cdots}_1$ is in $\cu^{k+1}(\ns)$. Note that $f_{k-1}(v)= f_0(v\sbr{1},v\sbr{2})$ if $v\sbr{3}=\cdots=v\sbr{k+1}=1$, and $f_{k-1}(v)=e$ otherwise. We now claim that $f_{k-1}\in \cu^{k+1}(\ns)$ if and only if $x\sim_{k-1} y$ in $\nss$.\\
\indent Let $f:\{0,1\}^k\to \ns\times \ns$ be defined by $f(v)=f_{k-1}(0,v\sbr{1},v\sbr{2},\dots,v\sbr{k})\times f_{k-1}(1,v\sbr{1},v\sbr{2},\dots,v\sbr{k})$, and note that $f(0,1,\dots,1)=x$, $f(1^k)=y$, and $f(v)=(e,e)$ otherwise. We also have $f_{k-1}\in \cu^{k+1}(\ns)$ if and only if $f\in \cu^k(\nss)$ (using \eqref{eq:arrowcubechar} and cube automorphisms). Since $\nss$ is a $k$-nilspace, by Remark \ref{rem:cubesuff} we have $f\in \cu^k(\nss)$ if and only if $\pi_{k-1}\co f \in \cu^k(\cF_{k-1}(\nss))$. Now, since $\nss$ is $(k-1)$-fold ergodic, so is the $(k-1)$-nilspace $\cF_{k-1}(\nss)$. Hence, by our assumption, $\cF_{k-1}(\nss)$ is isomorphic to $\cD_{k-1}(\ab')$ for some abelian group $\ab'$. Therefore $\pi_{k-1}\co f \in \cu^k(\cF_{k-1}(\nss))$ if and only if $\sigma_k(\pi_{k-1}\co f)=0_{\ab'}$. But given the form of $f$, this is equivalent to $\pi_{k-1}(x) = \pi_{k-1}(y)$, as required.
\end{proof}

\begin{proof}[Proof of Proposition \ref{prop:k-erg}]
We argue by induction on $k$. The case $k=1$ is given by Proposition  \ref{prop:abel}. Fix $k\geq 2$, assume that the statement holds for $k-1$,  fix an element $e\in \ns$, and suppose that $\ab$ is an abelian group such that $\partial_e^{k-1}\ns = \mathcal{D}_1(\ab)$. Our aim is to show that for every $n\geq 0$ we have $\cu^n(\ns)=\cu^n(\cD_k(\ab))$. This holds trivially for $n\leq k$ since $\ns=\ab$ as a set and $\cD_k(\ab)$, $\ns$ are both $k$-fold ergodic. To prove the case $n=k+1$, first we claim that if $\q \in \cu^{k+1}(\ns)$ then adding an element $a\in \ab$ to the values of $\q$ at two endpoints of an arbitrary 1-face  in $\{0,1\}^{k+1}$, we obtain another $(k+1)$-cube $\q'$. To see this, we assume without loss of generality that the values in question are $\q(0^{k+1})=x_0$, $\q(0,\ldots, 0,1)=x_1$. Letting $\nss=\ns\Join_1\ns$, the map $f:=\q(\cdot,0) \times \q(\cdot ,1):\{0,1\}^k \to \nss$ is in $\cu^k(\nss)$, since $\q\in \cu^{k+1}(\ns)$. Moreover, $f$ takes the value $(x_0,x_1)$ at $0^k$. On the other hand, by Lemma \ref{lem:k-erg-keyfact} we have $(x_0,x_1)\sim_{k-1} (x_0+a,x_1+a)$. Therefore, letting $f':\{0,1\}^k\to \nss$ take value $(x_0+a,x_1+a)$ at $0^k$ and $f(v)$ otherwise, by Corollary \ref{sim2cor} we have $f'\in \cu^k(\nss)$. Hence the arrow $\q'= \arr{f'(\cdot,0),f'(\cdot,1)}_1$ is in $\cu^{k+1}(\ns)$, which proves our claim.

Now, given any cube $\q\in \cu^{k+1}(\ns)$, by adding appropriate  elements of $\ab$ on 1-faces of the cube as above, along a Hamiltonian path in the hypercube graph\footnote{The edges of this graph are the 1-faces of $\{0,1\}^{k+1}$.} on $\{0,1\}^{k+1}$, we obtain a cube $\q'$  taking value $e$ everywhere except perhaps at $0^{k+1}$. But since the constant $e$ map is in $\cu^{k+1}(\ns)$, by uniqueness of completion we must have in fact $\q'(0^{k+1})=e$. Thus we have shown that every $\q \in \cu^{k+1}(\ns)$ is reducible to a constant map by these 1-face modifications. Since constant maps are in $\cu^{k+1}(\cD_k(\ab))$, and these modifications conserve this cube set, it follows that $\cu^{k+1}(\ns)\subset \cu^{k+1}(\cD_k(\ab))$. Conversely, if $\q \in \cu^{k+1}(\cD_k(\ab))$, then restricting $\q$ to any lower $k$-face in $\{0,1\}^{k+1}$ we obtain a cube in $\cu^k(\cD_k(\ab))=\cu^k(\ns)$. Therefore $\q$ restricted to $\{0,1\}^{k+1}\setminus \{1^{k+1}\}$ gives a $(k+1)$-corner on $\ns$, which then has a completion $\q'\in \cu^{k+1}(\ns)\subset \cu^{k+1}(\cD_k(\ab))$. Since completions are unique in the latter cube set, we must have $\q'=\q$. Hence $\cu^{k+1}(\ns)= \cu^{k+1}(\cD_k(\ab))$.\\
 \indent We now deduce that $\cu^n(\ns)= \cu^n(\cD_k(\ab))$ for every $n>k+1$, using Lemma \ref{lem:cubechar}.
\end{proof}

\noindent We close this section by using Proposition \ref{prop:k-erg} to characterize equivalence classes of $\sim_{k-1}$ in a $k$-step nilspace. It turns out that every such class is a degree-$k$ abelian group.

\begin{corollary}\label{cor:fibres}
Let $\ns$ be a $k$-step nilspace and let $F$ be an equivalence class of $\sim_{k-1}$ in $\ns$. Then $F$ with the cubespace structure restricted from $\ns$ is isomorphic to $\cD_k(\ab)$ for some abelian group $\ab$.
\end{corollary}

\begin{proof}
By Proposition \ref{prop:k-erg} it is enough to show that $F$ is a $k$-fold ergodic $k$-nilspace. For every $x\in F$,  the constant $x$ function on $\{0,1\}^k$ is in $\cu^k(\ns)$ and so by Corollary \ref{sim2cor} every function $\{0,1\}^k\rightarrow F$ is in $\cu^k(\ns)$. To see that completion holds, let $\q'$ be an $n$-corner on $F$ with $n\geq k+1$. Since $\ns$ is a $k$-step nilspace, $\q'$ has a unique completion $\q\in \cu^n(\ns)$. But then $\pi_{k-1}\co \q'$ is constant and so $\pi_{k-1}\co \q$ must also be constant, as this is the only completion of $\pi_{k-1}\co \q'$, whence $\q$ is $F$-valued as required.
\end{proof}
\medskip

\subsection{Abelian bundles, and the bundle decomposition}\label{subsec:bundecomp}
\medskip

In this section we shall establish a general decomposition theorem for nilspaces, Theorem \ref{thm:bundle-decomp}. This important result relies on the following basic concept.

\begin{defn}[Abelian bundle]\label{def:AbelianBundle}
Let $\ab$ be an abelian group. An \emph{abelian bundle} over a set $S$ with structure group $\ab$ (or $\ab$-\emph{bundle over} $S$) is a set $\bnd$ with an action $\alpha: \ab\times \bnd\to \bnd$, $(z,x)\mapsto z+x$, and a map $\pi: \bnd\to S$ (called the \emph{bundle map} or \emph{projection}) satisfying the following properties:

1. The action $\alpha$ is free: $\forall \,x\in \bnd$ we have $\{z\in \ab: z+ x=x\}=\{0_{\ab}\}$.

2. The map $s\mapsto \pi^{-1}(s)$ is a bijection from $S$ to the set of orbits of $\ab$ in $\bnd$.

\noindent A set $\bnd$ is a $k$-\emph{fold abelian bundle} with structure groups $\ab_1,\ab_2,\ldots, \ab_k$ if there is a sequence $\bnd_0,\bnd_1,\ldots,\bnd_k=\bnd$ where $\bnd_0$ is a singleton and $\bnd_i$ is a $\ab_i$-bundle over $\bnd_{i-1}$. We denote by $\pi_{i+1,i}$ the bundle map $\bnd_{i+1}\to \bnd_i$. More generally $\pi_{i,j}$ denotes the map $\pi_{j+1,j}\co\pi_{j+2,j+1} \co \cdots \co \pi_{i,i-1}:\bnd_i\to \bnd_j$,   $i\geq j$. We write $\pi_i$ for $\pi_{k,i}$.

A \emph{relative $k$-fold abelian bundle} is a generalization of a $k$-fold abelian bundle in which the ground set $\bnd_0$ can be an arbitrary set. 
\end{defn}

\medskip

\noindent We shall often call the set $S$ the \emph{base} of the bundle. Note that it follows from the second condition above that the action of $\ab$ is transitive on each fibre of $\pi$, i.e. on each set $\pi^{-1}(s)$, $s\in S$. Thus, an abelian bundle has the algebraic properties of a principal bundle with abelian structure group $\ab$, without any topological assumptions.

\begin{defn}\label{def:k-deg-bund}
A \emph{degree-$k$ bundle} is a cubespace $\ns$ that is also a $k$-fold abelian bundle, with factors $\bnd_0,\bnd_1,\ldots, \bnd_k=\ns$ and structure groups $\ab_1,\ab_2,\ldots, \ab_k$, and with the following property: for every integer $i \in [0,k-1]$ and every $n\in \N$, we have $\cu^n(\bnd_i) = \{\pi_i\co \q: \q\in  \cu^n(\ns)\}$, and for every $\q\in \cu^n(\bnd_{i+1})$ we have
\begin{equation}\label{eq:k-deg-bund}
\{\q_2 \in \cu^n(\bnd_{i+1}): \pi_i\co \q = \pi_i\co \q_2\} = \{\q+\q_3: \q_3\in \cu^n(\cD_{i+1}(\ab_{i+1}))\}.
\end{equation}
\end{defn}
\noindent Note that the condition $\q_3\in \cu^n(\cD_j(\ab_j))$ is satisfied for all maps $\q_3\in \ab_j^{\{0,1\}^n}$ if $n\leq j$, since $\cD_j(\ab_j)$ is $j$-fold ergodic.

We can now state the main result of this section.

\begin{theorem}\label{thm:bundle-decomp}
Let $\ns$ be a cubespace. Then $\ns$ is a $k$-step nilspace if and only if it is a degree-$k$ bundle. Moreover, we then have $\cF_i(\ns)=\bnd_i$ for every $i\in [k]$.
\end{theorem}

We split the proof into several results. 
\begin{lemma}\label{lem:k-degbundisns}
A degree-$k$ bundle $\ns=\bnd_k$ is a $k$-nilspace. Moreover, we have $\mathcal{F}_i(\ns)=\bnd_i$ for every $i\in [k]$.
\end{lemma}
\begin{proof}
Ergodicity can be shown to hold by induction on $k$. Indeed, a degree-$1$ bundle is an affine abelian group, so clearly ergodic. For $k>1$, given $f=(x_0,x_1)\in \ns^{\{0,1\}}$, the projection $\pi_{k-1}\co f$ is in $\cu^1(\bnd_{k-1})$  by induction, then by surjectivity there is $\q\in \cu^1(\ns)$ with $\pi_{k-1}\co f = \pi_{k-1}\co \q$, but then $f=\q+\q_3$ where $\q_3\in \ab_k^{\{0,1\}}= \cu^1(\cD_k(\ab_k))$, so by \eqref{eq:k-deg-bund} we have $f\in \cu^1(\bnd_k)$.\\
\indent The completion axiom is also checked by induction. Completion is trivial on $\bnd_0$ so suppose that $k>0$ and we have completion on any degree-$(k-1)$ bundle. Let $\q'$ be an $n$-corner on $\bnd_k$. Then $\pi_{k-1} \co \q'$ has a completion $\q_1\in \cu^n(\bnd_{k-1})$. Since $\cu^n(\bnd_{k-1})=\pi_{k-1}(\cu^n(\ns))$, there exists $\q_2\in \cu^n(\ns)$ such that $\q_1=\pi\co \q_2$. The map $\q_3'=\q'-\q_2$ is an $n$-corner on $\cD_k(\ab_k)$, so it can be completed to an $n$-cube $\q_3:\{0,1\}^n\to \cD_k(\ab_k)$. The cube $\q_2+\q_3$ completes $\q'$.\\
\indent
To check that $\cF_i(\ns)=\bnd_i$ for all $i\in [k]$, we can assume that $\ns=\bnd_{i+1}$ and show that $\pi_i$ induces a bijection from the classes of $\sim_i$ to the points of $\bnd_i$. Suppose first that $x\sim_i y$, i.e. there exist $\q_0,\q_1\in \cu^{i+1}(\ns)$ with $\q_0(1^{i+1})=x$, $\q_1(1^{i+1})=y$ and $\q_0(v)=\q_1(v)$ otherwise. Then $\pi_i\co \q_0$, $\pi_i\co \q_1$ are both $(i+1)$-cubes on the $i$-nilspace $\bnd_i$, and they complete the same $(i+1)$-corner, whence $\pi_i(x)=\pi_i(y)$. Next, note that if $\pi_i(x)=\pi_i(y)$ then there is $z\in \ab_{i+1}$ such that $y=x+z$. The map $\q_0$ taking value $x$ everywhere on $\{0,1\}^{i+1}$ is in $\cu^{i+1}(\bnd_{i+1})$. Letting $\q_1(v)=\q_0(v)$ for $v\neq 1^{i+1}$ and  $\q_1(1^{i+1})=y$, we have $\q_1=\q_0+\q_3$ where $\q_3(1^{i+1})=z$ and $\q_3(v)=0_{\ab_{i+1}}$ otherwise. Since $\cD_{i+1}(\ab_{i+1})$ is $(i+1)$-fold ergodic, we have $\q_3\in \cu^{i+1}(\cD_{i+1}(\ab_{i+1}))$, whence $\q_1 \in \cu^{i+1}(\ns)$ by \eqref{eq:k-deg-bund}, and so $x\sim_i y$. 
\end{proof}
\noindent We now start with a $k$-step nilspace $\ns$. By Corollary \ref{cor:fibres}, each equivalence class (or fibre) $F$ of $\sim_{k-1}$ is isomorphic to $\cD_k(\ab_F)$ for some abelian group $\ab_F$. To complete the proof of Theorem \ref{thm:bundle-decomp}, our task now is to relate different fibres to identify a unique group $\ab_k$ acting on every fibre. To this end we shall use the following relation.
\begin{defn}\label{defn:keyrelation}
Let $\ns$ be a $k$-nilspace. Let $M=\{(x,y)\in \ns^2: x\sim_{k-1} y\}$. We define a relation $\sim$ on $M$ by declaring that $(x_0,y_0)\sim (x_1,y_1)$ if $(x_0,x_1)\sim_{k-1} (y_0,y_1)$ in $\nss=\ns\Join_1 \ns$.
\end{defn}
\noindent Recall that in Lemma \ref{lem:k-erg-keyfact} we used this same relation, but restricted to a degree-$k$ abelian group, which amounts here to a single fibre $F$. In that case we could characterize this relation in terms of the group acting on the fibre. In the present setting, we are working over all of $\ns$.
\begin{lemma}\label{lem:relchar}
Given $(x_0,y_0),(x_1,y_1)\in M$, let $f:\{0,1\}^{k+1}\to \ns$ be defined by setting $f(\cdot,0)$ equal to $y_0$ at $0^k$ and $x_0$ elsewhere, and  $f(\cdot,1)$ equal to $y_1$ at $0^k$ and $x_1$ elsewhere. Then $(x_0,y_0)\sim (x_1,y_1)$ if and only if $f\in \cu^{k+1}(\ns)$. The relation $\sim$ is an equivalence relation on $M$.
\end{lemma}
\begin{proof}
By Lemma \ref{lem:sim1}, $(x_0,x_1)\sim_{k-1} (y_0,y_1)$ if and only if there is $\q=\q_0\times \q_1 \in \cu^k(\nss)$ such that $\q(0^k) = (y_0,y_1)$ and $\q(v)=(x_0,x_1)$ otherwise. Recall that $\q \in \cu^k(\nss)$ means that the  arrow $\arr{\q_0, \q_1}_1$ is in $\cu^{k+1}(\ns)$. Noting that $\arr{\q_0, \q_1}_1$ is the map $f$ in the lemma, the first claim follows. We have that $\sim$ is reflexive since if $(x,y)$ is in $M$ then letting $\q_0,\q_1$ both be the cube witnessing that $x\sim_{k-1} y$, we have $f=\arr{\q_0, \q_1}_1\in \cu^{k+1}(\ns)$. Symmetry is clear using automorphisms of cubes. Finally, to see transitivity, note that if $(x_0,y_0)\sim (x_1,y_1) \sim (x_2,y_2)$ then the cube $f_1$ for the first relation can be concatenated with the cube $f_2$ for the second relation, yielding a cube $f$ showing that $(x_0,y_0)\sim (x_2,y_2)$.
\end{proof}
\begin{lemma}\label{lem:isokey}
Let $\ns$ be a $k$-nilspace and let $F_0,F_1$ be distinct equivalence classes of $\sim_{k-1}$ on $\ns$. For any $x_0,y_0\in F_0$ and $x_1\in F_1$,  there exists a unique element $y_1\in F_1$ such that $(x_0,y_0)\sim (x_1,y_1)$.
\end{lemma}
\begin{proof}
First let us show that there exists a unique $y_1\in \ns$ such that $(x_0,y_0)\sim (x_1,y_1)$. It suffices to show that the map $f'$, obtained by restricting to $\{0,1\}^{k+1}\setminus \{(0^k,1)\}$ the 1-arrow $f$ from Lemma \ref{lem:relchar}, is a $(k+1)$-corner (modulo an automorphism), for then by completion on $\ns$ there exists a unique such $y_1$. To see that $f'$ is a $(k+1)$-corner, note that any of its $k$-face restrictions is either a map $f_1$ equal to $y_0$ at $0^k$ and $x_0$ elsewhere in $\{0,1\}^k$, or is a map $f_2$ equal to $x_0$ on one $(k-1)$-face and $x_1$ on the opposite face. We have $f_1\in \cu^k(\ns)$ by Lemma \ref{lem:sim1}, since $x_0\sim_{k-1} y_0$. 
We have $f_2\in \cu^k(\ns)$ because $f_2=f_2'\co \phi$, where $f_2'$ is the 1-cube $(x_0,x_1)$ and $\phi:\{0,1\}^k\to \{0,1\}$ is the morphism projecting to some fixed coordinate. Hence there is indeed a unique $y_1$ completing $f'$ to a cube $f$ of the form given in Lemma \ref{lem:relchar}.\\
\indent To see that $y_1\in F_1$, note that since $\pi_{k-1}(x_0)=\pi_{k-1}(y_0)$, the $(k+1)$-corner $\pi_{k-1}\co f'$ is constant on the $k$-face involving only $x_0,y_0$. We can complete this to the cube that is the constant $\pi_{k-1}(x_1)$ on the opposite face. Since this completion is unique, it must be the cube $\pi_{k-1}\co f$, so we must have $\pi_{k-1}(y_1)=\pi_{k-1}(x_1)$.
\end{proof}
\noindent Recall that by Lemma \ref{lem:k-erg-keyfact}, if $F$ is a fibre of $\sim_{k-1}$ and $x_0,x_1,y_0,y_1\in F$ then $(x_0,y_0)\sim (x_1,y_1)$ if and only if $x_0-x_1=y_0-y_1$ in $\ab_F$. This means that inside any fibre $F$ the elements of $\ab_F$ are in bijection with the fibres of $\sim$, the bijection being well-defined by $(x_0,y_0)\mapsto a=y_0-x_0\in \ab_F$. We now use this fact to relate the groups of two distinct fibres of $\sim_{k-1}$.
\begin{lemma}\label{lem:fibreiso}
Let $\ns$ be a $k$-nilspace and let $F_0,F_1$ be distinct fibres of $\sim_{k-1}$. Define $\vartheta:\ab_{F_0}\to \ab_{F_1}$ by
\begin{equation}
\vartheta(a)=b\;\textrm{ if and only if }\;(x_0,x_0+a)\sim (x_1,x_1+b)\;\textrm{ for some }\;x_0\in F_0,\;x_1\in F_1.
\end{equation}
Then $\vartheta$ is a group isomorphism.
\end{lemma}
\begin{proof}
The map $\vartheta$ is well-defined, for if $(x_1',y_1')\sim (x_0,x_0+a)$ with $x_1'\in F_1$, then $y_1'$ is also in $F_1$, and then by transitivity $(x_1,x_1+b)\sim (x_1',y_1')$, so $y_1'=x_1'+b$. Moreover, $\vartheta$ is a bijection because given $b\in \ab_{F_1}$ and any $x_1\in F_1$, for any $x_0\in F_0$ there is a unique $y_0\in F_0$ such that $(x_0,y_0)\sim (x_1,x_1+b)$, but then $\vartheta(y_0-x_0)=b$.\\
\indent To see that $\vartheta$ is a homomorphism, let $a,a'\in \ab_{F_0}$, fix $x_0\in F_0$, let $y_0=x_0+a$, $z_0=y_0+a'$. Fix any $x_1\in F_1$ and let $y_1,z_1\in F_1$ be the unique elements given by Lemma \ref{lem:isokey} such that $(x_0,y_0)\sim (x_1,y_1)$ and $(y_0,z_0)\sim (y_1,z_1)$. By transitivity of $\sim_{k-1}$ we have $(x_0,z_0)\sim (x_1,z_1)$, and it follows from the definitions that $\vartheta(a+a')=\vartheta(a)+\vartheta(a')$.
\end{proof}
\noindent Thanks to these isomorphisms, we can now identify all the groups $\ab_F$ and talk about a single abelian group $\ab$ acting on every fibre $F$, sending $x\in F$ to $x+z\in F$ for each $z\in \ab$. Thus $\ab$ acts on all of $\ns$. The action is also free and transitive on each fibre.\\ 

\noindent We have thus shown that the given $k$-step nilspace $\ns$ is a $\ab$-bundle over $\mathcal{F}_{k-1}(\ns)$, with bundle map the canonical projection $\pi_{k-1}$ for $\sim_{k-1}$. It remains only to show that $\ns$ is a degree-$k$ bundle.
\begin{lemma}\label{lem:Z-bund-is-k-degbund}
The $\ab$-bundle $\ns$ satisfies \eqref{eq:k-deg-bund}.
\end{lemma}
\begin{proof}
First we claim that if $\q\in \cu^{k+1}(\ns)$ and we let any $a\in \ab$ act on the values of $\q$ at the two points of an arbitrary 1-face in $\{0,1\}^{k+1}$, then the resulting map $\q'$ is also in $\cu^{k+1}(\ns)$. To see this, note that without loss we can suppose that the given 1-face is $0^{k+1},(1,0,\dots,0)$, and let us denote the  values of $\q$ at these points by $x$, $y$ respectively. We have $(x,x+a)\sim (y,y+a)$, which means that the corresponding map $f:\{0,1\}^{k+1}\to \ns$ from Lemma \ref{lem:relchar} is in $\cu^{k+1}(\ns)$. 
Let $\phi$ be the map $T_{k+1}\to \{0,1\}^{k+1}$ used in the proof of Lemma \ref{lem:sim2}. As noted there, the morphism $\q\co\phi :T_{k+1}\to \ns$ agrees with $\q$ on the outer points, and we also have that it has constant value $x$ on the subcube $\{-1,0\}^{k+1}$. Note moreover that on the subcube containing $(1,-1,\dots,-1)$ the map $\q\co \phi$ is constantly $x$ on face $v\sbr{1}=0$ and constantly $y$ on face $v\sbr{1}=1$. Changing the values $\q\co\phi (-1^{k+1})$ and $\q\co\phi(0,-1,\dots,-1)$ both to $x+a$, changing $\q\co\phi(1,-1,\dots,-1)$ to $y+a$, and using the fact that $f\in \cu^{k+1}(\ns)$, we see that the resulting map is still a morphism $T_{k+1}\to \ns$ and that on  the outer points it agrees with $\q'$. The claim now follows from Lemma \ref{lem:tricube-comp}.

Finally, this last claim enables us to show that \eqref{eq:k-deg-bund} holds. For suppose that $\q_1,\q_2\in \cu^{k+1}(\ns)$ have $\pi_{k-1}\co \q_1 = \pi_{k-1}\co \q_2$. Then starting at any given $1$-face of $\{0,1\}^{k+1}$, we can add an element from $\ab$ to the values of $\q_2$ at the vertices of that face so as to obtain a new cube which agrees with $\q_1$ at one of those vertices. Repeating this along a Hamiltonian path of 1-faces on $\{0,1\}^{k+1}$, we can produce a new $(k+1)$-cube $\q_2'$ differing from $\q_1$ at one vertex at most. Note also that we obtain $\q_2'$ by adding to $\q_2$ only elements from $\cu^{k+1}(\cD_k(\ab))$, so $\q_2' = \q_2+\q_3$ where $\q_3 \in \cu^{k+1}(\cD_k(\ab))$. Moreover, we must then have $\q_2' = \q_1$, by uniqueness of completion on $\ns$. This shows that
\[
\{\q_2\in \cu^{k+1}(\ns): \pi_{k-1}\co \q_2=\pi_{k-1}\co \q_1\} \subset \{\q_1 + \q_3 : \q_3 \in \cu^{k+1}(\cD_k(\ab))\}.
\]
The opposite inclusion follows similarly, modifying $\q_1\in \cu^{k+1}(\ns)$ into $\q_1+\q_3$ along 1-faces, showing thus that $\q_1+\q_3\in \cu^{k+1}(\ns)$. The resulting equality then extends to all $\cu^n(\ns)$ with $n>k+1$, essentially by Lemma \ref{lem:cubechar} (and we had the equality also for $n<k+1$ by induction).
\end{proof}

The proof of Theorem \ref{thm:bundle-decomp} is now complete.

\begin{defn}\label{def:loc-trans}
Let $\ns$ be a $k$-step nilspace, let $F_0,F_1$ be two fibres of $\sim_{k-1}$, and let $x_0\in F_0$, $x_1\in F_1$. We define the \emph{local translation} corresponding to $x_0,x_1$ to be the map $\phi_{x_0,x_1}$ sending each $y_0\in F_0$ to the unique $y_1\in F_1$ such that $(x_0,y_0)\sim(x_1,y_1)$, that is the unique $y_1$ completing the corner $\q':\{0,1\}^{k+1}\setminus\{1^{k+1}\}\to \ns$ defined by $\q'(v,0)=x_0$ for $v\neq 1^k$, $\q'(1^k,0) = y_0$ and $\q'(v,1)=x_1$ for $v\neq 1^k$.
\end{defn}

\subsection{Translations on a $k$-step nilspace form a filtered group of degree $k$}\label{subsec:Facegp}
Given any subset $F$ of $\{0,1\}^n$ and any map $\alpha: \ns \to \ns$, we define the map
\[
\alpha^F: \ns^{\{0,1\}^n}\to \ns^{\{0,1\}^n},\quad  \alpha^F(f): v \mapsto \left\{\begin{array}{lr} \alpha(f(v)),& v\in F \\  f(v), & v\not\in F\end{array}\right..
\]
\begin{defn}[Translations]\label{def:facetrans}
Let $\ns$ be a nilspace and let $i$ be a positive integer. A map $\alpha:\ns\to \ns$ is a \emph{translation of height} $i$ on $\ns$ (or \emph{$i$-translation}) if for every $n\geq i$ and every face $F\subset \{0,1\}^n$ of codimension $i$, the map $\alpha^F$ is cube-preserving, that is for every $\q\in \cu^n(\ns)$ we have $\alpha^F(\q)\in \cu^n(\ns)$. We denote the set of translations of height $i$ by $\tran_i(\ns)$.
\end{defn}
\noindent Using discrete-cube automorphisms, we see that in order for $\alpha$ to be an $i$-translation it suffices to have $\alpha^F$ being cube-preserving for \emph{some} face $F$ of codimension $i$. Note also that $\tran_i(\ns)\supset \tran_j(\ns)$ for any $i\leq j$. We shall often write simply $\tran(\ns)$ for $\tran_1(\ns)$, and refer to translations of height 1 simply as translations.

These translations were introduced by Host and Kra in \cite{HK,HKparas}, where they were studied in relation to parallelepiped structures and shown to form a nilpotent group. Analogously, in this subsection we treat the following main result.

\begin{proposition}\label{prop:TransGroup}
Let $\ns$ be a $k$-step nilspace. Then $\tran(\ns)$ is a group and $(\tran_i(\ns))_{i\geq 0}$ is a filtration of degree at most $k$. \noindent \textup{(}We set $\tran_0(\ns)=\tran_1(\ns)=\tran(\ns)$.\textup{)}
\end{proposition}
\noindent We refer to $\tran(\ns)$ as the \emph{translation group} of $\ns$. Before we turn to the proof of Proposition \ref{prop:TransGroup}, let us collect some basic observations about translations, which will also help to gain intuition about them. We begin with simple examples using the nilspaces from Chapter \ref{chap:Centrexs}.

\begin{example}
Let $(G,G_\bullet)$ be a filtered group of degree $k$, let $\Gamma$ be a subgroup of $G$, and let $\ns$ be the corresponding $k$-step coset nilspace $G/\Gamma$ with cube sets $\cu^n(\ns)=\{\pi_\Gamma\co \q: \q\in\cu^n(G_\bullet)\}$. Let $i\in [k]$,  $g\in G_i$, and let $\alpha:\ns\to \ns$ be the map $x\Gamma\mapsto (gx)\Gamma$. Then $\alpha \in \tran_i(\ns)$. To see this, note that for any $n$-cube $\pi_\Gamma\co \q$ on $\ns$ and any face $F$ of codimension $i$ in $\{0,1\}^n$, we have $\alpha^F(\pi_\Gamma\co \q )= \pi_\Gamma\co(g^F\q)$, where $g^F\q$ is the pointwise-multiplication map $\{0,1\}^n\to G$, $v\mapsto g^F(v)\q(v)$. By the definition and group property of $\cu^n(G_\bullet)$, we have $g^F\q\in \cu^n(G_\bullet)$. Hence $\alpha^F(\pi_\Gamma\co \q)\in \cu^n(\ns)$, and so indeed $\alpha\in \tran_i(\ns)$.
\end{example}
The following result describes some basic general properties of translations.
\begin{lemma}\label{lem:transfact1}
Let $\ns$ be a nilspace. Then every translation on $\ns$ is a nilspace morphism $\ns\to\ns$. Furthermore, for each $i\geq 1$ we have
\begin{equation}\label{eq:transfact11}
\forall\, \alpha\in \tran(\ns),\;\;\textrm{ if }\; x,y\in \ns\; \textrm{satisfy }\; x\sim_i y, \; \textrm{then } \alpha(x)\sim_i \alpha(y),
\end{equation}
and we also have
\begin{equation}\label{eq:transfact12}
\alpha\in \tran_i(\ns) \;\;\Rightarrow  \;\;\forall\,x\in \ns, \; \alpha(x)\sim_{i-1} x.
\end{equation}
\end{lemma}
\begin{proof}
To see that every $\alpha\in \tran(\ns)$ is a morphism, note that given any cube $\q\in \cu^n(\ns)$, the map $\alpha\co \q$ can be written as  the composition $\alpha^F\co \alpha^{F'}(\q)$ for two opposite faces $F,F'$ of codimension 1, so $\alpha\co \q\in \cu^n(\ns)$ as required. 

To see \eqref{eq:transfact11}, recall that by Lemma \ref{lem:sim1} we have $x\sim_i y$ if and only if there is $\q\in \cu^{i+1}(\ns)$ such that $\q(1^{i+1})=y$ and $\q(v)=x$ otherwise. Now given $\alpha\in \tran(\ns)$, for this cube $\q$ we have by the previous paragraph that $\alpha\co \q\in \cu^{i+1}(\ns)$, and therefore $\alpha(x)\sim_i\alpha(y)$.

To see \eqref{eq:transfact12}, note that letting $\q$ denote the constant $i$-dimensional cube with value $x$, and letting $\q'$ denote the map $\{0,1\}^i \to \ns$ such that $\q'(v)=\q(v)=x$ for $v\neq 0^i$ and $\q'(0^i)=\alpha(x)$, we have by the definition of $i$-translations that $\q'$ is also in $\cu^i(\ns)$, so by the fact recalled above we have $\alpha(x) \sim_{i-1} x$.
\end{proof}
\noindent Thus, translations of height $i$ leave the equivalence classes of $\sim_{i-1}$ in $\ns$ globally invariant, and, more generally, translations conserve all the equivalence relations $\sim_j$. When $\ns$ is a $k$-step nilspace, we can deduce the following fact, relating translations to the notion of \emph{local} translation from Definition \ref{def:loc-trans}.
\begin{lemma}\label{lem:res-trans}
Let $\ns$ be a $k$-step nilspace. Then for every $\alpha\in \tran(\ns)$ and $x\in \ns$, the restriction of $\alpha$ to the class of $\sim_{k-1}$ containing $x$ is the local translation $\phi_{x,\alpha(x)}$.
\end{lemma}
\begin{proof}
Suppose that $y\sim_{k-1} x$, so that there exists $\q\in \cu^k(\ns)$ such that $\q(1^k)=y$ and $\q(v)=x$ otherwise. Consider the corner  $\q':\{0,1\}^{k+1}\setminus\{1^{k+1}\}\to \ns$ defined by letting the restriction of $\q'$ to $\{v\sbr{k}=0\}$ be equal to $\q$ and the restriction of $\q'$ to $\{v\sbr{k}=1\}\setminus\{1^k\}$ be equal to $\alpha(\q)$. Then since the arrow $\arr{\q,\alpha\co\q}_1$ is a $(k+1)$-cube, the unique way to complete $\q'$ is to set the value at $1^{k+1}$ equal to $\alpha(y)$. This means that, letting $\sim$ denote the relation from Definigion \ref{defn:keyrelation}, we have $(x,y)\sim (\alpha(x),\alpha(y))$. Hence, setting $y=x+a$, we must have $\alpha(y)=\alpha(x)+a$, and so by Definition \ref{def:loc-trans} the result follows.
\end{proof}
\noindent Let us turn to the proof of Proposition \ref{prop:TransGroup}. Our first step towards this result is the following lemma relating $\tran_i(\ns)$ to the $i$-th arrow space $\ns\Join_i \ns$ (recall Definition \ref{def:arrowspace}). Note that if $\alpha\in \tran(\ns)$ then for every cube $\q\in \cu^n(\ns)$ the map $\{0,1\}^{n+1}\to \ns$ consisting of $\q$ and $\alpha\co \q$ on opposite faces is a cube. The converse holds as well. The resulting characterization generalizes to translations of any height, as follows.

\begin{lemma}\label{lem:transcriterion}
For $i\geq 1$, a map $\alpha: \ns\to \ns$ is in $\tran_i(\ns)$ if and only if the map $h_\alpha: \ns\to \ns\times \ns$ defined by $h_\alpha(x)=(x,\alpha(x))$ is a morphism into $\ns\Join_i \ns$. In other words, we have $\alpha\in \tran_i(\ns)$ if and only if for every $\q \in \cu^n(\ns)$ the $i$-arrow $\arr{\q,\alpha\co \q}_i$ is in $\cu^{n+i}(\ns)$.
\end{lemma}

\begin{proof}
If $\alpha\in \tran_i(\ns)$, and $\q\in \cu^n(\ns)$, then $\arr{\q, \alpha\co \q}_i(v,w)=\alpha^F(\q')(v,w)$, where $F$ is the face of $\{0,1\}^{n+i}$ defined by $w=1^i$ (so $\codim(F)=i$) and $\q'$ is the $(n+i)$-cube $\q\co \phi$ where $\phi:\{0,1\}^{n+i}\to \{0,1\}^n$ is the morphism $(v,w)\mapsto v$. Hence $\arr{\q, \alpha\co \q}_i$ is an $(n+i)$-cube on $\ns$ as required.\\
\indent For the converse, suppose that $h_\alpha$ is a morphism into $\ns\Join_i \ns$. Given $n\geq i$, let $F_0$ denote the face $\{v\in \{0,1\}^n: v\sbr{j}=1\;\;\forall\, j\in [n-i+1,n]\}$ of codimension $i$, and let $\q\in \cu^n(\ns)$. We want to show that $\alpha^{F_0}(\q)$ is also in $\cu^n(\ns)$.  Let $Q$ be the product cubespace $\{0,1\}^{n-i}\times T_i$ where $T_i$ is the $i$-dimensional tricube. Let $f_1$ be the identity on $\{0,1\}$ and $f_2$ be the function with $f_2(-1)=0,f_2(0)=1,f_2(1)=1$. Let $f=f_1^{n-i}\times f_2^i:Q\to \{0,1\}^n$. This is a surjective morphism with two useful properties: firstly, the $n$-dimensional subcube of $Q$ containing $1^n$ (i.e. $\{0,1\}^n$ viewed as a subcube of $Q$) is mapped by $f$ onto the face $F_0$ (as $f_2$ is constantly $1^i$ on $\{0,1\}^i$); secondly, letting $\id$ denote the identity map on $\{0,1\}^{n-i}$ and $\omega_i$ the outer-point map $\{0,1\}^i\to\{-1,1\}^i$ (recall Definition \ref{def:ope}), we have $\q\co f\co\,(\id\times \omega_i) = \q$. Now let $g=\q\co f$, and let $g':Q\to \ns$ be the map obtained from $g$ by applying $\alpha$ to the values of $g$ on $\{0,1\}^{n-i}\times \{1^i\}$. Then $g'\co\,(\id\times \omega_i) = \alpha^{F_0}(\q)$. Therefore, to show that $\alpha^{F_0}(\q)\in \cu^n(\ns)$ it suffices to show that $g'$ is a morphism from $Q$ to $\ns$.\\
\indent To show this, note first that $g:Q\to \ns$ is already a morphism (using Lemma \ref{lem:tricube-comp} and the product structure on $Q$). Given any cube (or morphism) $\tilde \q:\{0,1\}^m\to Q$, we want to show that $g'\co \tilde \q\in \cu^m(\ns)$. Now $\tilde \q= \q_0\times \q_1$ where $\q_0$ is a cube $\{0,1\}^m\to \{0,1\}^{n-i}$ and $\q_1$ is a cube $\{0,1\}^m\to T_i$. If $\q_1$ is one of the cubes on $T_i$ not having $1^i$ in its image, then $g'\co \tilde \q = g \co\tilde \q$ is indeed a cube on $\ns$ (since $g$ is a morphism). 
If $\q_1$ is the cube on $T_i$ with $1^i$ in its image, then by Definition \ref{def:tricube} we have $\tilde \q=\q_0\times (\trem_{1^i}\co \phi)$ for some morphism $\phi:\{0,1\}^m\to \{0,1\}^i$. The preimage under $\tilde \q$ of $\{0,1\}^{n-i}\times \{1^i\}$ is a face $F\subset \{0,1\}^m$ of codimension $i$, which we can suppose without loss of generality to be the face with last $i$ components equal to 1. Moreover, by definition of $f$ the map $g\co \tilde \q$ gives the same cube $\q':\{0,1\}^{m-i}\to \ns$ when restricted to each of the other faces $\tilde\q^{-1}(\{0,1\}^{n-i}\times \{w\})$, for each $w\in \{0,1\}^i\setminus\{1^i\}$. It follows that $g' \co \tilde\q=\arr{\q',\alpha\co \q'}_i$, which is a cube by our assumption on $h_\alpha$, so we are done.
\end{proof}
\noindent Next, we show that if $\ns$ is a $k$-step nilspace then the cube-preserving property for $h_\alpha: \ns\to \ns\Join_i \ns$ needs to be checked only for $(k+1)$-cubes.
\begin{lemma}\label{lem:trans-k-suffices}
Let $\ns$ be a $k$-step nilspace. A map $\alpha: \ns\to \ns$ is in $\tran_i(\ns)$ if and only if for every $\q \in \cu^{k+1}(\ns)$ we have $\arr{\q,\alpha\co \q}_i\in \cu^{k+1+i}(\ns)$.
\end{lemma}
\begin{proof}
By Lemma \ref{lem:transcriterion} it suffices to check that given any $\q \in \cu^n(\ns)$ we have $\q'=\q \times (\alpha\co \q)\in \cu^n(\ns\Join_i\ns)$. By Lemma \ref{lem:cubechar} it suffices to check that any restriction of $\q'$ to a $(k+1)$-dimensional face $F$ containing some $v$ with $v\sbr{k+1}=0$ is a $(k+1)$-cube. Now any such restriction is of the form $\q_0 \times (\alpha\co\q_0)$, where $\q_0\in \cu^{k+1}(\ns)$ is the restriction of $\q$ to $F$. This restriction is in $\cu^{k+1}(\ns\Join_i \ns)$ by definition if $\arr{\q_0,\alpha\co\q_0}_i\in \cu^{k+1+i}(\ns)$, which is the case by our assumption.
\end{proof}

We can now establish the first main claim of Proposition \ref{prop:TransGroup}.
\begin{lemma}\label{lem:transisgroup}
Let $\ns$ be a $k$-step nilspace. Then $\tran_i(\ns)$ is a group for every $i\geq 1$.
\end{lemma}
\noindent Since every translation in $\tran(\ns)$ is a nilspace morphism from $\ns$ to itself, with Lemma \ref{lem:transisgroup} we then have that $\tran(\ns)$ is a subgroup of the group $\aut(\ns)$ of nilspace automorphisms of $\ns$.
\begin{proof}
It is clear that composition of two elements of $\tran_i(\ns)$ is still in $\tran_i(\ns)$, so we just need to check invertibility.\\
\indent First we show, by induction on $k$, that any translation is an invertible map. For $k=1$ this is clear since in this case a translation is just addition of some fixed group-element. Now suppose that invertibility holds for $k-1$ and let $\alpha$ be a translation on $\ns$. By Lemma \ref{lem:res-trans}, $\alpha$ preserves the classes of $\sim_{k-1}$, so $\alpha$ induces a well-defined map on $\mathcal{F}_{k-1}(\ns)$, namely $\pi_{k-1}(x)\mapsto \pi_{k-1}\co\alpha(x)$. We denote this map by $h(\alpha)$.  If $\q'$ is an $n$-cube on $\mathcal{F}_{k-1}(\ns)$ and $F$ is a face of codimension 1, then letting $\q\in \cu^n(\ns)$ satisfy $\pi_{k-1}\co \q = \q'$, we have $\alpha^F (\q)\in \cu^n(\ns)$, so $h(\alpha)^F (\q') = \pi_{k-1} \co \alpha^F(\q)$ is in $\cu^n(\mathcal{F}_{k-1}(\ns))$, hence $h(\alpha)$ is a translation on $\mathcal{F}_{k-1}(\ns)$. By induction, this has an inverse $h(\alpha)^{-1}$. Now fix any $y\in \ns$ and let $F_2$ denote the $\sim_{k-1}$ class containing $y$. Let $F_1$ denote the class mapped onto $F_2$ by $\alpha$, namely $F_1= \pi_{k-1}^{-1}\big(h(\alpha)^{-1} \pi_{k-1}(y)\big)$. Since $h(\alpha)$ is bijective, any $x$ satisfying $\alpha(x)=y$ must lie in $F_1$.  Lemma \ref{lem:res-trans} implies that $\alpha$ is a bijection from $F_1$ to $F_2$, so there must be a unique such $x$ in $F_1$.

Finally, note that the inverse of an $i$-translation $\alpha$ is an $i$-translation. This can be shown by induction on $i$. For $i>1$, let $F$ be the face $\{v\in \{0,1\}^n:v\sbr{j}=1,\;\forall\, j\in[n-i+1,n]\}$. It suffices to show that ${\alpha^{-1}}^F$ is cube-preserving. By inclusion-exclusion the  indicator-function $1_F$ on $\{0,1\}^n$ can be written as a linear  combination $1_F=\sum_j \lambda_j 1_{F_j}$ where $\lambda_j\in \{-1,1\}$ and $F_j$ are  faces of codimension at most $i-1$. Each map ${\alpha^{-1}}^{F_j}$ is cube-preserving, by induction, and so their composition ${\alpha^{-1}}^F$ is also cube-preserving.
\end{proof}
\noindent Note that the map $h$ in the above proof is a homomorphism from $\tran(\ns)$ to $\tran(\cF_{k-1}(\ns))$. We can now establish the filtration property.
\begin{lemma}
For every $i,j\geq 1$ we have $[\tran_i(\ns),\tran_j(\ns)]\subset \tran_{i+j}(\ns)$.
\end{lemma}
(Note that setting $\tran_0(\ns)=\tran_1(\ns)$ the lemma then holds for all $i,j\geq 0$.)
\begin{proof}
Let $\alpha_1$ be in $\tran_i(\ns)$ and $\alpha_2\in\tran_j(\ns)$. Let $F$ be a face in $\{0,1\}^n$ of codimension $i+j$, thus $F$ is defined by a choice of $i+j$ coordinates that are each fixed to be 0 or 1. We can write $F=F_1\cap F_2$ where $F_1$ is a face of codimension $i$ and $F_2$ is a face of codimension $j$. We then have $[\alpha_1^{F_1},\alpha_2^{F_2}]={\alpha_1^{F_1}}^{-1}{\alpha_2^{F_2}}^{-1}\alpha_1^{F_1}\alpha_2^{F_2}=[\alpha_1,\alpha_2]^F$. Therefore, for every $\q\in \cu^n(\ns)$ we have $[\alpha_1,\alpha_2]^F(\q)\in \cu^n(\ns)$ and so $[\alpha_1,\alpha_2]^F\in \tran_{i+j}(\ns)$.
\end{proof}
\begin{corollary}
If $\ns$ is a $k$-step nilspace then $(\tran_i(\ns))_{i\geq 0}$ is a filtration of degree $k$.
\end{corollary}
\begin{proof}
The group $\tran_{k+1}(\ns)$ consists of maps $\alpha: \ns \to \ns$ such that given $\q\in \cu^{k+1}(\ns)$ and the $(k+1)$-codimensional face $F=\{1^{k+1}\}$ we have $\alpha^F( \q)\in \cu^{k+1}(\ns)$. For every $x\in \ns$, the cube $\q\in \cu^{k+1}(\ns)$ with constant value $x$ is the unique completion of the restriction of $\alpha^F(\q)$ to $\{0,1\}^{k+1}\setminus F$, so we must have $\alpha(x)=x$. Hence $\alpha$ is the identity.
\end{proof}
This completes the proof of Proposition \ref{prop:TransGroup}.\\ \vspace{-0.3cm}

\noindent We conclude this subsection with a few additional remarks on translation groups.

Firstly, one may wonder whether there is a simple relationship between the groups $\tran_i(\ns)$ and the groups $\ab_i$ describing the bundle structure of $\ns$. For general $k$-step nilspaces, it is not hard to establish such a relation in the case of the group $\tran_k(\ns)$.

\begin{lemma}\label{lem:transfact2}
Let $\ns$ be a $k$-step nilspace. Let $\tau$ denote the map sending $z\in \ab_k$ to the map $\tau_z:\ns\to\ns$, $x\mapsto x+z$. Then we have 
$\tran_k(\ns) = \tau(\ab_k)$.
\end{lemma}
\begin{proof}
To see that $\tran_k(\ns)\subset\tau(\ab_k)$, note first that if $\alpha\in \tran_k(\ns)$ then by \eqref{eq:transfact12} we have $\alpha(x)\sim_{k-1} x$ for all $x\in \ns$. Combined with Lemma \ref{lem:res-trans}, this implies that for each fibre $\pi_{k-1}^{-1}(y)$, $y\in \ns_{k-1}$, there exists $z_y\in \ab_k$ such that on this fibre $\alpha$ is the map $x\mapsto x+z_y$ . Now if $y, y'$ are distinct points in $\ns_{k-1}$ then for any $x,x'$ in the fibres $\pi_{k-1}^{-1}(y)$, $\pi_{k-1}^{-1}(y')$ respectively, note that the map $\q:\{0,1\}^{k+1}\to \ns$ with $\q(\cdot,0)=x$ and $\q(\cdot,1)=x'$ is in $\cu^{k+1}(\ns)$ (by the ergodicity and composition axioms). The map $\q'$ obtained from $\q$ by changing $\q(0^{k+1})$ to $x+z_y$ and $\q(0^k,1)$ to $x'+z_{y'}$ is still a cube, since it is $\alpha^F(\q)$ for the $k$-codimensional face $F=\{0^{k+1}, (0^k,1)\}$. Then, since $\pi_{k-1}\co \q=\pi_{k-1}\co \q'$, we must have $z_y=z_{y'}$ (since $\ns$ is a degree-$k$ bundle and so $\sigma_{k+1}(\q-\q')=0$). Hence $\alpha\in \tau(\ab_k)$. To see that $\tau(\ab_k) \subset \tran_k(\ns)$, note that if $\alpha=\tau(z)$, then for every $\q\in \cu^{k+1}(\ns)$ we have the $k$-arrow identity $\arr{\q,\alpha\co\q}_k=\arr{\q,\q}_k+\arr{0,z}_k$. This is in $\cu^{2k+1}(\ns)$, by \eqref{eq:k-deg-bund}, because $\arr{\q,\q}_k\in \cu^{2k+1}(\ns)$ and $\arr{0,z}_k\in \cu^{2k+1}(\cD_k(\ab_k))$. Hence $\alpha\in \tran_k(\ns)$ by Lemma \ref{lem:transcriterion}. 
\end{proof}
\noindent For general nilspaces the relationship between other translation groups and structure groups is less easy to describe. In the topological part of the theory, we shall be able to give such a description under some natural topological assumptions (see \cite[Proposition 2.9.20]{Cand:Notes2}).

A second remark is that one may distinguish, among all cubes on a nilspace, those that arise from the action of the translation groups, in the following sense.

\begin{defn}\label{def:transequivcubes} Two cubes $\q_1,\q_2\in \cu^n(\ns)$ are \emph{translation equivalent} if there are translations $\alpha_1\in \tran_{i_1}(\ns),\ldots, \alpha_\ell\in \tran_{i_\ell}(\ns)$, and a face $F_j$ of codimension $i_j$ for each $j\in [\ell]$, such that $\q_2=\alpha_\ell^{F_\ell}\cdots \alpha_1^{F_1} (\q_1)$. A cube is called a \emph{translation cube} if it is translation equivalent to a constant cube.
\end{defn}
\noindent In other words, letting $G=\tran(\ns)$ and $G_\bullet$ be the filtration with $G_i=\tran_i(\ns)$, we have by Definition \ref{def:G-cubes} that $\q_1,\q_2\in \cu^n(\ns)$ are translation equivalent if and only if there is $\q\in \cu^n(G_\bullet)$ such that $\q_2(v)=\q(v)(\q_1(v))$ for all $v\in\{0,1\}^n$.
\begin{remark}\label{rem:transcubes}
As indicated in Section \ref{sec:Motiv}, all the examples of nilspaces $\ns$ seen up to that section have the property that all cubes on $\ns$ are  translation cubes. This is not the case for a general nilspace. Indeed, it follows from the ergodicity axiom that if every cube on $\ns$ is a translation cube then the action of $\tran(\ns)$ on $\ns$ is transitive. There are examples of nilspaces for which this transitivity does not hold, see \cite[Example 6]{HKparas}.
\end{remark}
\noindent In the remainder of this chapter we gather additional algebraic tools. Apart from the valuable information that these tools can provide on a  nilspace, they are also very useful in the topological part of the theory.
\medskip

\section{Additional tools}

\subsection{Nilspace morphisms as bundle morphisms}

Theorem \ref{thm:bundle-decomp} describes a $k$-step nilspace as a degree-$k$ bundle. In light of this, it is natural to ask for a description of how a morphism between two $k$-nilspaces relates the bundle structures. The first result of this section, Proposition \ref{prop:nilmorph-bundlemorph}, provides such a description. To formulate it we need the following concept.

\begin{defn}[Bundle morphism]\label{def:bundmorph} Let $\bnd$ and $\bnd'$ be two $k$-fold abelian bundles with factors $\bnd_i, \bnd'_i$ ($i\in [0, k]$), structure groups $\ab_i, \ab_i'$ and projections $\pi_i,\pi_i'$ ($i\in [k]$). A \emph{bundle morphism} from $\bnd$ to $\bnd'$ is a map $\psi:\bnd\rightarrow \bnd'$ satisfying the following properties:
\begin{enumerate}
\item For every $i\in [k]$, if $\pi_i(x)=\pi_i(y)$ then $\pi_i(\psi(x))=\pi_i(\psi(y))$. Thus $\psi$ induces well defined maps $\psi_i:\bnd_i\rightarrow \bnd'_i$.
\item For every $i\in [k]$ there is a map $\alpha_i:\ab_i\rightarrow \ab_i'$ such that for every $x\in \bnd_i$ and $a\in \ab_i$ we have $\psi_i(x+a)=\psi_i(x)+\alpha_i(a)$.
\end{enumerate}
The maps $\alpha_i$ are called the \emph{structure morphisms} of $\psi$. We say that $\psi$ is \emph{totally surjective} if all its structure morphisms are surjective.
\end{defn}
Note that the condition required of the maps $\alpha_i$ implies that they  are group homomorphisms.

\begin{proposition}\label{prop:nilmorph-bundlemorph}
Let $\ns, \ns'$ be $k$-step nilspaces. Then a nilspace morphism $\ns \to \ns'$ is a bundle morphism between the corresponding $k$-degree bundles.
\end{proposition}

\begin{proof}
Condition (i) from Definition \ref{def:bundmorph} holds, for if $x\sim_i y$ then there is an $(i+1)$-cube $\q$ equal to $x$ everywhere except at one vertex where it equals $y$, and then $\psi\co\q$ is a cube equal to $\psi(x)$ everywhere except at one vertex where it equals $\psi(y)$, so $\psi(x)\sim_i \psi(y)$.

For condition (ii), we first consider the case where $\ns$ and $\ns'$ are $i$-fold ergodic and therefore degree-$i$ abelian torsors (by Proposition \ref{prop:k-erg}). By Lemma \ref{lem:nilspace-dif} and Proposition \ref{prop:abel}, if we fix $x\in \ns$ and apply $\partial_x^{i-1}$ and $\partial_{\psi_i(x)}^{i-1}$ to $\ns,\ns'$  respectively, the resulting nilspaces are isomorphic to $\cD_1(\ab_i),\cD_1(\ab_i')$ respectively, for some abelian groups $\ab_i,\ab_i'$. Since $\psi_i$ preserves cubes, it follows that it is an affine homomorphism between the torsors $\ns,\ns'$, which means that there is a homomorphism $\alpha: \ab_i \to \ab'_i$ such that for each $x\in \ns$ we have $\psi_i(x+a)-\psi_i(x)=\alpha(a)$. 

Now every equivalence class $F$ of $\sim_{i-1}$ in $\bnd_i$ is isomorphic to such a degree-$i$ abelian torsor $\cD_i(\ab_i)$ and so $\psi_i$ restricted to $F$ satisfies $\psi_i(x+a)=\psi_i(x)+\alpha_F(a)$, for some homomorphism $\alpha_F: \ab_i\to \ab_i'$ which could depend on $F$. We claim that $\alpha_F$ is in fact independent of $F$, modulo the isomorphism from Lemma \ref{lem:fibreiso}. Indeed, if $F_1,F_2$ are two classes in $\ns$ and we have $(x,x+a)\sim (y,y+a)$ (where $\sim$ is the relation from Definition \ref{defn:keyrelation}, so $(x,y)\sim_{i-1}(x+a,y+a)$ in $\ns \Join_1 \ns$), then since $\sim$ is defined in terms of cubes this relation is preserved by $\psi_i$, so we have $(\psi_i(x),\psi_i(x)+\alpha_{F_1}(a))\sim (\psi_i(y),\psi_i(y)+\alpha_{F_2}(a))$, which implies that $\alpha_{F_1}(a)$ and $\alpha_{F_2}(a)$ are indeed identified by the isomorphism from Lemma \ref{lem:fibreiso}.
\end{proof}
\noindent Recall from Definition \ref{def:AbelianBundle} the notation $\pi_{i,j}$ for the projection $\bnd_i\to\bnd_j$, $i\geq j$.
\begin{defn}[Sub-bundle]\label{def:sub-bund}
Let $\bnd$ be a $k$-fold abelian bundle with factors $\bnd_0,\bnd_1,\dots,\bnd_k=\bnd$, structure groups $\ab_1,\ab_2,\dots, \ab_k$, and projections $\pi_1,\pi_2,\dots,\pi_k$. A \emph{sub-bundle} of $\bnd$ is a $k$-fold abelian bundle $\bnd'$ with factors $\bnd'_0=\bnd_0,\bnd'_1\subseteq \bnd_1,\ldots,\bnd_k'\subseteq \bnd_k$, structure groups $\ab_i'\leq \ab_i$, $i\in [k]$, and with each projection $\pi_{i,i-1}'$ being the restriction of $\pi_{i,i-1}$ to $\bnd_i'$ and satisfying the following condition: for every $x\in \bnd_i'$ we have ${\pi'}^{\,-1}_{i,i-1}(\pi_{i,i-1}'(x)) = x+\ab_i'$, equivalently $\{z\in \ab_i:x+z\in \bnd_i'\}=\ab_i'$.
\end{defn}

\noindent In particular, if $k=1$ then a sub-bundle is an orbit of $\ab_1'$ inside a principal homogeneous space of $\ab_1$.\\

\noindent Given a bundle morphism $\psi: \bnd \to \bnd'$, we shall now define a certain relative abelian bundle on $\bnd'$ that generalizes in a natural way the concept of the kernel of a homomorphism between abelian groups. Let us first discuss briefly the intuition that leads to this definition.\\
\indent Given abelian groups $\ab,\ab'$ and a surjective homomorphism $\psi:\ab\to\ab'$, the kernel $\ker \psi$ is the preimage $\psi^{-1}(\{0\})$, and  $\ab$ is the disjoint union of cosets of $\ker \psi$, these cosets being the  preimages $\psi^{-1}(t)$, $t\in \ab'$. We can thus view $\ab$ as an abelian bundle over $\ab'$ with structure group $\ker \psi$. When we consider more generally a bundle morphism $\psi:\bnd\to \bnd'$, defining the kernel of $\psi$ as the preimage of a single point as in the abelian case is not clear since typically there is no distinguished point on $\bnd'$ playing the role of $0_{\ab'}$. However, one can still view $\bnd$ as a \emph{relative} bundle with base $\bnd'$ (recall Definition \ref{def:AbelianBundle}), and then restrict this bundle structure to each preimage $\psi^{-1}(t)$, $t\in \bnd'$, to obtain coherent bundle structures on these preimages. This is what we shall establish formally in the next two lemmas. (Recall from Definition \ref{def:bundmorph} the induced maps $\psi_i$ and the structure morphisms $\alpha_i$.)\\

\begin{defn}[Kernel of a bundle morphism]\label{defn:bundmorphker}
Let $\bnd,\bnd'$ be $k$-fold abelian bundles, and let $\psi:\bnd\to \bnd'$ be a totally surjective bundle morphism. The \emph{kernel} of $\psi$ is the relative $k$-fold abelian bundle denoted by $K$, with structure groups $\{\ker(\alpha_i)\}_{i=1}^k$, defined as follows. First, for each $i\in [k]$ we define $K_i=\{(x,y)\in \bnd_i\times \bnd' : \psi_i(x)=\pi_i(y)\}$.
We then define, for each $i\in [k]$, an action of $\ker(\alpha_i)$ on $K_i$ by $(x,y)+a=(x+a,y)$. Finally, the projection $\pi_i:K_j\to K_i$ is defined by $\pi_i(x,y)=(\pi_i(x),y)$ for $j\geq i$.
\end{defn}
\noindent Note that we can identify $K=K_k$ with $\bnd_k$ using the bijection $(x,y)\leftrightarrow x$, and $K_0$ can be identified with $\bnd'$ using the bijection $(x,y)\leftrightarrow y$. These bijections enable us to identify $\pi_0:K\to K_0$ with $\psi: \bnd\to \bnd'$.\\
\indent Let us confirm that this definition makes $K$ a relative abelian bundle.
\begin{lemma}
For each $i\in [k]$, $K_i$ is a $\ker(\alpha_i)$-bundle over $K_{i-1}$ with bundle map $\pi_{i-1}$.
\end{lemma}
\begin{proof}
First we check that $\pi_{i-1}$ is surjective. Let $(x,y)\in K_{i-1}$, so $\psi_{i-1}(x)=\pi_{i-1}(y)$.  Let $z\in \bnd_i$ be such that $\pi_{i-1}(z)=x$. We have $\pi_{i-1}(\psi_i(z))=\psi_{i-1}(\pi_{i-1}(z))=\psi_{i-1}(x)=\pi_{i-1}(y)$, so there exists $a'\in \ab_i'$ such that $\psi_i(z)+a'=\pi_i(y)$. Since $\alpha_i$ is surjective by assumption, there is $a\in \ab_i$ with $\alpha_i(a)=a'$. Then the pair $(z+a,y)$ is in $K_i$ and maps to $(x,y)$ under $\pi_{i-1}$.   

Clearly $\ker(\alpha_i) \subset \ab_i$ acts freely on $K_i$, so it remains to check that it acts transitively on the fibres of $\pi_{i-1}$. Let $(x_1,y)$ and $(x_2,y)$ be any elements in the same fibre of $\pi_{i-1}:K_i\to K_{i-1}$. Since $\pi_{i-1}(x_1) = \pi_{i-1}(x_2)$ there is $a\in \ab_i$ with $x_1=x_2+a$. Then $\pi_i(y)=\psi_i(x_1)=\psi_i(x_2+a)=\psi_i(x_2)+\alpha_i(a)=\pi_i(y)+\alpha_i(a)$, whence $a\in\ker(\alpha_i)$. 
\end{proof}

\noindent With the bijections $K \leftrightarrow \bnd$ and $K_0 \leftrightarrow \bnd'$ described above, we can thus view $\bnd$ as a bundle over $\bnd'$. Note that no set $K_i$ is contained in $\bnd_i$, we just have a map $(x,y)\mapsto x$ from $K_i$ to $\bnd_i$. Now a key fact is that, while this map is not a bijection, if we fix a single $t\in \bnd'$  then the map restricts to a bijection identifying the set $\{(x,y)\in K_i: y=t\}$ with $\psi_i^{-1}(\pi_i(t))$. From this it follows promptly that $\psi^{-1}(t)$ has a $k$-fold abelian-bundle structure, inherited from $K$ via these bijections (thus with structure groups $\ker \alpha_i$), and this makes $\psi^{-1}(t)$ a sub-bundle of $\bnd$. This gives a useful description of a preimage $\psi^{-1}(t)$, which we record as follows.

\begin{lemma}\label{lem:preimbund}
Let $\bnd,\bnd'$ be $k$-fold abelian bundles and let $\psi:\bnd\to \bnd'$ be a totally surjective bundle morphism. Then for any $t\in \bnd'$ the preimage $\psi^{-1}(t)$ is a sub-bundle of $\bnd$, with factors $\psi_i^{-1}(\pi_i(t))$ and structure groups $\ker(\alpha_i)$, $i\in [k]$.
\end{lemma}
\begin{proof}
One way to prove this was described above, namely by restricting the bundle structure from $K$ to $\psi^{-1}(t)$ via the bijections $(x,y)\mapsto x$. This way has the advantage of showing how all these preimages live inside the bundle $K$.

Alternatively, one may check directly that the stated factors and structure groups do indeed make $\psi^{-1}(t)$ a sub-bundle of $\bnd$, essentially by  repeating the arguments in the previous proof.
\end{proof}

\subsection{Fibre-surjective morphisms and restricted morphisms}

This section describes specific kinds of morphisms between nilspaces. This material is useful mainly in the topological part of the theory, but since the material itself is purely algebraic we include it here. 

To begin with, recall that on one hand, by Proposition \ref{prop:nilmorph-bundlemorph}, a morphism of nilspaces is a bundle morphism of the corresponding abelian bundles, and on the other hand, for abelian bundles we have a notion of totally-surjective bundle morphisms (recall Definition \ref{def:bundmorph}). The following definition gives a  counterpart of this notion for nilspaces, as shown in Lemma \ref{lem:fib-surj=tot-surj}.

\begin{defn}[Fibre-surjective morphism]\label{def:fibsurjmorph}
Let $\ns,\ns'$ be nilspaces. A morphism $\psi:\ns \to \ns'$ is said to be \emph{fibre surjective} if for every $n\in \N$ the image of an equivalence class of $\sim_n$ in $\ns$ is a full equivalence class of $\sim_n$ in $\ns'$.
\end{defn}
\noindent Note that if $\ns$ is $k$-step then $\ns'$ must also be $k$-step. Indeed, if $\q'$ is a $(k+1)$-corner on $\ns'$, then any two completions $\q_1,\q_2$ of $\q'$ satisfy $\q_1(1^{k+1})\sim_k \q_2(1^{k+1})$, so by surjectivity there must be a class of $\sim_k$ in $\ns$ in which some element is sent to $\q_1(1^{k+1})$ and another is sent to $\q_2(1^{k+1})$ by $\psi$. Since $\ns$ is $k$-step this class is a singleton, so $\q_1(1^{k+1})=\q_2(1^{k+1})$.

\begin{lemma}\label{lem:fib-surj=tot-surj}
Let $\ns,\ns'$ be $k$-step nilspaces. A map $\psi:\ns \to \ns'$ is a fibre-surjective morphism if and only if it is a totally surjective morphism between the corresponding $k$-fold abelian bundles.
\end{lemma}

\begin{proof}
A totally surjective bundle morphism between $k$-step nilspaces is fibre-surjective by definition. To see the converse, note first that by Proposition \ref{prop:nilmorph-bundlemorph} we know that $\psi$ is a bundle morphism. Then, since the orbits of the groups $\ab_i,\ab_i'$ are  equivalence classes of $\sim_i$, the condition for fibre-surjectivity implies that all the structure morphisms from Definition \ref{def:bundmorph} are surjective.
\end{proof}
\noindent The main result concerning fibre-surjective morphisms here is the following lemma stating that they are actually factor maps. More precisely, letting $\sim$ denote the equivalence relation on $\ns$ defined by $x\sim y$ if and only if $\psi(x)=\psi(y)$, then $\ns/\sim$ with the quotient cubespace structure is isomorphic to $\ns'$. Thus $\ns'$ can be viewed as a factor of $\ns$ in the sense of Definition \ref{def:factor}.

\begin{lemma}\label{lem:lifting2}
Let $\ns,\ns'$ be $k$-step nilspaces and let $\psi:\ns\to \ns'$ be a fibre-surjective morphism. Then every cube $\q'\in \cu^n(\ns')$ can be lifted to a cube $\q\in \cu^n(\ns)$ such that $\psi \co \q = \q'$. In other words,  $\ns'$ is a factor nilspace of $\ns$.
\end{lemma}
\begin{proof}
We argue by induction on $k$. For $k=0$ the claim is clear (a $0$-step non-empty nilspace is a 1-point space). Let $k>0$ and suppose that the claim holds for $k-1$. Viewing $\psi$ as a bundle morphism, we see that it  induces a fibre-surjective map $\psi'$ from $\cF_{k-1}(\ns)$ to $\cF_{k-1}(\ns')$. By induction, there exists $\q_1\in \cu^n(\cF_{k-1}(\ns))$ such that $\psi'\co \q_1=\pi_{k-1}\co \q'$. Lifting $\q_1$ to a cube $\q_2\in \cu^n(\ns)$ (using Lemma  \ref{lem:lifting}), the fact that $\psi$ preserves the fibres of $\sim_{k-1}$ implies that $\pi_{k-1}\co\psi\co\q_2=\q'$. Hence, the map $\q_3=\psi\co\q_2-\q'$ is in $\cu^n(\cD_k(\ab_k'))$. Now, letting $\alpha_k:\ab_k\to \ab_k'$ be the surjective homomorphism, note that if there was a cube $\q_4\in \cu^n(\cD_k(\ab_k))$ such that of $\alpha_k\co \q_4=\q_3$, then we would have $\psi\co(\q_2-\q_4)=\psi\co\q_2-\alpha_k\co\q_4=\q'+\q_3-\alpha_k\co\q_4=\q$, and so we would have a lift $\q=\q_2-\q_4$ as desired. To see that $\q_4$ exists, note that this is trivial for $n\leq k$ (by $k$-ergodicity of degree-$k$ abelian groups), and for $n\geq k+1$ note that we can take an $n$-corner of $\q_3$, lift this under $\alpha_k$ to an $n$-corner on $\cD_k(\ab_k)$ (by induction on $n$), and then complete it (uniquely) to an $n$-cube $\q_4$, which then has to satisfy $\alpha_k\co \q_4=\q_3$.
\end{proof}

We now move on to another specific type of morphism.

\begin{defn}[Restricted morphism]
Let $P$ and $\ns$ be cubespaces, let $S$ be a subcubespace of $P$, and let $f:S \to \ns$ be an arbitrary function. We define the \emph{restricted morphism set} $\hom_f(P,\ns)$ to be the set of morphisms $P\to \ns$ whose restriction to $S$ equals $f$.
\end{defn}
\noindent For instance, given an abelian group $\ab$, the set $\hom_{S\to 0}(P,\cD_k(\ab))$ (used below) is the set of morphisms $\phi$ from $P$ to  $\cD_k(\ab)$ such that $\phi(S)=0_{\ab}$. Note that this set is itself an abelian group under pointwise addition in $\ab$.

We shall now describe $\hom_f(P,\ns)$ as a bundle.

\begin{lemma}\label{lem:restrmorph=subbund}
Let $P$ be a subcubespace of $\{0,1\}^n$ with the extension property, let $S$ be a subcubespace of $P$ with the extension property in $P$, let $\ns$ be a $k$-step nilspace and $f:S\to \ns$ be a morphism. Then $\hom_f(P,\ns)$ is a $k$-fold abelian bundle that is a sub-bundle of $\ns^P$ with factors $\hom_{\pi_i\co f}(P,\ns_i)$ and structure groups $\hom_{S\to 0}(P,\cD_i(\ab_i))$, where $\ab_i$ is the $i$-th structure group of $\ns$.
\end{lemma}
\begin{proof}
We argue by induction on $k$ to check the condition in Definition \ref{def:sub-bund}. For $k=1$ we know that $\ns$ is the principal homogeneous space of an abelian group $\ab$, and then $\hom_f(P,\ns)$ is a principal homogeneous space of $\hom_{S\to 0}(P,\cD_1(\ab))$ in $\ns^P$.\\
\indent For $k>1$ suppose that $\ns$ is a $k$-step nilspace and the result is already established for $\ns_{k-1}=\cF_{k-1}(\ns)$. Let $f_2:P\to \ns_{k-1}$ be a morphism with restriction $f_2|_S=\pi_{k-1}\co f$. We claim that $f_2$ can be lifted to an element $f_3\in \hom_f(P,\ns)$.\\
\indent Indeed, by Lemma \ref{lem:lifting}, there is a lift $g:P\to \ns$ of $f_2$. Then the function $g_2=f-g|_S$ is a morphism from $S$ to $\cD_k(\ab_k)$. By the extension property there is a morphism $g_3:P\to \cD_k(\ab_k)$ extending $g_2$. Then $f_3=g+g_3$ is in $\hom_f(P,\ns)$ and is a lift of $f_2$, as claimed.\\
\indent Note that any other lift $f_3'$ must satisfy $f_3-f_3'\in \hom_{S\to 0}(P,\cD_k(\ab_k))$, and it follows that the set of possible lifts of $f_2$ is exactly $f_3+\hom_{S\to 0}(P,\cD_k(\ab_k))$, as required. 
\end{proof}
We end this section with a result relating various fibre-surjective morphisms.
\begin{lemma}\label{lem:collection} Let $P\subset\{0,1\}^n$ be a subcubespace with the extension property in $\{0,1\}^n$ and $S\subset P$ be a subcubespace with the extension property in $P$. Let $\ns,\ns'$ be $k$-step nilspaces and let $\psi:\ns \rightarrow \ns'$ be a fibre-surjective morphism with structure morphisms $\alpha_i:\ab_i\to\ab_i'$ \textup{(}recall Definition \ref{def:bundmorph}\textup{)}. Then we have
\begin{enumerate}
\item $\hom(P,\ns)$ is a sub-bundle of $\ns^P$ with structure groups $\hom(P,\cD_i(\ab_i))\leq \ab_i^P$.
\item $\psi^P:\hom(P,\ns)\to\hom(P,\ns')$ is a totally surjective bundle morphism with structure morphisms $\alpha_i^P:\hom(P,\cD_i(\ab_i))\to\hom(P,\cD_i(\ab_i'))$.
\item The preimage of $t\in\hom(P,\ns')$ under $\psi^P$ is a bundle with structure groups $\hom(P,\cD_i({\rm ker}(\alpha_i)))$.
\item Let $t\in\hom(P,\ns')$ and let $s\in\hom(S,\ns')$ be its restriction to $S$. Then the projection $\pi_S$ from $(\psi^{P})^{-1}(t)$ to $(\psi^{S})^{-1}(s)$ is a totally-surjective bundle morphism.
\end{enumerate}
\end{lemma}

\begin{proof}
Statement $(i)$ is just the special case of Lemma \ref{lem:restrmorph=subbund} with $S=\emptyset$ (note that by Definition \ref{def:ext-property} and the fact that there is always a morphism from $P$ to a non-empty nilspace, we have that $\emptyset$ has the extension property in $P$).

For statement $(ii)$ let us check that $\psi^P$ satisfies the two properties from Definition \ref{def:bundmorph}. For property $(i)$ it suffices to show that the map $\psi^P$ preserves the relation $\sim_i$ from Definition \ref{def:simdef} for each $i$. This is seen by a straightforward argument using Lemma \ref{lem:sim1} (using that $\psi$ preserves cubes). For property $(ii)$, note that we have $\psi_i^P(x+a)=\psi_i^P(x)+\alpha_i^P(a)$ (since $\psi_i$ satisfies this on each  component of $X^P$), so the structure morphisms of $\psi^P$ are indeed the maps $\alpha_i^P$ on $\hom(P,\cD_i(\ab_i))$. To see that each map $\alpha_i^P$ is onto $\hom(P,\cD_i(\ab_i'))$, we use Lemma \ref{lem:lifting2}.

For statement $(iii)$, note that by Lemma \ref{lem:preimbund} we have that $(\psi^P)^{-1}(t)$ is a sub-bundle of $\hom(P,\ns)$ with structure groups $\ker(\alpha_i^P)\cap \hom(P,\cD_i(\ab_i))=\hom(P,\cD_i(\ker(\alpha_i)))$.

Finally, to see statement $(iv)$, we first check from Definition \ref{def:bundmorph} that $\pi_S$ is indeed a bundle morphism. By statement $(iii)$, the structure morphisms of $\pi_S$ are seem to be the restriction maps
$\hom(P,\cD_i({\rm ker}(\alpha_i)))\to\hom(S,\cD_i({\rm ker}(\alpha_i)))$. Since $S$ has the extension property in $P$, these restriction maps are surjective.
\end{proof}

\subsection{Extensions and cocycles}\label{sec:exts&cohom}

This subsection treats an important method of building a new nilspace from an old one. The resulting new nilspace is called an extension, and consists in an abelian bundle over the old nilspace, equipped with a cube structure adequately related to the one on the old nilspace. The formal definition is the following.

\begin{defn}\label{def:extension}
Let $\ns$ be a nilspace. A \emph{degree-$k$ extension} of $\ns$ is an abelian bundle $\nss$ over $\ns$, with structure group $\ab$ and bundle map $\pi:\nss\to\ns$, such that $\nss$ is a cubespace with the following properties:\\ \vspace{-0.7cm}
\begin{enumerate}
\item For every $n\in \N$ the map $\q\mapsto \pi\co \q$ is a surjection $\cu^n(\nss)\to \cu^n(\ns)$. \vspace{-0.2cm}
\item For every $\q_1\in \cu^n(\nss)$ we have\\ \vspace{-0.4cm}
\begin{equation}\label{eq:ext-cube-corresp}
\{\q_2\in \cu^n(\nss):\pi\co\q_2=\pi\co\q_1\}=\{\q_1+\q_3:\q_3 \in \cu^n(\cD_k(\ab))\}.
\end{equation}
\end{enumerate}
The extension $\nss$ is called a \emph{split extension} if there is a (cube-preserving) morphism $m:\ns\to \nss$ such that $\pi\co m$ is the identity map on $\ns$.
\end{defn}
\noindent This notion is motivated by the fact that, by the bundle-decomposition result (Theorem \ref{thm:bundle-decomp}), any $k$-step nilspace can be built up from the one-point space by $k$ consecutive extensions of increasing degree.\\
\indent With each extension one can associate a certain function called a cocycle, which encodes algebraic information about the extension. These cocycles can then be used to parametrize the different extensions of a nilspace. This is very useful in particular in the topological part of the theory, to classify compact nilspaces.\\
\indent To define cocycles we use the following notation. Recall that $\aut(\{0,1\}^k)$ is generated by permutations of $[k]$ and coordinate-reflections. For $\theta\in \aut(\{0,1\}^k)$ we write $r(\theta)$ for the number of reflections involved in $\theta$. Equivalently, $r(\theta)$ is the number of coordinates equal to 1 in $\theta(0^k)$.
\begin{defn}[Cocycle]\label{def:cocycle} 
Let $\ns$ be a nilspace, let $\ab$ be an abelian group, and let $k\geq -1$ be an integer. A \emph{cocycle of degree} $k$ is a function $\rho: \cu^{k+1}(\ns)\to \ab$ with the following two properties:\\ \vspace{-0.6cm}
\begin{enumerate}
\item If $\q \in \cu^{k+1}(\ns)$ and $\theta\in \aut(\{0,1\}^{k+1})$ then $\rho(\q\co\theta )=(-1)^{r(\theta)}\rho(\q)$.\vspace{-0.2cm}
\item If $\q_3$ is the concatenation of adjacent cubes $\q_1,\q_2\in \cu^{k+1}(\ns)$ then $\rho(\q_3)=\rho(\q_1)+\rho(\q_2)$.
\end{enumerate}
The set of $\ab$-valued cocycles of degree $k$ is an abelian group under pointwise addition, denoted by  $Y_k(\ns,\ab)$. Note that $Y_{-1}(\ns,\ab)$ is just $\ab^{\ns}$.
\end{defn}
\noindent In going through the results below on extensions and cocycles, it can be useful to have a concrete case in mind as a source of intuition.
\begin{example}\label{ex:Heisenext}
Consider the Heisenberg nilmanifold $H/\Gamma= \begin{psmallmatrix} 1 & \R & \R\\[0.1em]  & 1 & \R \\[0.1em]  &  & 1 \end{psmallmatrix} / \begin{psmallmatrix} 1 & \Z & \Z\\[0.1em]  & 1 & \Z \\[0.1em]  &  & 1 \end{psmallmatrix}$ from Example \ref{ex:Heisen}. As mentioned there, one can identify $H/\Gamma$ with $[0,1)^3$. Let us write this as $H/\Gamma \cong \begin{psmallmatrix} 1 & [0,1) & [0,1)\\[0.1em]  & 1 & [0,1) \\[0.1em]  &  & 1 \end{psmallmatrix}$. Now, identifying the circle group $\T$ with $[0,1)$, we have that the map \[
\pi: H/\Gamma \to \T^2,\;\; \begin{psmallmatrix} 1 & x_1 & x_3\\[0.1em]  & 1 & x_2 \\[0.1em]  &  & 1 \end{psmallmatrix}\mapsto (x_1,x_2)
\]
is a bundle map showing that $H/\Gamma$ is an abelian bundle with fibre $\T$ over $\T^2$. Moreover, since $H/\Gamma$ is a 2-step nilspace, by the bundle characterization we have that it is a degree-2 extension of $\T^2$ by $\T$ (recall the proof of Lemma \ref{lem:Z-bund-is-k-degbund}). We can then find a cocycle of degree 2 associated with this extension, with the following procedure (which will be treated more formally and generally below). Consider the map $\cs:\T^2\to H/\Gamma$, $(x_1,x_2)\mapsto \begin{psmallmatrix} 1 & x_1 & 0\\[0.1em]  & 1 & x_2 \\[0.1em]  &  & 1 \end{psmallmatrix}$. This is a cross section for this extension, i.e. it satisfies $\pi\co\cs(x)=x$ on $\T^2$ (see Definition \ref{def:crossec}). We then have that for every $x\in H/\Gamma$, the points $x$ and $\cs\co\pi(x)$ are in the same fibre of the map $\pi$. This fibre is an affine version of $\T$, so we can take the difference $\cs\co\pi(x)-x \in \T$. We have thus defined a function $f:H/\Gamma\to \T$, $x\mapsto \cs\co\pi(x)-x$. Now consider the function $\rho$ on $\cu^3(H/\Gamma)$ sending a cube $\q$ to $\sigma_3(f\co\q)=\sum_{v\in \{0,1\}^3} (-1)^{|v|} f(\q(v))$. Using the defining property \eqref{eq:ext-cube-corresp} of the extension, it can be checked that $\rho(\q)$ is unchanged if we vary the values $\q(v)$ within their $\T$-fibres in $H/\Gamma$. More precisely, for every other cube $\q'$ such that $\pi\co\q'=\pi\co\q$, we have $\rho(\q')=\rho(\q)$. (We shall prove this more generally in Lemma \ref{lem:crossecgendef}.) Thus $\rho$ can be viewed as a function on $\cu^3(\T^2)$, and one can check that it is a cocycle (see Lemma \ref{lem:crossecgendef}). Note that, while $\rho$ was defined using the structure of $H/\Gamma$, we end up with a function of cubes on $\T^2$. This observation is important because it suggests that we may be able to go in the opposite direction, that is, that given a cocycle on the cubes of a given nilspace $\ns$ (here $\ns=\T^2$), it may be possible to construct an extension of $\ns$ that has some cross section with associated cocycle equal to $\rho$. This will be confirmed in Proposition \ref{prop:extcompleted} and subsequent results.
\end{example}
\noindent Before we go into the details indicated by the above example, let us describe how, given a cocycle $\rho$ of degree $k$, we can define a new cocycle of degree $k+1$ by taking the difference of $\rho$ on opposite faces of $(k+2)$-cubes. This will lead to the notion of a coboundary.
\begin{defn}
Let $k\in \N$ and let $\rho: \cu^k(\ns)\to \ab$ be a cocycle. We define $\partial \rho: \cu^{k+1}(\ns)\to \ab$ by $\partial\rho (\q)=\rho (\q(\cdot,0)) - \rho (\q(\cdot,1))$.
\end{defn}

\begin{lemma}
Let $k\geq 0$ and let $\rho$ be a cocycle of degree $k-1$. Then $\partial \rho$ is a cocycle of degree $k$.
\end{lemma}

\begin{proof}
We check that properties (i) and (ii) from Definition \ref{def:cocycle} hold. 

Let $\theta$ be an automorphism of $\{0,1\}^{k+1}$ and note that there is a unique $j\in [k+1]$ such that $\theta(v)\sbr{j}=v\sbr{k+1}$ or $1-v\sbr{k+1}$. Suppose first that $\theta(v)\sbr{j}=v\sbr{k+1}$. Let $\theta'$ be the restriction of $\theta$ to $\{0,1\}^k$, with the image of $\theta'$ being the cube obtained from $\{0,1\}^{k+1}$ by omitting the $j$-th coordinate. For $i=0,1$ let $\q_{j,i}$ denote the restriction of $\q$ to the face $v\sbr{j}=i$. We then have
\begin{eqnarray*}
\partial\rho(\q\co \theta) & = & \rho(\q\co \theta (\cdot,0))-\rho(\q \co \theta(\cdot,1))  \hspace{2cm} = \;\; \rho(\q_{j,0}\co \theta')-\rho(\q_{j,1}\co \theta')\\
& = & (-1)^{|\theta'(0^k)|} \rho(\q_{j,0})- (-1)^{|\theta'(0^k)|} \rho(\q_{j,1})
 \;\hspace{0.2cm} = \;\; (-1)^{|\theta(0^{k+1})|} (\rho(\q_{j,0})- \rho(\q_{j,1}))\\
& = & (-1)^{|\theta(0^{k+1})|} \partial \rho(\q).
\end{eqnarray*}
If $\theta(v)\sbr{j}$ was $1-v\sbr{k+1}$, then instead of the expression after the second equality above, we would have  $\rho(\q_{j,1}\co \theta')-\rho(\q_{j,0}\co \theta')=
-(-1)^{|\theta'(0^k)|} (\rho(\q_{j,0})- \rho(\q_{j,1}))=(-1)^{|\theta(0^{k+1})|} \partial\rho (\q)$. This proves (i).\\
\indent For property (ii), suppose that $\q''$ is the concatenation of $\q,\q'\in \cu^{k+1}(\ns)$. Then we have
\[
\partial\rho(\q'')=\rho(\q(\cdot,0))-\rho(\q'(\cdot,1))=\rho(\q(\cdot,0))-\rho(\q(\cdot,1))+\rho(\q'(\cdot,0))-\rho(\q'(\cdot,1))=\partial \rho (\q) +\partial \rho (\q'). \qedhere
\]
\end{proof}
\noindent Note also that if $\rho_1,\rho_2$ are two cocycles of degree $k-1$, then $\partial (\rho_1+\rho_2)=\partial\rho_1+\partial\rho_2$. Thus, for every $k\geq 0$ the map $\partial$ is a homomorphism from $Y_{k-1}(\ns,\ab)$ to $Y_k(\ns,\ab)$.

Recall from Definition \ref{defn:gen-sign} that given an abelian group $\ab$ and a function $f:\{0,1\}^n\to \ab$, we write $\sigma_n(f)$ for $\sum_{v\in \{0,1\}^n} (-1)^{|v|}f(v)$.

\begin{defn}\label{def:cobo}
A \emph{coboundary} of degree $k$ is an element of the abelian group $\partial^{k+1}Y_{-1}(\ns,\ab)$. Equivalently, a coboundary of degree $k$ is an element of $Y_k(\ns,\ab)$ of the form $\q\mapsto \sigma_{k+1}(f\co \q)$ for some function $f:\ns\to \ab$. We then define the abelian group
\begin{equation}
H_k(\ns,\ab) = Y_k(\ns,\ab)\;/\;\partial^{k+1} Y_{-1}(\ns,\ab).
\end{equation}
\end{defn}

\begin{example}
In the case of degree-2 cocycles $\cu^3(\T^2)\to \T$, such as those constructed using the Heisenberg nilmanifold in Example \ref{ex:Heisenext}, we have that two such cocycles $\rho_1,\rho_2$ represent the same element of $H_2(\T^2,\T)$ if and only if the difference map $\rho_1-\rho_2$ is of the form $\q\to \sum_{v\in \{0,1\}^3} (-1)^{|v|} f\co \q(v)$ for some function $f:\T^2\to \T$.
\end{example}
\noindent Our aim now is to show that every degree-$k$ extension of $\ns$ can be represented by an element of $H_k(\ns,\ab)$. This will eventually be achieved in Corollary \ref{cor:classrep}. Our first step in this direction is the following.
\medskip
\subsubsection{Associating an element of $H_k(\ns,\ab)$ with a given degree-$k$ extension}
\medskip
\begin{defn}\label{def:crossec}
Let $\nss$ be a degree-$k$ extension of $\ns$ with structure group $\ab$ and projection $\pi:\nss\to \ns$. A \emph{cross section} for this extension is a map $\cs:\ns\to \nss$ such that $\pi\co \cs$ is the identity on $\ns$.
\end{defn}

\begin{lemma}\label{lem:crossecgendef}
Let $\nss$ be a degree-$k$ extension of $\ns$ with structure group $\ab$ and projection $\pi:\nss\to \ns$, let $\cs$ be a cross section, and  define $f:\nss\to \ab$ by $f(y)=\cs\co \pi(y)-y$. Define $\rho:\cu^{k+1}(\ns)\to \ab$ by $\rho(\q)=\sigma_{k+1}(f\co \q')$ for some $\q'\in \cu^{k+1}(\nss)$ such that $\pi\co \q'=\q$. Then $\rho$ is a cocycle of degree $k$. We shall refer to $\rho$ as the \emph{cocycle generated by the cross section} $\cs$, and denote it $\rho_{\cs}$.
\end{lemma}

\begin{proof}
We first check that $\rho$ is well-defined. Suppose that $\q''$ is another $(k+1)$-cube on $\nss$ such that $\pi\co \q''=\pi\co \q' = \q$. Then by definition of the extension we have $\q''-\q'=\q_2\in \cu^{k+1}(\mathcal{D}_k(\ab))$. Then
\begin{eqnarray*}
\sigma_{k+1}(f\co \q'') & = & \sum_{v\in \{0,1\}^{k+1}} (-1)^{|v|} f\co (\q'+\q_2) (v)\;\;\; = \;\sum_{v\in \{0,1\}^{k+1}} (-1)^{|v|} \big(\;\cs\co\pi(\q'(v))-(\q'(v)+\q_2(v))\;\big)\\
&=& \sigma_{k+1}(f\co \q') - \sigma_{k+1}(\q_2)\;\;\;=\;\;\; \sigma_{k+1}(f\co \q').
\end{eqnarray*}
Let us now check that the cocycle properties hold.

Let $\theta\in \aut(\{0,1\}^{k+1})$. Then
\begin{eqnarray*}
\rho(\q\co \theta)  \;=\;  \sigma_{k+1} (f\co\, (\q\co \theta))
&=& \sum_{v\in\{0,1\}^{k+1}}  (-1)^{|v|-|\theta(v)|}  (-1)^{|\theta(v)|} f(\q \co \theta(v))\\
&=& (-1)^{r(\theta)} \sum_{v\in\{0,1\}^{k+1}} (-1)^{|\theta(v)|} f(\q \co \theta(v)) \;=\;  (-1)^{r(\theta)}\rho(\q),
\end{eqnarray*}
where we have used that $|\,|\theta(v)|-|v|\,|$ is exactly the number of coordinates that $\theta$ switches, which is precisely the number of coordinates equal to 1 in $\theta(0^{k+1})$.\\
\indent The second property follows from the basic property of $\sigma_{k+1}$ observed in \eqref{eq:concat-sgn}.
\end{proof}
Let us now examine how the cocycle $\rho_{\cs}$ varies when we change the cross section $\cs$.
\begin{lemma}\label{lem:cocyvar}
Let $\nss$ be a degree-$k$ extension of $\ns$ with structure group $\ab$ and projection $\pi:\nss\to \ns$, and let $\cs:\ns\to\nss$ be a cross section. Then for every other cross section $\cs'$, we have $\rho_{\cs'}\in \rho_{\cs} + \partial^{k+1}Y_{-1}(\ns,\ab)$. Conversely, every element of $\rho_{\cs} + \partial^{k+1}Y_{-1}(\ns,\ab)$ is a cocycle $\rho_{\cs'}$ for some cross section $\cs'$.
\end{lemma}
\noindent Thus the coset $\rho_{\cs} + \partial^{k+1}Y_{-1}(\ns,\ab)$ is the set of
all cocycles generated by cross sections for the extension $\nss$. This coset is the element of $H_k(\ns,\ab)$ that we associate with the extension $\nss$.
\begin{proof}
Letting $\q'$ be any lift of $\q$, we have
\begin{eqnarray*}
\rho_{\cs'}(\q) & = & \sigma_{k+1}(\cs'\co \pi(\q')-\q')= \sum_{v\in\{0,1\}^{k+1}} (-1)^{|v|} \big(\,\cs'\co\pi(\q'(v))-\q'(v)\,\big)\\
& = & \sum_{v\in\{0,1\}^{k+1}} (-1)^{|v|} [\, \cs'\co\pi(\q'(v))-\cs\co\pi(\q'(v)) + \cs\co\pi(\q'(v))-\q'(v)\,]\\
& = & \sigma_{k+1}(\cs\co \pi(\q')-\q')+\sigma_{k+1}(\cs'\co \q - \cs\co \q)=\rho_{\cs}(\q)+\sigma_{k+1}((\cs'-\cs)\co \q).
\end{eqnarray*}
Since $\cs'-\cs$ is in $Y_{-1}(\ns,\ab)$, the function $\q\mapsto \sigma_{k+1}((\cs'-\cs)\co\q)$ is a coboundary of degree $k$ as claimed.

Conversely, given any coboundary $\q \mapsto \sigma_{k+1}(f\co \q)$ for some $f:\ns \to \ab$, and a cocycle $\rho_{\cs}$, we have $\rho_{\cs}(\q)+\sigma_{k+1}(f\co \q)=\sigma_{k+1}((\cs+f)\co\pi(\q')-\q')=\rho_{\cs+f}$, where $\cs+f:x\mapsto \cs(x)+f(x)$ is indeed a cross-section (where $+$ denotes the action of $\ab$ on $\nss$).
\end{proof}

\noindent Thus, every degree-$k$ extension $\nss$ of $\ns$ with structure group $\ab$ has an associated element $\rho_{\cs} + \partial^{k+1}Y_{-1}(\ns,\ab)$ of $H_k(\ns,\ab)$, which is the same well-defined element for any choice of cross section $\cs$.
\begin{defn}
We say that two degree-$k$ extensions of $\ns$ with structure group $\ab$ are \emph{equivalent} if they correspond to the same element of $H_k(\ns,\ab)$.
\end{defn}
\noindent Note that if $\nss$ is a split extension, with cube-preserving cross section $\cs$, then $\rho(\q)=\sigma_{k+1}(\cs\co \q-\q')=0$ since $\cs \co \q$ is a $(k+1)$-cube on $\nss$ with image $\q$ under $\pi$, so $\cs \co \q-\q'$ must be in $\cu^{k+1}(\cD_k(\ab))$. Thus the element of $H_k(\ns,\ab)$ corresponding to $\nss$ is the identity, as expected.\\

\noindent We shall now go in the opposite direction, to show that given any element of $H_k(\ns,\ab)$, there is a unique class of equivalent extensions of degree $k$ over $\ns$ with group $\ab$ corresponding to this element. In \cite{CamSzeg},  Camarena and Szegedy leave this task for the topological part of the paper, to avoid repetition, and in that part they combine the purely algebraic construction with topological tools to obtain, given a measurable cocycle, an extension which is also a \emph{compact} nilspace (see \cite[Section 3.5]{CamSzeg}). In the following subsection we shall isolate the purely algebraic part of this construction.
\medskip
\subsubsection{Obtaining an extension from a cocycle}
\medskip

Let $\ns$ be a nilspace and let $\ab$ be an abelian group. Our aim here is to show that any $\ab$-valued cocycle $\rho$ on $\cu^k(\ns)$ yields a $\ab$-bundle over $\ns$ which admits a compatible nilspace structure.\\
\indent For each $x\in \ns$ we denote by $\cu^k_x(\ns)$ the set of $k$-cubes $\q$ satisfying $\pi(\q):=\q(0^k)=x$. We consider the restrictions of the cocycle $\rho$ to each such set $\cu^k_x(\ns)$, denoting such a restriction by $\rho_x$. The extension that we shall construct consists essentially of the functions $\rho_x$, but we need to define an action of $\ab$ on these functions  such that there is a natural projection from this space to $\ns$ that respects the orbits of the action. This leads to the following natural definition of the extension.

\begin{defn}
Let $\rho:\cu^k(\ns)\to \ab$ be a cocycle of degree $k-1$. We then define the set $M=M(\rho)$ as follows.
\begin{equation}\label{eq:cocyclext}
M=\bigcup_{x\in \ns} \{ \rho_x+z\,:\,z\in\ab \}.
\end{equation}
We define the map $\tilde\pi:M\to \ns$ by $\rho_x+z\mapsto x$, for every $x\in \ns$ and $z\in \ab$. We also define an action of $\ab$ on $M$ by $(\rho_x+z,z')\mapsto \rho_x+z+z'$.
\end{defn}
\noindent One checks easily that $M$ is a $\ab$-bundle over $\ns$ with projection $\tilde\pi$. 

Now we shall define cubes on $M$. Given any function $f:\{0,1\}^k\to M$, we  denote by $a=a_f$ the function $\{0,1\}^k\to \ab$, $v\mapsto \rho_x(v)-f(v)$, where $x=\tilde\pi(f(v))\in \ns$.
\begin{defn}\label{def:extcubes}
We define $\cu^k(M)$ to be the set of functions $f:\{0,1\}^k\to M$ such that $\tilde\pi\co f\in \cu^k(\ns)$ and
\begin{equation}\label{eq:extcubes}
\rho(\tilde\pi\co f) = \sigma_k(a).
\end{equation}
For $n\neq k$, a function $f:\{0,1\}^n\to M$ is declared to be in $\cu^n(M)$ if $\tilde\pi\co f\in \cu^n(\ns)$ and every $k$-dimensional face-restriction of $f$ is in $\cu^k(M)$. \footnote{If $n<k$ then there are no $k$-dimensional face restrictions and so the condition is just that $\tilde\pi\co f\in \cu^n(\ns)$.}
\end{defn}
We can now complete the main task of this subsection.

\begin{proposition}\label{prop:extcompleted}
The set $M$ together with the cube sets $\cu^n(M)$, $n\geq 0$, is a nilspace.
\end{proposition}

\begin{proof}
To see the composition axiom, suppose that $f:\{0,1\}^n\to M$ is a cube and that $\phi:\{0,1\}^m\to \{0,1\}^n$ is a morphism. If $m<k$ then $f\co\phi\in \cu^m(M)$ just by Definition \ref{def:extcubes} and the composition axiom for $\ns$.  Supposing then that $m\geq k$, it suffices to check that for every face map $\phi':\{0,1\}^k\to \{0,1\}^m$ we have that $f\co\phi\co\phi'\in \cu^k(M)$. Now $\phi\co\phi'$ is a morphism $\{0,1\}^k\to \{0,1\}^n$. If $\phi\co\phi'$ is not injective then we are done because $\phi\co\phi'$ factors through a lower-dimensional morphism $\psi:\{0,1\}^\ell \to \{0,1\}^n$, $\ell<k$ and we have that $f\co\psi$ is in $\cu^\ell(M)$ as in the case $m<k$ above. If $\phi\co\phi'$ is injective, then using Lemma \ref{lem:inj-morph-decomp} repeatedly we see that it is a concatenation of face maps $\phi_j:\{0,1\}^k\to\{0,1\}^n$ (modulo automorphisms of $\{0,1\}^k$). For each of these face maps, we have that $f\co \phi_j$ satisfies \eqref{eq:extcubes} by assumption. It follows that $f\co\phi\co\phi'$ also satisfies \eqref{eq:extcubes}, since both functions involved in this equation are additive on adjacent cubes (for the cocycle this holds by definition, and for $\sigma_k$ it follows from \eqref{eq:concat-sgn}).

To see the completion axiom, note that, since the axiom holds for $\ns$, by the argument establishing completion in the proof of Lemma \ref{lem:k-degbundisns}, it suffices to check that for every $n\in \N$ we have $\cu^n(\ns) = \{\tilde\pi\co \q: \q\in  \cu^n(M)\}$ and for every $\q\in \cu^n(M)$ we have
\begin{equation}\label{eq:M-k-deg}
\{\q_2 \in \cu^n(M): \tilde\pi\co \q = \tilde\pi\co \q_2\} = \{\q+\q_3: \q_3\in \cu^n(\cD_{k-1}(\ab))\}.
\end{equation}
For $n< k$ this equation follows from the definition of $\cu^n(M)$ and the fact that $\cD_{k-1}(\ab)$ is $(k-1)$-fold ergodic. To see the case $n\geq k$, recall that $\q_3\in \cu^n(\cD_{k-1}(\ab))$ if and only if $\sigma_k(\q_3\co\phi)=0$ for every $k$-face map $\phi$ into $\{0,1\}^n$.
\end{proof}
\noindent Note that if $\ns$ is of step $s$, then $M(\rho)$ is of step $\max(s+1,k-1)$, as can be seen by checking the uniqueness of completion in the proof of Lemma \ref{lem:k-degbundisns}. The typical examples of extensions involve extending a $(k-1)$-step nilspace by a cocycle of degree $k$ to obtain a $k$-step nilspace (as in Example \ref{ex:Heisenext}), but it can be useful to consider a degree-$k$ extension of an $s$-step nilspace with $k<s$ (see for instance Lemma \ref{lem:transext}).\medskip

\noindent Recall from Lemma \ref{lem:crossecgendef} that given a degree-$k$ extension of a nilspace, every cross section $\cs$ for this extension generates a cocycle $\rho_{\cs}$. As our last main goal in this subsection, we want to show that, up to isomorphisms of nilspaces, every extension of $\ns$ by $\ab$ is of the form  $M(\rho)$ above for some cocycle $\rho$. The formal statement will be given as Corollary \ref{cor:classrep}. The main idea is that we can just take $\rho$ to be the cocycle generated by some cross section for the extension.\\
\indent First let us confirm a fact that ought to be true, namely that for $M(\rho)$ itself, the obvious cross-section does indeed generate $\rho$.
\begin{lemma}\label{lem:obviouscrossec}
Let $\ns$ be a nilspace, let $\ab$ be an abelian group, and let $\rho:\cu^k(\ns)\to \ab$ be a cocycle. Let $\cs:\ns \to M(\rho)$ be the cross section $x\mapsto \rho_x$. Then $\rho_{\cs}=\rho$.
\end{lemma}
\noindent (In particular, every cocycle can be viewed as a cocycle generated by a cross section of some extension.)
\begin{proof}
Let $\q\in \cu^k(\ns)$ and let $\q'$ be a cube in $\cu^k(M)$ with $\tilde\pi\co\q'=\q$. We thus have $\q'(v)=\rho_{\tilde\pi(\q'(v))}-a(v)=\rho_{\q(v)}-a(v)$ for a $\ab$-valued function $a$ satisfying $\sigma_k(a)=\rho(\q)$. Then by definition of $\rho_{\cs}$, we have
\[
\rho_{\cs}(\q)=\sigma_k \big(\cs\co\tilde\pi(\q')-\q'\big)=\sum_{v\in \{0,1\}^k} (-1)^{|v|}\big(\cs\co \q(v)-(\rho_{\q(v)}-a(v))\big)=\sum_{v} (-1)^{|v|} a(v) =\rho(\q). \qedhere
\]
\end{proof}

\begin{lemma}\label{lem:extisotoM} Let $\ns$ be a nilspace. Let $\nss$ be a degree-$k$ extension of $\ns$ by an abelian group $\ab$. Let $\cs:\ns\to \nss$ be a cross-section and let $\rho=\rho_{\cs}$ be the associated cocycle. Then $\nss$ is isomorphic as a nilspace to the extension $M=M(\rho)$ in \eqref{eq:cocyclext}.
\end{lemma}

\begin{proof} Let $\pi:\nss\to \ns$ be the projection for the extension. The isomorphism is given by the following map:
\begin{equation}
\theta:\nss\to M,\;\; x\mapsto \rho_{\pi(x)}+(x-\cs\co\pi(x)).
\end{equation}
It is easily checked that $\theta$ is a bijection. We claim that for every $\q\in \cu^n(\nss)$, the map $\theta\co \q$ satisfies the conditions in Definition \ref{def:extcubes}. We check this for $n=k$. The first condition is that $\tilde\pi\co\theta\co \q\in \cu^k(\ns)$. But $\tilde\pi\co\theta\co \q=\pi\co\q$, so the condition holds. Next, consider the function $a(v)=\rho_{\tilde\pi\co\theta\co \q}(v)-\theta\co\q(v)=\cs\co\pi(\q(v))-\q(v)$. The second condition is that $a$ satisfies \eqref{eq:extcubes}, that is $\rho(\tilde\pi\co\theta\co \q)=\sigma_k(a)$. But this holds indeed, since  $\rho(\tilde\pi\co\theta\co \q)=\rho(\pi\co \q)=\sigma_k(a)$, by definition of $\rho_{\cs}$ (recall Lemma \ref{lem:crossecgendef}). The cases $n\neq k$ follow similarly.
\end{proof}

\noindent Recall from Subsection \ref{sec:exts&cohom} that $H_k(\ns,\ab)$ denotes the quotient of the abelian group of degree-$k$ cocycles by the subgroup of coboundaries. Combining the results from this section, we now deduce the following correspondence between $H_k(\ns,\ab)$ and the degree-$k$ extensions of $\ns$ up to isomorphisms of nilspaces.
\begin{corollary}\label{cor:classrep} Let $\Phi$ denote the map which sends each class $C\in H_k(\ns,\ab)$ to the isomorphism class of $M(\rho)$, for any choice of $\rho\in C$. Then $\Phi$ is a surjection from $H_k(\ns,\ab)$ to the set of isomorphism classes of degree-$k$ extensions of $\ns$ by $\ab$.
\end{corollary}

\begin{proof}
We first check that $\Phi(C)$ is well-defined. If $\rho'$ is another element of the class $C$ of $\rho$ then, as seen in Lemma \ref{lem:cocyvar}, there is another cross section $\cs'$ of $\nss$ such that $\rho'=\rho_{\cs'}$. Then, by Lemma \ref{lem:extisotoM} we have $M(\rho)\cong \nss \cong M(\rho')$. To see that $\Phi$ is surjective, note that given a degree-$k$ extension $\nss$ of $\ns$, there is a cross section $\cs:\ns\to \nss$ and corresponding cocycle $\rho=\rho_{\cs}$, and so the isomorphism class of $\nss$ is $\Phi(C)$ where $C$ is the class of $\rho$.
\end{proof}
\begin{remark} Note that $\Phi$ need not be injective. This is a counterpart for nilspaces of known facts concerning group extensions, for instance that the abelian group of classes of extensions $1\to A\to E\to G \to 1$ of a group $G$ by an abelian group $A$ can be larger than the set of extensions $E$ themselves up to group isomorphism.
\end{remark}
\noindent We end this subsection with an alternative formulation of condition \eqref{eq:extcubes} defining cubes on the extension $M(\rho)$. This formulation uses tricubes and is very useful in particular in the topological part of the theory.\\
\indent Recall from Definition \ref{def:tricube} and Lemma \ref{lem:tricube-comp} the concept of a morphism from a tricube $T_k$ into a nilspace $\ns$, and  the tricube composition $t\co \omega_k$ (where $\omega_k$ is the outer-point map from Definition \ref{def:ope}). Let $\ab$ be an abelian group, and let $\xi:\cu^k(\ns)\to \ab$ be an arbitrary function. We then define the alternating sum
\begin{equation}\label{eq:altcubesum2}
\beta(t,\xi)=\sum_{v\in\{0,1\}^k}(-1)^{|v|}\,\xi(t \co \trem_v).
\end{equation}
For a cocycle $\rho$, the following result reduces $\beta(t,\rho)$ to the value of $\rho$ on the outer-point set of $T_k$.
\begin{lemma}\label{lem:tricubesum} Let $t:T_k\to \ns$ be a morphism into a nilspace $\ns$, and let $\rho$ be a cocycle of degree $k-1$ on $\ns$. Then 
\[
\beta(t,\rho)=\rho(t \co \omega_k).
\]
\end{lemma}
\begin{proof}
This follows from the fact that the outer-point map $\omega_k$ can be expressed  as a sequence of concatenations of cubes of the form $\trem_v$ composed with automorphisms of $\{0,1\}^k$.
\end{proof}
\noindent Given a cube $\q\in \cu^k(\ns)$, let us denote by $\hom_{\q\co\omega_k^{-1}}(T_k,\ns)$ the set of morphisms $t:T_k\to \ns$ that agree with $\q$ on the subset $\{-1,1\}^k$. We can now give the alternative form of \eqref{eq:extcubes}. 
\begin{lemma}\label{lem:extcubes2}
Condition \eqref{eq:extcubes} is equivalent to the following equation holding for some \textup{(}and therefore every\textup{)} $t\in \hom_{(\tilde\pi\co f)\co\omega_k^{-1}}(T_k,\ns)\colon$
\begin{equation}\label{eq:extcubes2}
\sum_{v\in\{0,1\}^k}(-1)^{|v|}\,f(v)(t\co\trem_v)=0.
\end{equation}
\end{lemma}
\begin{proof}
Given such a morphism $t$, note that for each $v\in \{0,1\}^k$, the map $t\co\trem_v$ is a cube in $\cu^k(\ns)$ with base-point $t\co\psi_v(0^k)=t\co\omega_k(v)=\tilde\pi\co f(v)$. Therefore, we have trivially that $\rho_{\tilde\pi\co f(v)}(t\co\trem_v)=\rho (t\co\trem_v)$. By Lemma \ref{lem:tricubesum}, we then have
\begin{eqnarray*}
\sum_{v\in\{0,1\}^k}(-1)^{|v|} \,f(v)(t\co\trem_v)& = & \sum_{v\in\{0,1\}^k} (-1)^{|v|}\, \rho_{\tilde\pi\co f(v)}(t\co\trem_v)\;-\;(-1)^{|v|}\, a(v)\\
& = & \beta(t,\rho)- \sum_{v\in\{0,1\}^k}(-1)^{|v|}\, a(v)\; = \; \rho(\tilde\pi\co f)- \sigma_k(a),
\end{eqnarray*}
and the result follows.
\end{proof}
\medskip
\subsection{Translation bundles}\label{subsec:transbundles}
\medskip
Given a nilspace $\ns$, recall that $\pi_{k-1}$ denotes the projection to the factor $\cF_{k-1}(\ns)$.

\begin{defn}[Translation lift]
Let $\ns$ be a $k$-step nilspace and let $\alpha\in\tran_i(\cF_{k-1}(\ns))$. We say that $\alpha'\in \tran_i(\ns)$ is a \emph{lift of $\alpha$ to} $\tran_i(\ns)$ if for every $x\in \ns$ we have $\pi_{k-1}(\alpha'(x))=\alpha(\pi_{k-1}(x))$.
\end{defn}
\noindent The main result below, Proposition \ref{prop:transliftcrit}, gives a useful criterion for whether a translation on $\cF_{k-1}(\ns)$ can be lifted to a translation on $\ns$. The result involves the following construction.

\begin{defn}\label{def:transpairs}
Given a translation $\alpha\in \tran_i(\cF_{k-1}(\ns))$, we define
\begin{equation}\label{eq:transnilspace}
\cT=\cT(\alpha,\ns,i):= \{ (x_0,x_1)\in \ns^2: \alpha(\pi_{k-1}(x_0))=\pi_{k-1}(x_1)\}.
\end{equation}
\end{defn}

\begin{lemma}\label{lem:transbundlespace}
Let $\ns$ be a $k$-step nilspace, let $\alpha\in\tran_i(\cF_{k-1}(\ns))$, and let $i< k$. Then the set $\cT$, together with the restriction to $\cT$  of cubes on $\ns\Join_i \ns$, is a $k$-step nilspace.
\end{lemma}
\begin{proof}
Recall that by Definition \ref{def:arrowspace} we have $\q=\q_0\times \q_1$ in $\cu^n(\ns\Join_i \ns)$ if the $i$-th arrow $\arr{\q_0,\q_1}_i$ is in $\cu^{n+i}(\ns)$. The composition axiom for $\cT$ follows immediately from that for $\ns\Join_i \ns$.\\
\indent Next we claim that if $i\leq k-1$ then $\cT$ satisfies the ergodicity axiom. To see this let $f\in \cT^{\{0,1\}}$ be arbitrary, with $f(0)=(x_0,x_1)$ and $f(1)=(y_0,y_1)$ where $\alpha(\pi_{k-1}(x_0))=\pi_{k-1}(x_1)$ and $\alpha(\pi_{k-1}(y_0))=\pi_{k-1}(y_1)$. We have $f=\q_0\times \q_1$ where $\q_i(0)=x_i$, $\q_i(1)=y_i$ for $i=0,1$, and where $\q_0,\q_1$ are in $\cu^1(\ns)$ by ergodicity in $\ns$. We are claiming that $f\in \cu^1(\cT)$, i.e. that $\arr{\q_0,\q_1}_i\in \cu^{1+i}(\ns)$. By the assumption on $\alpha$ and Lemma \ref{lem:transcriterion},  we have that $\arr{\pi_{k-1}\co \q_0,\alpha\co \pi_{k-1}\co \q_0}_i\in \cu^{1+i}(\cF_{k-1}(\ns))$. But the latter cube is $\arr{\pi_{k-1}\co \q_0, \pi_{k-1}\co \q_1}_i=\pi_{k-1}\co \arr{\q_0,\q_1}_i$. Recalling Remark \ref{rem:cubesuff}, we see that if $i\leq k-1$ then $\pi_{k-1}\co \arr{\q_0,\q_1}_i \in \cu^{i+1}(\cF_{k-1}(\ns))$ implies $\arr{\q_0,\q_1}_i \in \cu^{i+1}(\ns)$. \\
\indent To check the completion axiom, let $\q'$ be an $n$-corner on $\cT$. Thus, as an $n$-corner on $\ns\Join_i\ns$, the map $\q'$ equals $\q_0'\times \q_1'$ where $\q_0',\q_1'$ are $n$-corners on $\ns$. We know from Lemma \ref{lem:arrow} that $\q'$ can be completed to an $n$-cube on $\ns\Join_i\ns$, but here we need to complete it on $\cT$, which can be done as follows. Let $\q_0\in \cu^n(\ns)$ be a completion of $\q_0'$. Then by Lemma  \ref{lem:transcriterion} we have $(\pi_{k-1}\co \q_0)\times (\alpha\co\pi_{k-1}\co \q_0)\in \cu^n(\cF_{k-1}(\ns)\Join_i\cF_{k-1}(\ns))$, that is $\arr{\pi_{k-1}\co \q_0,\alpha\co\pi_{k-1}\co \q_0}_i\in \cu^{n+i}(\cF_{k-1}(\ns))$. Let us lift the latter cube to $\tilde\q\in \cu^{n+i}(\ns)$, that is $\pi_{k-1}\co \tilde\q=\arr{\pi_{k-1}\co \q_0,\alpha\co\pi_{k-1}\co \q_0}_i$. Now let $\q''$ be the $(n+i)$-corner on $\ns$ defined by $\q''(v,w)=\q_0(v)$ if $w\neq 1^i$ and $\q''(v,1^i)=\q_1'(v)$ for $v\neq 1^n$. (This was shown to be a corner in the proof of Lemma \ref{lem:arrow}.) Note that, since $\pi_{k-1}\co\q_1'=\alpha\co \pi_{k-1}\co\q_0'$, we have that $\q''(v,w)$ lies above $\arr{\pi_{k-1}\co \q_0,\alpha\co\pi_{k-1}\co \q_0}_i(v,w)$ for every $(v,w)\neq 1^{n+i}$, and so $\tilde \q(v,w) \sim_{k-1} \q''(v,w)$ for such $(v,w)$. Now, by modifying the values of the cube $\tilde \q$ at 1-faces along an appropriate Hamiltonian path on $\{0,1\}^{n+i}$ (as in previous arguments, e.g. the proof of Lemma \ref{lem:Z-bund-is-k-degbund}), we obtain a cube $\overline{\q}\in \cu^{n+i}(\ns)$ which agrees with $\q''$ at every $(v,w)\neq 1^{n+i}$, and such that $\overline{\q}$ is still a lift of $\arr{\pi_{k-1}\co \q_0,\alpha\co\pi_{k-1}\co \q_0}_i$. It follows that $\overline{\q}=\arr{\q_0,\q_1}_i$ for some completion $\q_1$ of $\q_1'$ such that $\pi_{k-1}\co \q_1(1^n)=\alpha\co\pi_{k-1}\co \q_0(1^n)$, and so $\q_0\times\q_1$ completes $\q'$ on $\cT$.
\end{proof}

The nilspace that captures whether $\alpha$ can be lifted is the factor
\begin{equation}
\cT^*:=\cF_{k-1}(\cT(\alpha,\ns,i)).
\end{equation}

\noindent In order to establish the main result, we need to relate $\cT^*$ to $\ns$. To do so, the key fact is that each element of $\cT^*$ corresponds in a natural way to some local translation on $\ns$ (recall Definition \ref{def:loc-trans}). To state this fact formally, we shall be careful to distinguish the projection $\cT\to\cT^*$, which we denote by $\pi_{k-1,\cT}$, from the projection $\ns\to \cF_{k-1}(\ns)$, denoted by $\pi_{k-1,\ns}$.

\begin{proposition}\label{prop:Tstar-char}
Let $\tau\in \cT^*$, and fix any $(x_0,x_1)\in \cT$ such that $\tau=\pi_{k-1,\cT}((x_0,x_1))$. Then $\tau$ is the following equivalence class of points in $\cT$:
\begin{equation}\label{eq:Tstar-char}
\tau=\{(y_0,y_1)\in \ns\times\ns: \pi_{k-1,\ns}(y_0)=\pi_{k-1,\ns}(x_0),\; \pi_{k-1,\ns}(y_1)=\pi_{k-1,\ns}(x_1),\;y_1=\phi_{x_0,x_1}(y_0)\},
\end{equation}
for the local translation $\phi_{x_0,x_1}$.
\end{proposition}
\begin{proof}
Given an element $\tau=\pi_{k-1,\cT}((x_0,x_1))\in \cT^*$, first we have that $\alpha\co \pi_{k-1,\ns}(x_0)=\pi_{k-1}(x_1)$ (since $(x_0,x_1)\in \cT$), and so the fibres $F_0=\{y_0\in \ns: y_0\sim_{k-1} x_0\}$ and $F_1=\{y_1\in \ns: y_1\sim_{k-1} x_1\}$ satisfy $\alpha(F_0)=F_1$. The equivalence class $\tau$ is the set of couples $(y_0,y_1) \in \cT$ such that $(x_0,x_1)\sim_{k-1} (y_0,y_1)$ in $\cT$, i.e. such that there exists a cube $\q\in \cu^k(\cT)$ satisfying $\q(v)=(x_0,x_1)$ for $v\neq 1^k$ and $\q(1^k)=(y_0,y_1)$. But  then, since by definition $\q\in \cu^k(\ns \Join_i \ns)$, and $\cu^k(\ns \Join_i \ns)\subset \cu^k(\ns \Join \ns)$, we deduce that $(x_0,x_1)\sim_{k-1} (y_0,y_1)$ in $\ns \Join \ns$. This, by Lemma \ref{lem:fibreiso}, tells us that $y_0=x_0+a$ and $y_1=x_1+b$ where $b=a$ (under the identification of $\ab_{F_0}$ and $\ab_{F_1}$ with $\ab_k$ provided by Lemma \ref{lem:fibreiso}), so we have indeed $y_0\in F_0,y_1\in F_1$ and $y_1=\phi_{x_0,x_1}(y_0)$ where $\phi_{x_0,x_1}$ is the local translation corresponding to $x_0,x_1$.
\end{proof}
\noindent With this characterization of $\cT^*$, we can see that $\ab_k$ has an action on $\cT^*$ via the first coordinate, namely for $a\in \ab_k$ we set
\begin{equation}\label{eq:trans-bund-action}
a+\tau = \pi_{k-1,\cT}((a+x_0 ,x_1)).
\end{equation}
The fact that this action is well-defined follows from the fact that $\ab_k\times \ab_k$ has an action on $\cT$ given by $(a,b)+(x_0,x_1)=(x_0+a,x_1+b)$, which is clear. We now want to use this $\ab_k$-action to view $\cT^*$ as an extension of $\cF_{k-1}(\ns)$. To that end we shall apply the following result concerning $i$-arrows.

\begin{lemma}\label{lem:k-deg-abarrow-crit}
Let $\ab$ be an abelian group, let $\q_0,\q_1:\{0,1\}^n\to \ab$, and let $i,k\in \N$. Then $\arr{\q_0,\q_1}_i$ lies in $\cu^{n+i}(\cD_k(\ab))$ if and only if $\q_0\in \cu^n(\cD_k(\ab))$ and $\q_1-\q_0\in \cu^n(\cD_{k-i}(\ab))$.
\end{lemma}
\begin{proof}
This is obtained as a special case of Lemma \ref{lem:filtarrow}, recalling Definition \ref{def:k-deg-ab-cubes}. 
\end{proof}

\begin{lemma}\label{lem:transext}
Let $\gamma:\cT^*\to \cF_{k-1}(\ns),\; \pi_{k-1,\cT}((x_0,x_1))\mapsto \pi_{k-1,\ns}(x_0)$. Then $\cT^*$ is a degree-$(k-i)$ extension of $\cF_{k-1}(\ns)$ with bundle map $\gamma$ and structure group $\ab_k$.
\end{lemma}

\begin{proof}
That $\gamma$ is well-defined follows from the fact that if $(y_0,y_1)$ is any other couple in $\tau=\pi_{k-1,\cT}((x_0,x_1))$ then $x_0\sim_{k-1}y_0$, by \eqref{eq:Tstar-char}.

Next, we check that $\gamma$ is a morphism with the surjectivity property (i) in Definition \ref{def:extension}. Every $n$-cube $\q$ on $\cT^*$ is $\pi_{k-1,\cT}\co \q'$ for some $n$-cube $\q'$ on $\cT$, and the latter has the form $\q_0\times \q_1$ for two $n$-cubes on $\ns$. It follows that $\gamma\co \q=\pi_{k-1,\ns}\co \q_0$ is an $n$-cube on $\cF_{k-1}(\ns)$, whence $\gamma$ is a morphism. To see the surjectivity, let $\pi_{k-1,\ns}\co \q_0$ be an $n$-cube on $\cF_{k-1}(\ns)$, with $\q_0\in \cu^n(\ns)$. Then $\q_1'=\alpha \co \pi_{k-1,\ns}\co \q_0$ is also an $n$-cube on $\cF_{k-1}(\ns)$, satisfying $(\pi_{k-1,\ns}\co \q_0) \times \q_1' \in \cu^n (\cF_{k-1}(\ns) \Join_i \cF_{k-1}(\ns))$, since $\alpha \in \tran_i(\cF_{k-1}(\ns))$. Thus $\q_1'=\pi_{k-1,\ns}\co \q_1$ for some $\q_1\in \cu^n(\ns)$, and $\arr{\pi_{k-1,\ns}\co \q_0,\pi_{k-1,\ns}\co \q_1}_i= \pi_{k-1,\ns}\co \arr{\q_0,\q_1}_i\in \cu^{n+i}(\cF_{k-1}(\ns))$. There is therefore some $\tilde \q\in \cu^{n+i}(\ns)$ such that $\pi_{k-1,\ns}\co \q'=\pi_{k-1,\ns}\co \arr{\q_0,\q_1}_i$. Modifying $\tilde\q$ as we did in the proof of Lemma \ref{lem:transbundlespace}, we obtain a new cube $\overline{\q}$ still satisfying $\pi_{k-1,\ns}\co \overline{\q}=\pi_{k-1,\ns}\co \arr{\q_0,\q_1}_i$ but now with $\overline{\q}=\arr{\q_0,\q_1''}$ for some $n$-cube $\q_1''$ with $\q_1''(v)\sim_{k-1} \q_1(v)$ for all $v$. Thus $\q:=\q_0\times \q_1''$ is an $n$-cube on $\cT$, and $\pi_{k-1,\cT}\co \q$ is an $n$-cube on $\cT^*$ with $\gamma$-image equal to $\pi_{k-1,\ns}\co \q_0$, as required.\\
\indent Finally, we check condition (ii) from Definition \ref{def:extension}.  Let $\pi_{k-1,\cT}\co \q$ be an arbitrary cubes in $\cu^n(\cT^*)$, thus $\q=\q_0\times \q_1$ is in $\cu^n(\cT)$.\\
\indent We first show that the left side of condition (ii) is included in the right side. Assume that $\pi_{k-1,\cT}\co \q'$ is another cube with $\q'=\q_0'\times \q_1'\in \cu^n(\cT)$ such that $\gamma \co \pi_{k-1,\cT}\co \q' = \gamma\co \pi_{k-1,\cT}\co  \q$. This means that $\pi_{k-1,\ns}\co \q_0' = \pi_{k-1,\ns}\co \q_0$, and since $\pi_{k-1,\ns}\co \q_1=\alpha\co \pi_{k-1,\ns}\co \q_0$ and similarly for $\q_1'$, we have in fact $\pi_{k-1,\ns}\co \q'=\pi_{k-1,\ns}\co \q$. We therefore have $\ab_k$-valued functions $\q^*_0:=\q_0-\q_0'$ and $\q^*_1:=\q_1-\q_1'$, and these are cubes in $\cu^n(\cD_k(\ab_k))$ since $\ns$ is a degree-$k$ extension of $\cF_{k-1}(\ns)$ with structure group $\ab_k$. Moreover, we have $\arr{\q_0,\q_1}_i,\arr{\q_0',\q_1'}_i\in \cu^{n+i}(\ns)$. and these cubes are also equal modulo $\pi_{k-1,\ns}$, so we can also take their $\ab_k$-valued difference, which is $\arr{\q_0^*,\q_1^*}_i$. The latter cube must again be in $\cu^{n+i}(\cD_k(\ab_k))$, and by Lemma \ref{lem:k-deg-abarrow-crit} this holds if and only if $\q_1^*-\q_0^* \in \cu^n(\cD_{k-i}(\ab_k))$. Now, our assumption above implies that we can take the $\ab_k$-valued difference $\gamma\co \pi_{k-1,\cT}\co \q'\,- \,\gamma\co \pi_{k-1,\cT}\co  \q$, relative to the $\ab_k$-action defined in \eqref{eq:trans-bund-action}.  This difference is a function $a:\{0,1\}^n\to \ab_k$, and we claim that $a(v)= \q_1^*(v)-\q_0^*(v)$. To see this, observe using \eqref{eq:trans-bund-action} that for each $v\in \{0,1\}^n$ the difference $a(v)$ is the element of $\ab_k$ which has to be added to $\q_0(v)$ in order to have $(\q_0(v)+a(v))-\q_0'(v)= \q_1(v)-\q_1'(v)$. (Indeed, this is how to ensure that $a(v)+\pi_{k-1,\cT}\co \q(v)=\pi_{k-1,\cT}\co \q'(v)$.) Our claim follows, and then as shown above we have $a=\q_1^*-\q_0^*\in \cu^n(\cD_{k-i}(\ab_k))$.\\
\indent The opposite inclusion in condition (ii) is seen similarly. Indeed, if we add $a\in \cu^n(\cD_{k-i}(\ab_k))$ to $\pi_{k-1,\cT}\co\q$ we obtain the map $\pi_{k-1,\cT}\co ((\q_0+a)\times \q_1))$, which is a cube on $\cT^*$ if $(\q_0+a)\times \q_1 \in \cu^n(\cT)$, which holds if $\arr{\q_0+a,\q_1}_i \in \cu^{n+i}(\ns)$. But this does hold, because $\arr{\q_0+a,\q_1}_i=\arr{\q_0,\q_1}_i+\arr{a,0}_i$ and $\arr{a,0}_i\in \cu^{n+i}(\cD_k(\ab_k))$ by Lemma \ref{lem:k-deg-abarrow-crit}.
\end{proof}

We can now establish the main result.
\begin{proposition}\label{prop:transliftcrit}
Let $\ns$ be a $k$-step nilspace and let $\alpha\in \tran_i(\cF_{k-1}(\ns))$. If $\cT^*=\cT^*(\alpha,\ns,i)$ is a split extension then $\alpha$ has a lift in $\tran_i(\ns)$.
\end{proposition}

\begin{proof}
Suppose that there exists a morphism $m:\cF_{k-1}(\ns)\to \cT^*$ such that $\gamma \co m$ is the identity map. We define a map $\beta:\ns\to\ns$ as follows. For each $x\in \ns$, the element $m(\pi_{k-1}(x))\in \cT^*$ represents a local translation $\phi$, from the class $F_0$ of $\sim_{k-1,\ns}$ containing $x$, to the class $\alpha(F_0)$. Let $\beta(x)=\phi(x)$. We claim that $\beta\in \tran_i(\ns)$. To prove this, by Lemma \ref{lem:transcriterion} it suffices to show that for every $\q\in \cu^k(\ns)$ we have that $\q\times (\beta\co\q)\in \cu^k(\ns \Join_i \ns)$. Now $\q\times (\beta\co\q)$ is a lift to $\cT$ of the $\cT^*$-valued map $m\co \pi_{k-1}\co \q$. The latter map is in $\cu^k (\cT^*)$ because $m$ is a morphism and $\pi_{k-1}\co \q \in \cu^k (\cF_{k-1}(\ns))$. But then, by  Lemma \ref{lem:lifting}, the lift $\q\times (\beta\co\q)$ is in $\cu^k(\cT)\subset  \cu^k(\ns \Join_i \ns)$, so we are done.
\end{proof}



\begin{dajauthors}
\begin{authorinfo}[pgom]
  Pablo Candela\\
  Universidad Aut\'onoma de Madrid\\
  Madrid, Spain\\
  pablo\imagedot{}candela\imageat{}uam\imagedot{}es \\
  \url{http://verso.mat.uam.es/~pablo.candela/index.html}
\end{authorinfo}

\end{dajauthors}

\end{document}